\documentclass[leqno,11pt]{article}
\usepackage[utf8]{inputenc}
\usepackage[T1]{fontenc}
\usepackage{microtype}

\usepackage[letterpaper]{geometry}

\ifdefined\screenview
  \edef\mtht{\the\textheight}
  \edef\mtwd{\the\textwidth}
  \usepackage{xcolor}
  \definecolor{BackgroundColor}{RGB}{253, 246, 227}
  \pagecolor{BackgroundColor}
\fi

\usepackage{amsmath}
\usepackage{amsthm}
\usepackage{tikz-cd}
\usetikzlibrary{arrows} 
\tikzset{
  commutative diagrams/.cd, 
  arrow style=tikz, 
  diagrams={>=stealth}
}
\usetikzlibrary{matrix,decorations.pathreplacing,calc}

\usepackage{textcomp}

\usepackage[sb]{libertine}
\usepackage[varqu,varl]{zi4}
\usepackage[libertine,bigdelims,vvarbb]{newtxmath} 
\usepackage[supstfm=libertinesups,%
  supscaled=1.2,%
  raised=-.13em]{superiors}
\useosf 
\usepackage[scr=boondox,cal=euler]{mathalfa}

\usepackage{slashed}
\usepackage{esint} 
\usepackage[english]{babel}

\usepackage{imakeidx}
\makeindex[intoc]

\usepackage{csquotes}
\usepackage[
  backend=biber,
  hyperref=true,
  backref=true,
  isbn=false,
  doi=true,
  natbib=true,
  eprint=true,
  useprefix=true,
  maxcitenames=99,
  maxbibnames=99,  
  maxalphanames=99, 
  minalphanames=99,
  safeinputenc,
  style=alphabetic,
  citestyle=alphabetic,
  block=space,
  datamodel=preamble/ext-eprint
]{biblatex}
\usepackage[
  hypertexnames = false,
  colorlinks    = true,
  citecolor     = gray,
  linkcolor     = gray,
  urlcolor      = gray,
  breaklinks
]{hyperref}
\makeatletter
\DeclareFieldFormat{arxiv}{%
  arXiv\addcolon\space
  \ifhyperref
    {\href{http://arxiv.org/\abx@arxivpath/#1}{%
       \nolinkurl{#1}%
       \iffieldundef{arxivclass}
         {}
         {\addspace\texttt{\mkbibbrackets{\thefield{arxivclass}}}}}}
    {\nolinkurl{#1}
     \iffieldundef{arxivclass}
       {}
       {\addspace\texttt{\mkbibbrackets{\thefield{arxivclass}}}}}}
\makeatother
\DeclareFieldFormat{mr}{%
  MR\addcolon\space
  \ifhyperref
    {\href{http://www.ams.org/mathscinet-getitem?mr=MR#1}{\nolinkurl{#1}}}
    {\nolinkurl{#1}}}
\DeclareFieldFormat{zbl}{%
  Zbl\addcolon\space
  \ifhyperref
    {\href{http://zbmath.org/?q=an:#1}{\nolinkurl{#1}}}
    {\nolinkurl{#1}}}
\renewbibmacro*{eprint}{%
  \printfield{arxiv}%
  \newunit\newblock
  \printfield{mr}%
  \newunit\newblock
  \printfield{zbl}%
  \newunit\newblock
  \iffieldundef{eprinttype}
    {\printfield{eprint}}
    {\printfield[eprint:\strfield{eprinttype}]{eprint}}
}
\AtEveryBibitem{%
  \clearlist{publisher}%
}
\AtEveryBibitem{%
  \clearlist{address}%
}
\DeclareFieldFormat[article,inproceedings,inbook,incollection,thesis]{title}{\textit{#1}}
\renewbibmacro{in:}{}
\newcommand{\printreferences}{\printbibliography[heading=bibintoc]}

\usepackage[inline,shortlabels]{enumitem}

\usepackage{subcaption}

\usepackage[yyyymmdd]{datetime}

\usepackage{etoolbox}
\ifundef{\abstract}{}{\patchcmd{\abstract}%
    {\quotation}{\quotation\noindent\ignorespaces}{}{}}

\usepackage[super]{nth}

\usepackage{aliascnt}

\numberwithin{equation}{section}

\renewcommand{\eqref}[1]{\hyperref[#1]{\rm(\ref*{#1})}}

\def\makeautorefname#1#2{\AtBeginDocument{\expandafter\def\csname#1autorefname\endcsname{#2}}}

\newcommand{\mynewtheorem}[2]{
  \newaliascnt{#1}{equation}          
  \newtheorem{#1}[#1]{#2}
  \aliascntresetthe{#1}
  \makeautorefname{#1}{#2}
}

\mynewtheorem{axiom}{Axiom}
\newtheorem*{axiom*}{Axiom}
\mynewtheorem{theorem}{Theorem}
\newtheorem*{theorem*}{Theorem}
\mynewtheorem{prop}{Proposition}
\newtheorem*{prop*}{Proposition}

\mynewtheorem{cor}{Corollary}
\mynewtheorem{construction}{Construction}
\mynewtheorem{lemma}{Lemma}
\mynewtheorem{conjecture}{Conjecture}
\newtheorem*{conjecture*}{Conjecture}
\mynewtheorem{hyp}{Hypothesis}
\newtheorem{step}{Step}

\numberwithin{substep}{step}
\makeautorefname{step}{Step}
\makeautorefname{substep}{Step}

\numberwithin{subcase}{case}
\makeautorefname{case}{Case}
\makeautorefname{subcase}{case}

\usepackage{etoolbox}
\AtBeginEnvironment{proof}{\setcounter{step}{0}}

\theoremstyle{remark}
\mynewtheorem{remark}{Remark}
\newtheorem*{remark*}{Remark}
\mynewtheorem{convention}{Convention}
\newtheorem*{convention*}{Convention}
\newtheorem*{conventions*}{Conventions}

\theoremstyle{remark}
\mynewtheorem{warning}{Warning}

\theoremstyle{definition}
\mynewtheorem{definition}{Definition}
\newtheorem*{definition*}{Definition}
\mynewtheorem{notation}{Notation}
\mynewtheorem{data}{Data}
\mynewtheorem{example}{Example}
\newtheorem*{example*}{Example}
\mynewtheorem{exercise}{Exercise}
\mynewtheorem{solution}{Solution}
\mynewtheorem{question}{Question}
\mynewtheorem{problem}{Problem}
\newtheorem*{question*}{Question}

\mynewtheorem{summary}{Summary}

\makeautorefname{table}{Table}        
\makeautorefname{chapter}{Chapter}
\makeautorefname{section}{Section}
\makeautorefname{subsection}{Section}
\makeautorefname{subsubsection}{Section}
\makeautorefname{footnote}{Footnote}
\AtBeginDocument{\def\itemautorefname~#1\null{(#1)\null}}
\AtBeginDocument{\def\equationautorefname~#1\null{(#1)\null}}

\let\C\undefined
\let\U\undefined

\usepackage{bm}
\usepackage{mathtools} 
\usepackage{stmaryrd} 

\DeclareFontFamily{U}{mathx}{\hyphenchar\font45}
\DeclareFontShape{U}{mathx}{m}{n}{
      <5> <6> <7> <8> <9> <10>
      <10.95> <12> <14.4> <17.28> <20.74> <24.88>
      mathx10
      }{}
\DeclareSymbolFont{mathx}{U}{mathx}{m}{n}
\DeclareFontSubstitution{U}{mathx}{m}{n}
\DeclareMathAccent{\widecheck}{0}{mathx}{"71}
\DeclareMathAccent{\wideparen}{0}{mathx}{"75}

\DeclareMathOperator{\HF}{\HF}

\DeclareMathOperator{\Hom}{Hom}

\DeclareMathOperator{\OT}{OT}
\DeclareMathOperator{\PD}{PD}

\DeclareMathOperator{\coker}{coker}

\DeclareMathOperator{\im}{im}
\DeclareMathOperator{\ind}{index}

\DeclareMathOperator{\rank}{rank}

\DeclareMathOperator{\sign}{sign}

\DeclarePairedDelimiter\floor{\lfloor}{\rfloor}
\DeclarePairedDelimiter\parentheses{\lparen}{\rparen}
\DeclarePairedDelimiter\paren{\lparen}{\rparen}
\DeclarePairedDelimiter{\Abs}{\|}{\|}

\DeclarePairedDelimiter{\abs}{\lvert}{\rvert}
\DeclarePairedDelimiter{\braket}{\langle}{\rangle}

\DeclarePairedDelimiter{\set}{\lbrace}{\rbrace}
\def\({\left(}
\def\){\right)}
\def\<{\left\langle}
\def\>{\right\rangle}

\newcommand{\C}{{\mathbf{C}}}
\newcommand{\Gtwo}{G_2}

\newcommand{\Met}{\sM\!et}
\newcommand{\N}{{\mathbf{N}}}

\newcommand{\PSL}{\P\SL}

\newcommand{\R}{\mathbf{R}}
\newcommand{\SF}{\mathrm{SF}}
\newcommand{\SL}{\mathrm{SL}}

\newcommand{\SU}{\mathrm{SU}}

\newcommand{\Span}[1]{\braket{#1}}

\newcommand{\U}{\mathrm{U}}

\newcommand{\Z}{\mathbf{Z}}

\newcommand{\andq}{\text{and}\quad}

\newcommand{\ch}{\mathrm{ch}}

\newcommand{\co}{\mskip0.5mu\colon\thinspace}

\newcommand{\defined}[2][\key]{\def\key{#2}\textbf{#2}\index{#1}}

\newcommand{\del}{\partial}

\newcommand{\hkred}{{/\!\! /\!\! /}}

\newcommand{\id}{\mathrm{id}}

\newcommand{\inner}[2]{\braket{#1, #2}}
\newcommand{\into}{\hookrightarrow}
\newcommand{\iso}{\cong}
\newcommand{\itref}{\eqref}

\newcommand{\loc}{\mathrm{loc}}

\newcommand{\ob}{\mathrm{ob}}

\newcommand{\pr}{\mathrm{pr}}
\newcommand{\qandq}{\quad\text{and}\quad}

\newcommand{\qforq}{\quad\text{for}\quad}
\newcommand{\qand}{\quad\text{and}}

\newcommand{\qwithq}{\quad\text{with}\quad}
\newcommand{\reg}{\mathrm{reg}}

\newcommand{\su}{\mathfrak{su}}
\newcommand{\sw}{\fs\fw}

\renewcommand{\H}{\mathbf{H}}
\renewcommand{\Im}{\operatorname{Im}}

\renewcommand{\P}{\mathbf{P}}
\renewcommand{\Re}{\operatorname{Re}}
\renewcommand{\det}{\operatorname{det}}

\renewcommand{\emptyset}{\varnothing}
\renewcommand{\epsilon}{\varepsilon}
\renewcommand{\setminus}{{\backslash}}

\renewcommand{\leq}{\leqslant}
\renewcommand{\geq}{\geqslant}

\newcommand{\hsw}{\widehat\sw}

\makeatletter
\renewcommand*\env@matrix[1][*\c@MaxMatrixCols c]{%
  \hskip -\arraycolsep
  \let\@ifnextchar\new@ifnextchar
  \array{#1}}

\renewcommand\xleftrightarrow[2][]{%
  \ext@arrow 9999{\longleftrightarrowfill@}{#1}{#2}}
\newcommand\longleftrightarrowfill@{%
  \arrowfill@\leftarrow\relbar\rightarrow}
\makeatother

\newcommand{\rd}{{\rm d}}

\newcommand{\rII}{{\rm II}}

\newcommand{\bp}{{\mathbf{p}}}
\newcommand{\bq}{{\mathbf{q}}}

\newcommand{\bx}{{\mathbf{x}}}

\newcommand{\bE}{{\mathbf{E}}}

\newcommand{\bL}{{\mathbf{L}}}

\newcommand{\bS}{{\mathbf{S}}}

\newcommand{\bV}{{\mathbf{V}}}
\newcommand{\bW}{{\mathbf{W}}}

\newcommand{\sA}{\mathscr{A}}
\newcommand{\sB}{\mathscr{B}}
\newcommand{\sC}{\mathscr{C}}

\newcommand{\sG}{\mathscr{G}}
\newcommand{\sH}{\mathscr{H}}

\newcommand{\sM}{\mathscr{M}}
\newcommand{\sN}{\mathscr{N}}

\newcommand{\sP}{\mathscr{P}}
\newcommand{\sQ}{\mathscr{Q}}

\newcommand{\sS}{\mathscr{S}}

\newcommand{\sU}{\mathscr{U}}

\newcommand{\sW}{\mathscr{W}}

\newcommand{\sZ}{\mathscr{Z}}

\newcommand{\fa}{{\mathfrak a}}

\newcommand{\fc}{{\mathfrak c}}
\newcommand{\fd}{{\mathfrak d}}
\newcommand{\fe}{{\mathfrak e}}
\newcommand{\ff}{{\mathfrak f}}

\newcommand{\fl}{{\mathfrak l}}

\newcommand{\fs}{{\mathfrak s}}

\newcommand{\fw}{{\mathfrak w}}
\newcommand{\fx}{{\mathfrak x}}

\newcommand{\fD}{{\mathfrak D}}

\newcommand{\fM}{{\mathfrak M}}

\newcommand{\fR}{{\mathfrak R}}

\newcommand{\fW}{{\mathfrak W}}
\newcommand{\fX}{{\mathfrak X}}

\newcommand{\slD}{\slashed D}
\newcommand{\slS}{\slashed S}

\newcommand{\bsP}{{\bm{\sP}}}
\newcommand{\bsQ}{{\bm{\sQ}}}
\newcommand{\bsigma}{{\bm{\sigma}}}

\newcommand{\btau}{{\bm{\tau}}}

\author{
  Aleksander Doan
  \and
  Thomas Walpuski
}
\title{
  On the existence of harmonic $\Z_2$ spinors
}
\date{2018-04-24}
\begin{document}

\maketitle

\begin{abstract}
  We prove the existence of singular harmonic $\Z_2$ spinors on $3$--manifolds with $b_1 > 1$. 
  The proof relies on a wall-crossing formula for solutions to the Seiberg--Witten equation with two spinors.
  The existence of singular harmonic $\Z_2$ spinors and the shape of our wall-crossing formula shed new light on recent observations made by \citet{Joyce2016} regarding Donaldson and Segal's proposal for counting $\Gtwo$--instantons \cite{Donaldson2009}.
\end{abstract}

\section{Introduction}
\label{Sec_Introduction}

The notion of a harmonic $\Z_2$ spinor was introduced by \citet{Taubes2014} as an abstraction of various limiting objects appearing in compactifications of moduli spaces of flat $\PSL_2(\C)$--connections over $3$--manifolds \cite{Taubes2012} and solutions to the Kapustin--Witten equation \cite{Taubes2013}, the Vafa--Witten equation \cite{Taubes2017}, and the Seiberg--Witten equation with multiple spinors \cites{Haydys2014}{Taubes2016}.

\begin{definition}
  \label{Def_HarmonicZ2Spinor}
  Let $M$ be a closed Riemannian manifold and $\bS$ a Dirac bundle over $M$.%
  \footnote{%
    A Dirac bundle is a bundle of Clifford modules together with a metric and a compatible connection; see, \cite[Chapter II, Definition 5.2]{Lawson1989}.%
  }
  Denote by $\slD\co \Gamma(\bS) \to \Gamma(\bS)$ the associated Dirac operator.
  A \defined{$\Z_2$ spinor} with values in $\bS$ is a triple $(Z,\fl,\Psi)$ consisting of:
  \begin{enumerate}
    \item a proper closed subset $Z \subset M$,
    \item a Euclidean line bundle $\fl \to M \setminus Z$, and
    \item a section $\Psi\in\Gamma(M\setminus Z, \bS\otimes\fl)$
  \end{enumerate}
  such that $\abs{\Psi}$ extends to a Hölder continuous function on $M$ with $\abs{\Psi}^{-1}(0) = Z$ and $\abs{\nabla\Psi} \in L^2(M\setminus Z)$.
  We say that  $(Z,\fl,\Psi)$ is \defined{singular} if $\fl$ does not extend to a Euclidean line bundle on $M$.
  A $\Z_2$ spinor $(Z,\fl,\Psi)$ is called \defined{harmonic} if
  \begin{equation*}
    \slD\Psi=0
  \end{equation*}
  holds on $M\setminus Z$.
\end{definition}

\begin{remark}
  \citet{Taubes2014} proved that if $M$ is of dimension $3$ or $4$, then $Z$ has Hausdorff codimension at least 2.
  More recently, \citet{Zhang2017} proved that $Z$ is, in fact, rectifiable.
\end{remark}

\begin{remark}
  If $\fl$ extends to a Euclidean line bundle on $M$, then $\bS\otimes\fl$ extends to a Dirac bundle on $M$ and $\Psi$ extends to a harmonic spinor defined on all of $M$ which takes values in $\bS\otimes\fl$ and vanishes precisely along $Z$.
\end{remark}

The harmonic $\Z_2$ spinors appearing as limits of flat $\PSL_2(\C)$--connections over a $3$--manifold $M$ take values in the bundle $\underline{\R}\oplus T^*M$ equipped with the Dirac operator
\begin{equation*}
  \slD =
  \begin{pmatrix}
    0 & \rd^* \\
    \rd & * \rd 
  \end{pmatrix}.
\end{equation*}

The harmonic $\Z_2$ spinors appearing as limits of the Seiberg--Witten equation with two spinors in dimension three take values in the Dirac bundle
\begin{equation*}
  \bS = \Re(S\otimes E).
\end{equation*}
This bundle is constructed as follows.
Denote by $S$ the spinor bundle of a spin structure $\fs$ on the $3$--manifold $M$.
Denote by $E$ a rank two Hermitian bundle with trivial determinant line bundle $\Lambda^2_\C E$ and equipped with a compatible connection.
Both $S$ and $E$ are quaternionic vector bundles and thus have complex anti-linear endomorphism $j_S$, $j_E$ satisfying $j_S^2 = -\id_S$ and $j_E^2 = -\id_E$;
see also \autoref{Sec_H2///U(1)}.
These endow the complex vector bundle $S\otimes E$ with a real structure: $\overline{s\otimes e} \coloneq j_ss \otimes j_Ee$.
The $\C$--linear Dirac operator acting on $\Gamma(S \otimes E)$ preserves $\Gamma(\Re(S\otimes E))$ and gives rise to an $\R$--linear Dirac operator acting on $\Gamma(\Re(S\otimes E))$.

Henceforth, we specialize situation described in the previous paragraph.
The Dirac operator on $\Re(S\otimes E)$, and thus also the notion of a harmonic $\Z_2$ spinor, depends on the choice of a Riemannian metric on $M$ and a connection on $E$.

\begin{definition}
  \label{Def_ParameterSpace}
  Let $\Met(M)$ be the space of Riemannian metrics on $M$ and $\sA(E)$ the space of $\SU(2)$ connections on $E$.
  The \defined{space of parameters} is
  \begin{equation*}
    \sP \coloneq \Met(M) \times \sA(E)
  \end{equation*}
  equipped with the $C^{\infty}$ topology.
  Given a spin structure $\fs$ on $M$ and $\bp\in\sP$,
  we denote by $\slD^\fs_\bp$ the corresponding Dirac operator on $\Gamma(\Re(S\otimes E))$.
  We will say that a triple $(Z,\fl,\Psi)$ is a harmonic $\Z_2$ spinor \defined{with respect to} $\bp$ if it satisfies the conditions of 
  \autoref{Def_HarmonicZ2Spinor} with $\slD = \slD^\fs_\bp$.
\end{definition}

\begin{question}
  \label{Q_SingularHarmonicZ2Spinors}
  For which parameters $\bp\in\sP$ does there exist a singular harmonic $\Z_2$ spinor with respect to $\bp$?
\end{question}

The answer to this question for non-singular harmonic $\Z_2$ spinors (that is: harmonic spinors) is well-understood.
Let $\sW^\fs$ be the set of $\bp\in\sP$ for which $\dim\ker\slD^\fs_\bp > 0$.
It is the closure of $\sW_1^\fs$, the set of $\bp$ for which $\dim\ker\slD^\fs_\bp = 1 $.
Moreover, $\sW^\fs_1$ is a cooriented, codimension one submanifold of $\sP$ and $\sW^\fs\setminus \sW^\fs_1$ has codimension three.
(See \autoref{Prop_NaturalCoorientation} and \autoref{Prop_Wsk}.)
The intersection number of a path $(\bp_t)_{t\in[0,1]}$ with $\sW^\fs_1$ is given by the \defined{spectral flow} of the path of operators $(\slD^\fs_{\bp_t})_{t\in[0,1]}$, defined by Atiyah, Patodi, and Singer \cite[Section 7]{Atiyah1976}.
Therefore, along any path with non-zero spectral flow there exists a parameter $\bp_\star$ such that $\dim\ker\slD^\fs_{\bp_\star} > 0$.
Moreover, if the path is generic,%
\footnote{%
  \defined{Generic} means from a residual subset of the space of objects in question.
  A subset of a topological space is residual if it contains a countable intersection of open and dense subsets. 
  Baire's theorem asserts that a residual subset of a complete metric space is dense.
}
then the kernel is spanned by a nowhere vanishing spinor (because $\dim M < \rank\bS$; see \autoref{Def_RegularPath} and \autoref{Prop_TransverseSpectralFlow}).

By contrast, little is known about the existence of singular harmonic $\Z_2$ spinors.
The only examples known thus far have been obtained by means of complex geometry, on Riemannian $3$--manifolds of the form $M=S^1\times\Sigma$ for a Riemann surface $\Sigma$; see \cite[Theorem 1.2]{Taubes2012} in the case $\bS=\underline{\R}\oplus T^*M$ and \cite[Section 3]{Doan2017} in the case $\bS=\Re(S\otimes E)$. 
We remedy this situation by proving---in a rather indirect way---that $3$-manifolds abound with singular harmonic $\Z_2$ spinors.

\begin{theorem}
  \label{Thm_ExistenceOfSingularHarmonicZ2Spinors}
  For every closed, connected, oriented $3$--manifold $M$ with $b_1(M) > 1$ there exist a $\bp_\star\in\sP$ and a singular harmonic $\Z_2$ spinor with respect to $\bp_\star$.
  In fact, there is a closed subset $\sW_b \subset \sP$ and a non-zero cohomology class $\omega \in H^1(\sP\setminus\sW_b, \Z)$ with the property that if $(\bp_t)_{t\in S^1}$ is a generic loop in $\sP\setminus \sW_b$ and
  \begin{equation*}
    \omega([\bp_t]) \neq 0,
  \end{equation*}
  then there exists a singular harmonic $\Z_2$ spinor with respect to some $\bp_\star$ in $(\bp_t)_{t\in S^1}$.
\end{theorem}

\begin{remark}
  The definition of $\sW_b$ is given in \autoref{Def_Walls} and the precise meaning of a generic loop is given in \autoref{Def_RegularPath} and \autoref{Prop_Transversality}.
\end{remark}

\begin{remark}
  \autoref{Thm_ExistenceOfSingularHarmonicZ2Spinors} suggests that on $3$--manifolds the appearance of singular harmonic $\Z_2$ spinors is a codimension one phenomenon---as is the appearance of harmonic spinors.
  This is in consensus with the work of \citet{Takahashi2015,Takahashi2017}, who proved that the linearized deformation theory of singular harmonic $\Z_2$ spinors with $Z = S^1$ is an index zero Fredholm problem (or index minus one, after scaling is taken into account).
\end{remark}

\begin{remark}
  \label{Rmk_ReducibleSolutions}
  The assumption $b_1(M)>1$ has to do with reducible solutions to the Seiberg--Witten equation with two spinors.
  We expect a variant of the theorem to be true for $b_1(M) \in \set{0,1}$ as well.
  In this case, one also has to take into account the wall-crossing caused by reducible solutions which was studied in classical Seiberg--Witten theory by \citet{Chen1997,Lim2000}.
\end{remark}

The proof of \autoref{Thm_ExistenceOfSingularHarmonicZ2Spinors} relies on the wall-crossing formula for $n(\bp)$,
the signed count of solutions to the Seiberg--Witten equation with two spinors.
The number $n(\bp)$ is defined provided $\bp$ is generic and there are no singular harmonic $\Z_2$ spinors with respect to $\bp$.
The wall-crossing formula, whose precise statement is \autoref{Thm_TwoSeibergWittenWallCrossing}, can be described as follows.
Let $\sW^\fs_{1,\emptyset}$ be the set of $\bp \in \sP$ for which $\ker\slD^\fs_\bp = \R\Span{\Psi}$ with $\Psi$ nowhere vanishing,
and let $\sW_{1,\emptyset}$ be the union of all $\sW^\fs_{1,\emptyset}$ for all spin structures $\fs$.
There is a closed subset $\sW_b \subset \sP$, as in \autoref{Thm_ExistenceOfSingularHarmonicZ2Spinors}, such that $\sW_{1,\emptyset}$ is a closed, cooriented, codimension one submanifold of $\sP\setminus \sW_b$. 
If $(\bp_t)_{t\in[0,1]}$ is a generic path in $\sP\setminus \sW_b$ and there are no singular harmonic $\Z_2$ spinors with respect to any $\bp_t$ with $t \in [0,1]$, then the difference
\begin{equation*}
  n(\bp_1) - n(\bp_0) 
\end{equation*}
is equal to the intersection number of the path $(\bp_t)_{t\in[0,1]}$ with $\sW_{1,\emptyset}$.
In particular, if $(\bp_t)_{t\in S^1}$ is a generic loop whose intersection number with $\sW_{1,\emptyset}$ is non-zero, then there must be a singular harmonic $\Z_2$ spinor for some $\bp_\star$ in $(\bp_t)$.

\begin{remark}
  Although the wall-crossing for $n(\bp)$ does occur when the spectrum of $\slD^\fs_\bp$ crosses zero,
  the contribution of a nowhere vanishing harmonic spinor to the wall-crossing formula is not given by the sign of the spectral crossing but instead by the mod $2$ topological degree of the harmonic spinor, as in \autoref{Def_RelativeDegree}.
  Therefore, the wall-crossing formula is not given by the spectral flow.
  This should be contrasted with the wall-crossing phenomenon for the classical Seiberg--Witten equation caused by reducible solutions, as in \autoref{Rmk_ReducibleSolutions}.
  Indeed, the wall-crossing described in this paper is a result of the non-compactness of the moduli spaces of solutions and as such it is a new phenomenon, with no counterpart in classical Seiberg--Witten theory.
\end{remark}

The cohomology class $\omega \in H^1(\sP\setminus\sW_b,\Z)$ in \autoref{Thm_ExistenceOfSingularHarmonicZ2Spinors} is defined by intersecting loops in $\sP\setminus\sW_b$ with $\sW_{1,\emptyset}$.
We prove that $\omega$ is non-trivial, by exhibiting a loop $(\bp_t)_{t\in S^1}$ on which $\omega$ evaluates as $\pm 2$.
In particular, $\sP\setminus\sW_b$ is not simply-connected and we can take $(\bp_t)_{t\in S^1}$ to be a small loop linking $\sW_b$.

\begin{remark}
  The discussion in the article is related to an observation made by \citet[Section 8.4]{Joyce2016} which points out potential issues with the Donaldson--Segal program for counting $\Gtwo$--instantons.
  We discuss this in detail in \autoref{Sec_JoyceDonaldsonSegal}.
\end{remark}


\paragraph{Acknowledgements}
This project originated from a question posed to us by Simon Donaldson.
It is a pleasure to acknowledge his influence on our work.
We are particularly grateful for his insistence that one should investigate the relationship between the sign of the spectral crossing and the wall-crossing formula for the Seiberg--Witten equation with two spinors.
We thank the anonymous referee for numerous helpful comments and suggestions for improving the exposition of this article.
This material is based upon work supported by  \href{https://www.nsf.gov/awardsearch/showAward?AWD_ID=1754967&HistoricalAwards=false}{the National Science Foundation under Grant No.~1754967}
and
\href{https://sites.duke.edu/scshgap/}{the Simons Collaboration Grant on ``Special Holonomy in Geometry, Analysis and Physics''}.

\paragraph{Notation}

Here is a summary of various notations used throughout the article:
\begin{itemize}
\item
  We use $\fw$ to denote a spin$^c$ structure.
  The associated complex spinor bundle is denoted by $W$. 
  The \defined{determinant line bundle} of $W$ is $\det W = \Lambda^2_\C W$. 
\item
  For every $\bp = (g,B)$ in the parameter space $\sP$ and for every connection $A \in \sA(\det W)$ we write
  \begin{equation*}
    \slD_{A,\bp} \colon \Gamma(\Hom(E,W)) \to \Gamma(\Hom(E,W))
  \end{equation*}
  for the $\C$--linear Dirac operator induced the connection $B$ on $E$ as well as the spin$^c$ connection on $W$,
  determined by $A$ and the Levi--Civita connection of $g$.  
  We will suppress the subscript $\bp$ from the notation when its presence is not relevant to the current discussion.
\item
  We use $\fs$ (possibly with a subscript: $\fs_0$, $\fs_1$, etc.) to denote a spin structure on $M$ .
  The associated spinor bundle is denoted by $S_{\fs}$. 
\item
  For every $\bp = (g,B) \in \sP$ we write
  \begin{equation*}
    \slD_{\bp}^{\fs} \colon \Gamma(\Re(S_{\fs} \otimes E)) \to \Gamma(\Re(S_{\fs} \otimes E))
  \end{equation*}
  for the $\R$--linear Dirac operator induced by the spin connection on $S_{\fs}$,
  associated with the Levi--Civita connection of $g$,
  and the connection $B$ on $E$.
\end{itemize}

\section{The Seiberg--Witten equation with two spinors}
\label{Sec_CountingSolutions}

Fix a spin$^c$ structure $\fw$ and denote its complex spinor bundle by $W$.
Set
\begin{equation*}
  \sZ \coloneq \ker\(\rd\co \Omega^2(M,i\R) \to \Omega^3(M,i\R)\).
\end{equation*}

\begin{definition}
  Let $\bp = (g,B) \in \sP$ and $\eta \in \sZ$.
  The \defined{$\eta$--perturbed Seiberg--Witten equation with two spinors} is the following differential equation for $(\Psi,A) \in \Gamma(\Hom(E,W)) \times \sA(\det W)$:
  \begin{equation}
    \label{Eq_PerturbedTwoSeibergWitten}
    \begin{split}
      \slD_A\Psi &= 0 \qand \\
      \frac12 F_{A}+\eta &= \mu(\Psi).
    \end{split}
  \end{equation}
  Here $\slD_A = \slD_{A,\bp}$ is the $\C$--linear Dirac operator on $\Hom(E,W)$ and
  \begin{equation*}
    \mu(\Psi) = \mu_\bp(\Psi) \coloneq \Psi \Psi^* - \frac12\abs{\Psi}^2\,\id_W
  \end{equation*}
  is a section of $i\su(W)$
  which is identified with an element of $\Omega^2(M,i\R)$ using the Clifford multiplication.
  Both equations depend on the choice of $\bp$ and the second equation depends also on the choice of $\eta$.%
  \footnotemark
  \footnotetext{%
    While the equation \eqref{Eq_PerturbedTwoSeibergWitten} makes sense for any $\eta \in \Omega^2(M,i\R)$, the existence of a solution implies that $\rd \eta = 0$;
    see \cites[Proposition 2.4]{Doan2017}[Proposition A.4]{Doan2017a}.
    Thus, we consider only $\eta \in \sZ$.
  }
  
  Let $\sG(\det W)$ be the gauge group of $\det W$.
  For $(\bp,\eta) \in \sP\times\sZ$, we denote by
  \begin{equation*}
    \fM_\fw(\bp,\eta)
    \coloneq
    \bigg\{
    [\Psi,A] \in \frac{\Gamma(\Hom(E,W)) \times \sA(\det W)}{\sG(\det W)}
    :
    \begin{array}{@{}l@{}}
      (\Psi,A) \text{ satisfies } \eqref{Eq_PerturbedTwoSeibergWitten} \\
      \text{with respect to } \bp \text{ and } \eta
    \end{array}
    \bigg\}
  \end{equation*}
  the \defined{moduli space} of solutions to \eqref{Eq_PerturbedTwoSeibergWitten}. 
\end{definition}

As discussed in \cites[Section 2.2]{Doan2017}[Section 2]{Doan2017a},
the infinitesimal deformation theory of \eqref{Eq_PerturbedTwoSeibergWitten} around a solution $(\Psi,A)$, is controlled by the linear operator
\begin{multline*}
  L_{\Psi,A} =  L_{\Psi,A,\bp} \co \Gamma(\Hom(E,W) \oplus \Omega^1(M,i\R) \oplus \Omega^0(M,i\R) \\
  \to \Gamma(\Hom(E,W)) \oplus \Omega^1(M,i\R) \oplus \Omega^0(M,i\R)
\end{multline*}  
defined by
\begin{equation}
  \label{Eq_LinearizedOperator}
  \begin{split}
  L_{\Psi,A,\bp}
  \coloneq
  \begin{pmatrix}
    -\slD_{A,\bp} & -\fa_{\Psi,\bp} \\
    -\fa_{\Psi,\bp}^* & \fd_\bp
  \end{pmatrix}
\end{split}
\end{equation}
with
\begin{equation}
  \label{Eq_DeltaAndA}
  \fd = \fd_\bp
  \coloneq
  \begin{pmatrix}
    *\rd & \rd \\
    \rd^* &
  \end{pmatrix}
  \qandq
  \fa_\Psi = \fa_{\Psi,\bp}
  \coloneq
  \begin{pmatrix}
    \bar\gamma(\cdot)\Psi & \rho(\cdot)\Psi
  \end{pmatrix}.
  \footnotemark
\end{equation}
\footnotetext{%
  This linearization is obtained by writing a connection near $A_0$ as $A_0 + 2a$.%
}%
Here $\bar\gamma$ is the Clifford multiplication by elements of $T^*M\otimes i\R$ and $\rho$ is the linearized action of the gauge group: pointwise multiplication by elements of $i\R$.
The Hodge star operator $*$ and $\bar\gamma$ both depend on $\bp$, but we have suppressed this dependence in the notation.
  
\begin{definition}
  We say that a solution $(\Psi,A)$ of \eqref{Eq_PerturbedTwoSeibergWitten} is \defined{irreducible} if $\Psi \neq 0$,
  and \defined{unobstructed} if $L_{\Psi,A}$ is invertible.
\end{definition}

If $(\Psi,A)$ is irreducible and unobstructed,
then it represents an isolated point in $\fM_\fw(\bp,\eta)$;
see \cites[Proposition 2.13]{Doan2017}[Proposition 1.25]{Doan2017a}.

\begin{prop}[{\cite[Proposition 2.28]{Doan2017}}]
  \label{Prop_Transversality}
  If $b_1(M) > 0$, then for every $(\bp,\eta)$ from a residual subset of $\sP\times\sZ$ all solutions to \eqref{Eq_PerturbedTwoSeibergWitten} are irreducible and unobstructed.
\end{prop}

For every $(\bp,\eta)$ as in \autoref{Prop_Transversality}, the moduli space $\fM_\fw(\bp,\eta)$ is a zero-dimensional manifold.
It can be oriented as explained in \cite[Section 2.6]{Doan2017}.
The following discussion outlines the orientation procedure.

\begin{prop}
  \label{Prop_SpectralFlow}
  Let $(\Psi_t,A_t,\bp_t)_{t \in [0,1]}$ be a path in $\Gamma(\Hom(E,W))\times\sA(\det W)\times\sP$
  The value of the spectral flow
  \begin{equation*}
    \SF\((L_{\Psi_t,A_t,\bp_t})_{t\in[0,1]}\) \in \Z
  \end{equation*}
  only depends on $(\Psi_0,A_0,\bp_0)$ and $(\Psi_1,A_1,\bp_1)$.
\end{prop}

\begin{convention}
  \label{Conv_SpectralFlow}
  If $(D_t)_{t\in[0,1]}$ is a path of self-adjoint Fredholm operators with $D_0$ or $D_1$ not invertible,
  we define $\SF\((D_t)_{t\in[0,1]}\) \coloneq (D_t+\lambda)_{t\in[0,1]}$ for $0 < \lambda \ll 1$.
  This convention was introduced by \citet[Section 7]{Atiyah1976}.
\end{convention}

\begin{proof}[Proof of \autoref{Prop_SpectralFlow}]
  Since the spectral flow is homotopy invariant,
  this is a consequence of the fact that $\Gamma(\Hom(E,W))\times\sA(\det W)\times\sP$ is contractible.
\end{proof}

\begin{definition}
  \label{Def_OrientationTransport}
  For $(\Psi_0,A_0,\bp_0)$ and $(\Psi_1,A_1,\bp_1) \in \Gamma(\Hom(E,W))\times\sA(\det W)\times\sP$,
  we define the \defined{orientation transport} 
  \begin{equation*}
    \OT\((\Psi_0,A_0,\bp_0),(\Psi_1,A_1,\bp_1)\) \coloneq (-1)^{\SF\((L_{\Psi_t,A_t,\bp_t})_{t\in[0,1]}\)},
  \end{equation*}
  for path $(\Psi_t,A_t,\bp_t)_{t\in[0,1]}$ from $(\Psi_0,A_0,\bp_0)$ to $(\Psi_1,A_1,\bp_1)$.
\end{definition}

\begin{remark}
  The orientation transport can be alternatively defined using the determinant line bundle of the family of Fredholm operators $(L_{\Psi,A,\bp})$ as $(\Psi,A,\bp)$ varies in $\Gamma(\Hom(E,W))\times\sA(\det W)\times\sP$.
  This point of view is explained in detail in \autoref{Sec_DeterminantLineBundles} as it is needed in a technical, but crucial, part in the proof of \autoref{Prop_HowToDetermineSigma}.
\end{remark}

Since the spectral flow is additive with respect to path composition,
we have
\begin{equation}
  \label{Eq_OrientationTransportComposition}
  \begin{split}
    &\OT\((\Psi_0,A_0,\bp_0),(\Psi_2,A_2,\bp_2)\) \\
    &\qquad=
    \OT\((\Psi_0,A_0,\bp_0),(\Psi_1,A_1,\bp_1)\)\cdot
    \OT\((\Psi_1,A_1,\bp_1),(\Psi_2,A_2,\bp_2)\).
  \end{split}
\end{equation}

\begin{prop}
  \label{Prop_OrientationTransport}
  ~
  \begin{enumerate}
  \item
    \label{Prop_OrientationTransport_0}
    For every $(A_0,\bp_0)$ and $(A_1,\bp_1) \in \sA(\det W)\times\sP$,
    we have
    \begin{equation*}
      \OT((0,A_0,\bp_0),(0,A_1,\bp_1)) = +1.
    \end{equation*}
  \item
    \label{Prop_OrientationTransport_Gauge}
    For every $(\Psi,A,\bp) \in \Gamma(\Hom(E,W))\times\sA(\det W)\times\sP$ and every $u \in \sG(\det W)$,
    we have
    \begin{equation*}
      \OT((\Psi,A,\bp),(u\cdot \Psi,u\cdot A,\bp)) = +1.
    \end{equation*}
  \end{enumerate}
\end{prop}

\begin{proof}
  Observe that
  \begin{equation*}
    L_{0,A_i,\bp_i} = \slD_{A_i,\bp_i} \oplus \fd_{\bp_i}
  \end{equation*}
  where $\slD_{A_i,\bp_i}$ is a complex-linear Dirac operator on $\Hom(E,W)$ and $\fd_{\bp_i}$ is defined in \eqref{Eq_DeltaAndA}.
  Let $(A_t,\bp_t)_{t\in [0,1]}$ be a path in $\sA(\det W)\times \sP$ joining $(A_0,\bp_0)$ and $(A_1,\bp_1)$.
  The spectral flow of $(\slD_{A_t,\bp_t})_{t\in[0,1]}$ is even because the operators $\slD_{A_t,\bp_t}$ are complex linear.
  The spectral flow of $(\fd_{\bp_t})_{t\in[0,1]}$ is trivial because the dimension of the kernel of $\fd_{\bp_t}$ is $1 + b_1(M)$ and does not depend on $t \in [0,1]$.
  This proves \autoref{Prop_OrientationTransport_0}.
  
  We prove \autoref{Prop_OrientationTransport_Gauge}.
  Choose a path $(\Psi_t,A_t)_{t\in[0,1]}$ from $(\Psi,A)$ to $(u\cdot\Psi,u\cdot A)$.  
  The spetral flow $\SF(L_{\Psi_t, A_t})$ can be computed using a theorem of Atiyah--Patodi--Singer \cite[Section 7]{Atiyah1976} as follows.
  Denote by $\bW^+ \to S^1 \times M$ the mapping cylinder of $u$, that is, the bundle obtained from pulling-back $W$ to $[0,1] \times M$ and identifying the fibers over $\set{0} \times M$ and $\set{1} \times M$ using $u$.
  $\bW^+$ is in fact the positive spinor bundle of a spin$^c$ structure on the $4$--manifold $X \coloneq S^1 \times M$.
  The path of operators $\del_t - L_{\Psi_t,A_t}$ glues to an operator $\bL$ acting on $\Gamma(X, \Hom(E,\bW^+)) \oplus \Omega^1(X,i\R)$.
  Here we use the identification $\Lambda^1 T^*X = \Lambda^0 T^*M \oplus \Lambda^1 T^*M$.
  The Atiyah--Patodi--Singer Index Theorem asserts that
  \begin{equation*}
    \SF(L_{\Psi_t, A_t}) = \ind \bL.
  \end{equation*}
  The operator $\bL$ is seen to be homotopic through Fredholm operators to the direct sum of a complex-linear Dirac operator and the operator
  \begin{equation*}
    \rd^+ + \rd^* \colon \Omega^1(X,i\R) \to \Omega^+(X,i\R)  \oplus \Omega^0(X,i\R).
  \end{equation*}
  The index of this operator is $b^1(X) - 1 - b^+(X)$.
  For $X = S^1 \times M$ we have $b^+(X) = b^1(M)$ and $b^1(X) = b^1(M) + 1$, so $b^1(X) - 1 - b^+(X) = 0$.
  We conclude that the index of $\bL$ is even.
  This concludes the proof of \autoref{Prop_OrientationTransport_Gauge}.
\end{proof}

\autoref{Prop_OrientationTransport} and \autoref{Eq_OrientationTransportComposition} show that the following definition is independent of any of the choices being made.

\begin{definition}
  \label{Def_Sign}
  For $[\Psi,A] \in \fM_\fw(\bp,\eta)$,
  we define
  \begin{equation*}
    \sign [\Psi,A] \coloneq \OT((0,A_0,\bp_0),(\Psi,A,\bp)),
  \end{equation*}
  for any choice of $(A_0,\bp_0) \in \sA(\det W)\times \sP$.
\end{definition}

As we will see shortly, if $(\bp,\eta)$ is generic, and under the assumption that there are no singular harmonic $\Z_2$ spinors with respect to $\bp$,
then $\fM_\fw(\bp,\eta)$ is a compact, oriented, zero-dimensional manifold, that is: a finite set of points with prescribed signs.
In this situation, we define
\begin{equation}
  \label{Eq_nSWCount}
  n_\fw(\bp,\eta)
  \coloneq
  \sum_{[\Psi,A] \in \fM_\fw(\bp, \eta)} \sign [\Psi,A],
\end{equation}
the signed count of solutions to the Seiberg--Witten equation with two spinors.


\section{Compactness of the moduli space}
\label{Sec_Compactness}

In general, $\fM_\fw(\bp,\eta)$ might be non-compact;
and even if it is compact for given $(\bp,\eta)$, compactness might still fail as $(\bp,\eta)$ varies.
This can only happen when, along a sequence of solutions to \eqref{Eq_PerturbedTwoSeibergWitten}, the $L^2$--norm of the spinors goes to infinity.
The following result describes in which sense one can still take a rescaled limit in this situation.

\begin{theorem}[{\cite[Theorem 1.5]{Haydys2014}}]
  \label{Thm_HW}
  Let $(\bp_i, \eta_i)$ be a sequence in $\sP\times \sZ$ which converges to $(\bp,\eta)$ in $C^\infty$.
  Let $(\Psi_i,A_i)$ be a sequence of solutions of \eqref{Eq_PerturbedTwoSeibergWitten} with respect to $(\bp_i,\eta_i)$.
  If $\limsup_{i \to \infty} \Abs{\Psi_i}_{L^2} = \infty$, then after rescaling $\tilde\Psi_i = \Psi_i/\Abs{\Psi_i}_{L^2}$ and passing to a subsequence the following hold:
  \begin{enumerate}
  \item The subset 
    \begin{equation*}
      Z \coloneq \set*{ x \in M : \limsup_{i \to \infty} \abs{\tilde\Psi_i(x)} = 0 }
    \end{equation*}
    is closed and nowhere-dense.
    (In fact, $Z$ has Hausdorff dimension at most one \cite[Theorem 1.2]{Taubes2014}.)
  \item
    There exist $\Psi \in \Gamma(M\setminus Z, \Hom(E,W))$ and a connection $A$ on $\det W|_{M\setminus Z}$ satisfying the limiting equation
    \begin{equation}
      \label{Eq_LimitingEquation}
      \slD_A \Psi = 0
      \qandq
      \mu(\Psi) = 0
    \end{equation}
    on $M\setminus Z$ with respect to $\bp$.
    The pointwise norm $\abs{\Psi}$ extends to a Hölder continuous function on all of $M$ and
    \begin{equation*}
      Z = \abs{\Psi}^{-1}(0).
    \end{equation*}
    Moreover, $A$ is flat with monodromy in $\Z_2$.
  \item
    On $M \setminus Z$, up to gauge transformations,
    $\tilde\Psi_i$ weakly converges to $\Psi$ in $W^{2,2}_{\loc}$ and
    $A_i$ weakly converges to $A$ in $W^{1,2}_{\loc}$.
    There is a constant $\gamma > 0$ such that $\abs{\tilde\Psi_i}$ converges to $\abs{\Psi}$ in $C^{0,\gamma}(M)$.
  \end{enumerate}
\end{theorem}

We expect that the convergence $(\tilde\Psi_i,A_i) \to (\Psi,A)$ can be improved to $C^{\infty}_{\loc}$ on $M \setminus Z$;
cf.~\cite[Theorem 1.5]{Doan2017b}.
In \autoref{Sec_SmoothConvergenceIfZEmpty}, we will show that this is indeed the case if $Z$ is empty.

The following proposition will give us a concrete understanding of solutions to the limiting equation \eqref{Eq_LimitingEquation} which are defined on all of $M$, that is: for which the set $Z$ is empty.
It is a special case of the \defined{Haydys correspondence}.

\begin{prop}[{cf. \cite[Appendix A]{Haydys2014}}]
  \label{Prop_SpecialHaydysCorrespondence}
  If $(\Psi,A) \in \Gamma(\Hom(E,W)) \times \sA(\det W)$ is a solution of \eqref{Eq_LimitingEquation} and $\Psi$ is nowhere vanishing,
  then:
  \begin{enumerate}
  \item 
    $\det W$ is trivial; in particular, $\fw$ is induced by a spin structure,
  \item
    after a gauge transformation we can assume that $A$ is the product connection and there exists a unique spin structure $\fs$ inducing $\fw$ and such that $\Psi$ takes values in $\Re(E\otimes S_\fs) \subset \Gamma(\Hom(E,W))$.
    Here $S_\fs$ is the spinor bundle of $\fs$.
  \item
    $\Psi$ lies in the kernel of $\slD_{\bp}^\fs \co \Gamma(\Re(S_\fs\otimes E)) \to \Gamma(\Re(S_\fs\otimes E))$.
  \end{enumerate}
  Moreover, any nowhere vanishing element in $\ker \slD_{\bp}^\fs$ for any spin structure $\fs$ inducing $\fw$ gives rise to a solution of \eqref{Eq_LimitingEquation}.
\end{prop}

\begin{remark}
  \label{Rmk_SpinStructure}
  The set of spin structures is a torsor over $H^1(M,\Z_2)$ while the set of spin$^c$ structures is a torsor over $H^2(M,Z)$.
  If $\beta \co H^1(M,\Z_2) \to H^2(M,\Z)$ denotes the Bockstein homomorphism in the exact sequence
  \begin{equation*}
    \cdots \to H^1(M,\Z_2) \xrightarrow{\beta} H^2(M,\Z) \xrightarrow{2\times} H^2(M,\Z) \to \cdots,
  \end{equation*}
  then the set of a all spin structures $\fs$ inducing the spin$^c$ structure $\fw$ is a torsor over $\ker\beta$.
  The set of all spin$^c$ structures $\fw$ with trivial determinant is a torsor over $\ker 2\times$, the $2$--torsion subgroup of $H^2(M,\Z)$.
\end{remark}

\begin{proof}[Proof of \autoref{Prop_SpecialHaydysCorrespondence}]
  Fix a spin structure $\fs_0$ and a Hermitian line bundle $L$ which induce $\fw$;
  in particular, $W = S_{\fs_0} \otimes L$ and $A$ induces a connection $A_0$ on $L$.
  By \autoref{Prop_H2///U(1)=Re/Z2},
  \begin{equation*}
    \Hom_\C(\C^2,\H)\hkred\U(1)
    =
    \paren*{\Re(\H\otimes_\C\C^2)\setminus\set{0}}/\Z_2;
  \end{equation*}
  hence, $\Psi$ gives rise to a section $s \in \Gamma(\fX)$ with
  \begin{equation*}
    \fX = \paren*{\Re(S_{\fs_0}\otimes E)\setminus\set{0}}/\Z_2
  \end{equation*}
  satisfying the \defined{Fueter equation}, that is: local lifts of $s$ to $\Re(\H\otimes_\C\C^2)$ satisfy the Dirac equation.
  The Haydys correspondence \cite[Proposition 3.2]{Doan2017a}
  asserts that:
  \begin{itemize}
  \item
    any $s \in \Gamma(\fX)$ can be lifted; that is:
    there exist a Hermitian line bundle $L$, $\Psi \in \Gamma(\Hom(E, S_{\fs_0}\otimes L)) = \Gamma(\Hom(E,W))$, as well as $A_0 \in \sA(L)$ satisfying \eqref{Eq_LimitingEquation}, and
  \item
    $L$ is determined by $s$ up to isomorphism and any two lifts of $s$ are related by a unique gauge transformation in $\sG(L)$.
  \end{itemize}

  We claim that $s$ can, in fact, be lifted to a section $\tilde \Psi \in \Gamma(\Re(E\otimes S_{\fs_0})\otimes \fl)$ for some Euclidean line bundle $\fl$.
  To see this, cover $M$ with a finite collection of open balls $(U_\alpha)$ and trivialize $\Re(S_{\fs_0}\otimes E)$ over each $U_\alpha$. 
  On $U_\alpha$ the section $s$ is given by a smooth function $U_\alpha \to (\R^4 \setminus \set{0}) / \Z_2$  which can be lifted to a map $\tilde\Psi_\alpha \co U_\alpha \to \R^4 \setminus \set{0}$. 
  Over the intersection $U_\alpha \cap U_\beta$ of two different balls $U_\alpha$, $U_\beta$,
  we have $\tilde\Psi_\alpha = f_{\alpha\beta} \tilde\Psi_\beta$ for a local constant function $f_{\alpha\beta} \colon U_\alpha \cap U_\beta \to \set{-1,+1}$.
  The collection $(f_{\alpha\beta})$ is a Čech cocycle with values in $\Z_2$ and defines a $\Z_2$--bundle on $M$.
  Let $\fl$ be the associated Euclidean line bundle,
  The collection of local sections $(\tilde\Psi_\alpha)$ defines a section $\tilde \Psi \in \Gamma(\Re(S_{\fs_0}\otimes E)\otimes \fl)$ as we wanted to show. 
  
  Set $\tilde L \coloneq \fl\otimes_\R\C$.
  By \autoref{Prop_H2///U(1)=Re/Z2},
  \begin{equation*}
  \Re(E\otimes S_{\fs_0})\otimes \fl \subset \mu^{-1}(0) \subset \Hom(E, S_{\fs_0}\otimes\tilde L) =  \Hom(E,W);
  \end{equation*}
  and if $\tilde A \in \sA(\tilde L)$ denotes the connection induced by the canonical connection on $\fl$, then $\slD_{\tilde A}\tilde \Psi = 0$.

  It thus follows from the Haydys correspondence that
  $L \iso \fl\otimes_\R\C$ and, after this identification has been made and a suitable gauge transformation has been applied, $\Psi = \tilde \Psi$ and $A_0 = \tilde A$.
  This shows that $\det W = L^2$ is trivial and $\fw$ is induced by the spin structure $\fs$ obtained by twisting $\fs_0$ with $\fl$.
\end{proof}

\begin{definition}
  \label{Def_RegularSet}
  Denote by $\sP^\reg \subset \sP$ the subset consisting of those $\bp$ for which the real Dirac operator,
  introduced at the end of \autoref{Sec_Introduction},
  \begin{equation}
    \label{Eq_slDbpfs}
    \slD_\bp^\fs \co \Gamma(\Re(S_\fs\otimes E)) \to \Gamma(\Re(S_\fs\otimes E))
  \end{equation}
  is invertible for all spin structures $\fs$. 
  Set
  \begin{equation*}
    \sQ \coloneq \sP\times\sZ,
  \end{equation*}
  and denote by
  \begin{equation*}
    \sQ^\reg \subset \sQ
  \end{equation*}
  the subset consisting of those $(\bp,\eta)$ for which $\bp \in \sP^\reg$ and every solution $(\Psi,A)$ of \eqref{Eq_PerturbedTwoSeibergWitten} with respect to $(\bp,\eta)$ is irreducible and unobstructed.
\end{definition}

\begin{prop}[{\cites[Theorem 2.32]{Doan2017}[Theorem 1.5]{Anghel1996}[Theorem 1.2]{Maier1997}}]
  If $b_1(M) > 0$, then $\sQ^\reg$ is residual in $\sQ$.
\end{prop}

\begin{prop}
  If $(\bp,\eta) \in \sQ^\reg$ and there are no singular harmonic $\Z_2$ spinors with respect to $\bp$, then $\fM_\fw(\bp,\eta)$ is compact.
  In particular, $n_\fw(\bp,\eta) \in \Z$ as in \eqref{Eq_nSWCount} is defined.
\end{prop}

\begin{proof}
  By hypothesis we know that there are no singular harmonic $\Z_2$ spinors.
  By the definition of $\sP^\reg$ there are also no harmonic spinors.
  It thus follows from \autoref{Thm_HW} and \autoref{Prop_SpecialHaydysCorrespondence} that $\fM_\fw(\bp,\eta)$ is compact.
\end{proof}


\section{Wall-crossing and the spectral flow}
\label{Sec_WallCrossing}

In the absence of singular harmonic $\Z_2$ spinors, we can define $n_\fw(\bp,\eta)$ for every $(\bp,\eta) \in \sQ^\reg$.
However, $\sQ^\reg$ is not path-connected and $n_\fw(\bp,\eta)$ does depend on the path-connected component of $\sQ^\reg$ in which $(\bp,\eta)$ lies.
We study the wall-crossing for $n_\fw(\bp,\eta)$ by analyzing the family of moduli spaces $\fM_\fw(\bp_t,\eta_t)$ along paths of the following kind.

\begin{definition}
  \label{Def_RegularPath}
  Given $\bp_0, \bp_1 \in \sP^\reg$, denote by $\bsP^\reg(\bp_0,\bp_1)$ the space of smooth paths from $\bp_0$ to $\bp_1$ in $\sP$ such that for every spin structure $\fs$:
  \begin{enumerate}
  \item
    \label{Def_RegularPath_TransverseSpectralFlow}
    the path of Dirac operators $\paren[\big]{\slD_{\bp_t}^\fs}_{t\in[0,1]}$ has transverse spectral flow and
  \item
    \label{Def_RegularPath_PsiNowhereVanishing}
    whenever the spectrum of $\slD_{\bp_t}^\fs$ crosses zero, $\ker \slD_{\bp_t}^\fs$ is spanned by a nowhere vanishing section $\Psi \in \Gamma(\Re(E\otimes S_\fs))$.
  \end{enumerate}

  Given $(\bp_0,\eta_0),(\bp_1,\eta_1) \in \sQ^\reg$, denote by $\bsQ^\reg\paren[\big]{(\bp_0,\eta_0),(\bp_1,\eta_1)}$ the space of smooth paths $(\bp_t,\eta_t)_{t\in[0,1]}$ from $(\bp_0,\eta_0)$ to $(\bp_1,\eta_1)$ in $\sQ$ such that \itref{Def_RegularPath_TransverseSpectralFlow} and \itref{Def_RegularPath_PsiNowhereVanishing} hold and, moreover:
  \begin{enumerate}[resume]
  \item
    \label{Def_RegularPath_Unobstructed}
    For every $t_0 \in [0,1]$, every solution $(\Psi,A)$ of \eqref{Eq_PerturbedTwoSeibergWitten} is irreducible and either it is unobstructed (that is, the linearized operator $L_{\Psi,A}$ is invertible) or else $\coker L_{\Psi,A}$ has dimension one and is spanned by
    \begin{equation*}
      \pi\(\left.\frac{\rd}{\rd t}\right|_{t=t_0}
      \begin{pmatrix}
        -\slD_{\bp_t,A}\Psi \\
        *(F_A+\eta_t-\mu_{\bp_t}(\Psi)) \\
        0
      \end{pmatrix}
      \)
    \end{equation*}
    where $\pi\co \Gamma(\Hom(E,W))\oplus\Omega^1(M,i\R)\oplus\Omega^0(M,i\R) \to \coker L_{\Psi,A}$ denotes the $L^2$--orthogonal projection.
  \item
    \label{Def_RegularPath_DeltaNonZero}
    For any $\Psi$ and $t_0$ as in \itref{Def_RegularPath_PsiNowhereVanishing} with $\Abs{\Psi}_{L^2} = 1$,
    denote by
    \begin{equation*}
      \set*{ \(\Psi_\epsilon = \Psi + \epsilon^2\psi + O(\epsilon^4), A_\epsilon; t(\epsilon) = t_0 + O(\epsilon^2)\) : 0 \leq \epsilon  \ll 1 }
    \end{equation*}
    the family of solutions to
    \begin{align*}
      \slD_{A_\epsilon,\bp_{t(\epsilon)}}\Psi_\epsilon &= 0, \\
      \epsilon^2\paren*{\frac12F_{A_\epsilon}+\eta_{t(\epsilon)}} &= \mu_{\bp_{t(\epsilon)}}(\Psi_\epsilon), \qand \\
      \Abs{\Psi_\epsilon}_{L^2} &= 1
    \end{align*}
    obtained from \cite[Theorem 1.38]{Doan2017a}.
    Define
    \begin{equation*}
      \delta(\Psi,\bp_{t_0},\eta_{t_0})
      \coloneq \inner{\slD_{A_0,\bp_{t_0}}\psi}{\psi}_{L^2}.
    \end{equation*}
    We require that $\delta(\Psi,\bp_{t_0},\eta_{t_0}) \neq 0$.
  \end{enumerate}
\end{definition}

Condition \itref{Def_RegularPath_TransverseSpectralFlow} is necessary since the wall-crossing formula will involve the spectral flow of $\slD_{\bp_t}^\fs$.
In particular, it guarantees that $\dim \ker \slD^{\fs}_{\bp_t} > 0$ for only finitely many $t \in (0,1)$.
Condition \itref{Def_RegularPath_PsiNowhereVanishing} ensures that the harmonic $\Z_2$ spinors produced by \autoref{Thm_ExistenceOfSingularHarmonicZ2Spinors} are indeed singular.
Condition \itref{Def_RegularPath_Unobstructed} is used to show that
the union of all moduli spaces $\fM_\fw(\bp_t,\eta_t)$ as $t$ varies from $0$ to $1$ is an oriented smooth $1$--manifold (i.e., a disjoint union of circles and intervals) with oriented boundary $\fM_\fw(\bp_1,\eta_1) \cup -\fM_\fw(\bp_0,\eta_0)$.%
\footnote{%
  Here $-\fM_\fw(\bp_0,\eta_0)$ is the same space as $\fM_\fw(\bp_0,\eta_0)$, but all the orientations are reversed.
}
Finally, condition \itref{Def_RegularPath_DeltaNonZero} ensures that we can use the local model from \cite[Theorem 1.38]{Doan2017a} 
to study the wall-crossing phenomenon.

The following result shows that a generic path from $(\bp_0,\eta_0)$ to $(\bp_1,\eta_1)$ satisfies the conditions in \autoref{Def_RegularPath}.
Its proof is postponed to \autoref{Sec_TransversalityForPaths}.

\begin{prop}
  \label{Prop_TransversalityForPaths}
  Given $(\bp_0,\eta_0),(\bp_1,\eta_1) \in \sQ^\reg$, the subspace $\bsQ^\reg\paren[\big]{(\bp_0,\eta_0),(\bp_1,\eta_1)}$ is residual in the space of all smooth paths from $(\bp_0,\eta_0)$ to $(\bp_1,\eta_1)$ in $\sQ$.
\end{prop}

The next three sections are occupied with studying the wall crossing along paths in $\bsQ^\reg\paren[\big]{(\bp_0,\eta_0),(\bp_1,\eta_1)}$.
In order to state the wall-crossing formula, we need the following preparation.

\begin{prop}
  \label{Prop_DPsi}
  Denote by $\fs$ a spin structure inducing the spin$^c$ structure $\fw$ and by $A$ the product connection on $\det W$.
  If $\Psi$ is a nowhere vanishing section of $\Re(S_\fs\otimes E)$,
  then the following hold:
  \begin{enumerate}
  \item 
    Let $\fa_{\Psi}$ be the algebraic operator given by \eqref{Eq_DeltaAndA}.
    The map 
    \begin{equation*}
      \tilde\fa_\Psi \coloneq \abs{\Psi}^{-1}\fa_\Psi \co (T^*M\oplus\underline{\R})\otimes i\R \to \Im(S_\fs\otimes E)
    \end{equation*}
    is an isometry. 
    Here $\Im(S_\fs\otimes E)$ denotes the imaginary part of $S_\fs\otimes E$ defined using the real structure on $S_\fs\otimes E$.
  \item
    \label{Prop_HowToDetermineSigma_DT}
    Denote by $\slD_{\Im}$ the restriction of $\slD_{A,\bp}$ to $\Im(S_\fs\otimes E) \subset \Hom(E,W)$ and define the operator $\fd_\Psi\co \Omega^1(M,i\R)\oplus\Omega^0(M,i\R) \to \Omega^1(M,i\R)\oplus\Omega^0(M,i\R)$ by
    \begin{equation*}
      \fd_\Psi
      \coloneq
      \tilde\fa_\Psi^*\circ \slD_{\Im}\circ \tilde\fa_\Psi.
    \end{equation*}
    For each $t \in [0,1]$,
    the operator
    \begin{equation*}
      \fd_\Psi^t \coloneq (1-t)\fd_\Psi + t\fd
    \end{equation*}
    is a self-adjoint and Fredholm.
  \end{enumerate}
\end{prop}

\begin{proof}
  The fact that $\tilde\fa_\Psi$ is an isometry is a consequence of
  \begin{equation}
    \label{Eq_APsiAPsi*}
    \fa_\Psi^*\fa_\Psi = \abs{\Psi}^2
  \end{equation}
  which in turn follows from the following calculation for $(a,\xi)$ and $(b,\eta)$ in $\Omega^1(M,i\R)\oplus\Omega^0(M, i\R)$:
  \begin{align*}
    \inner{\fa_\Psi(a,\xi)}{\fa_\Psi(b,\eta)}
    &=
      \inner{\bar\gamma(a)\Psi + \rho(\xi)\Psi}{\bar\gamma(b)\Psi +  \rho(\eta)\Psi} \\
    &=
    \abs{\Psi}^2(\inner{a}{b} + \inner{\xi}{\eta}).
  \end{align*}
  
  Since
  \begin{equation}
    \label{Eq_EPsi}
    \begin{split}
      \slD_{\Im}\fa_\Psi(a,\xi)
      &=
      \slD_{\Im}(\bar\gamma(a)\Psi + \rho(\xi)\Psi) \\
      &=
      \bar\gamma(*\rd a)\Psi
      + \rho(\rd^*a)\Psi
      + \bar\gamma(\rd\xi)\Psi \\
      &\quad
      - \bar\gamma(a)\slD_{\Re} \Psi
      + \rho(\xi)\slD_{\Re} \Psi 
      - 2\sum_{i=1}^3 \rho(a(e_i))\nabla_{e_i}\Psi \\
      &=
      \fa_\Psi\fd(a,\xi) - 2\sum_{i=1}^3 \rho(a(e_i))\nabla_{e_i}\Psi,
    \end{split}
  \end{equation}
  we have
  \begin{equation*}
    \fd_\Psi
    =
    \fd + \fe_\Psi
  \end{equation*}
  with $\fe_\Psi$ a zeroth order operator depending on $\Psi$ and its derivative.
  This implies that $\fd_\Psi^t$ is a Fredholm operator.
  By construction $\fd_\Psi^t$ is self-adjoint.
\end{proof}

\begin{definition}
  \label{Def_Sigma}
  In the situation of \autoref{Prop_DPsi}, define
  \begin{equation*}
    \sigma(\Psi,\bp)
    \coloneq
      (-1)^{b_1(M)}\cdot (-1)^{\SF\((-\fd_\Psi^t)_{t\in[0,1]}\)}.    
  \end{equation*}
\end{definition}

\begin{remark}
  \label{Rmk_SigmaPsi}
  The operator $\fd_\Psi$ only depends on $\Psi$ up to multiplication by a constant in $\R^*$; hence, the same holds for $\sigma(\Psi,\bp)$.
\end{remark}

\begin{definition}
  \label{Def_RelativeDegree}
  For a pair of nowhere vanishing sections $\Psi, \Phi \in \Gamma(\Re(S_\fs\otimes E))$ we define their \defined{relative degree} $\deg(\Psi,\Phi)\in\Z$ as follows.
  Choose any trivializations of $E$ and $S_\fs$ compatible with the $\SU(2)$ structures. 
  In the induced trivialization of $\Re(S_\fs\otimes E)$ the sections $\Psi/\abs{\Psi}$ and $\Phi/\abs{\Phi}$ are represented by maps $M\to S^3$.
  Set
  \begin{equation*}
    \deg(\Psi,\Phi) \coloneq \deg(\Psi / \abs{\Psi}) - \deg(\Phi / \abs{\Phi}).
  \end{equation*}
  This number does not depend on the choice of the trivializations.
\end{definition}

\pagebreak[2]

\begin{theorem}
  \label{Thm_TwoSeibergWittenWallCrossing}
  Let $(\bp_t,\eta_t)_{t\in[0,1]} \in \bsQ^\reg\paren[\big]{(\bp_0,\eta_0),(\bp_1,\eta_1)}$.
  For each spin structure $\fs$ inducing the spin$^c$ structure $\fw$, denote
  \begin{itemize}
  \item
    by $\set{ t^\fs_1, \ldots, t^\fs_{N_\fs} } \subset [0,1]$ the set of times at which the spectrum of $\slD_{\bp_t}^\fs$ crosses zero%
    \footnote{%
      This set is finite by \autoref{Def_RegularPath}\autoref{Def_RegularPath_TransverseSpectralFlow}.
    }
  \end{itemize}
  and, for each $i = 1, \ldots, N_\fs$, denote
  \begin{itemize}
  \item
    by $\chi^\fs_i \in \set{\pm 1}$ the sign of the spectral crossing of the family $(\slD^\fs_{\bp_t})$ at $t^\fs_i$ and
  \item
    by $\Psi^\fs_i$ a nowhere vanishing spinor spanning $\ker \slD_{\bp_t}^\fs$.
  \end{itemize}
  If there are no singular harmonic $\Z_2$ spinors with respect to $\bp_t$ for any $t \in [0,1]$,
  then
  \begin{equation}
    \label{Eq_WallCrossingFormula}
    n_\fw(\bp_1,\eta_1)
    =
      n_\fw(\bp_0,\eta_0)
    +
      \sum_\fs \sum_{i=1}^{N_\fs} \chi^\fs_i \cdot \sigma(\Psi^\fs_i,\bp_{t^\fs_i})
  \end{equation}
  or, equivalently,
  \begin{equation}
    \label{Eq_WallCrossingFormula_2}
    n_\fw(\bp_1,\eta_1)
    =
      n_\fw(\bp_0,\eta_0)
    +
      \sum_\fs
        \chi^\fs_1 \cdot \sigma(\Psi^\fs_1,\bp_{t^\fs_1}) \cdot
        \sum_{i=1}^{N_\fs}
          (-1)^{i+1}\cdot(-1)^{\deg(\Psi^\fs_1,\Psi^\fs_i)}.
  \end{equation}
  Here the sums are over all spin structures $\fs$ inducing $\fw$.
\end{theorem}

\begin{remark}
  It follows from \autoref{Thm_TwoSeibergWittenWallCrossing}, that $n_\fw(\bp,\eta)$ does not depend on $\eta$.
\end{remark}

The proof of the \eqref{Eq_WallCrossingFormula} proceeds by analyzing the $1$--parameter family of moduli spaces 
\begin{equation*}
  \fW \coloneq \bigcup_{t \in [0,1]} \fM_\fw(\bp_t, \eta_t).
\end{equation*}
By \autoref{Def_RegularPath}\itref{Def_RegularPath_Unobstructed},
$\fW$ is an oriented, one-dimensional manifold with oriented boundary
\begin{equation*}
  \partial\fW = \fM_\fw(\bp_0, \eta_0) \cup -\fM_\fw(\bp_1, \eta_1).
\end{equation*}
If $\fW$ were compact, then it would follow that $n_\fw(\bp_1,\eta_1) = n_\fw(\bp_0,\eta_0)$.
However, $\fW$ may be non-compact.


\section{Compactification of the cobordism}
\label{Sec_SmoothConvergenceIfZEmpty}

Set
\begin{equation*}
  \overline{\fW}
  \coloneq
  \set*{
    (t,\epsilon,[\Psi,A]) \in [0,1]\times[0,\infty)\times\frac{\Gamma(\Hom(E,W)\times\sA(\det W)}{\sG(\det W)} : (*)
  }
\end{equation*}
with $(*)$ meaning that:
\begin{itemize}
\item 
  the differential equation
  \begin{equation}
    \label{Eq_BlownUpPerturbedTwoSeibergWitten}
    \begin{split}
      \slD_{A,\bp_t}\Psi &= 0, \\
      \epsilon^2\paren*{\frac12F_{A}+\eta_{t}} &= \mu_{\bp_{t}}(\Psi), \qand \\
      \Abs{\Psi}_{L^2} &= 1
    \end{split}
  \end{equation}
  holds and
\item
  if $\epsilon = 0$, then $\Psi$ is nowhere vanishing.
\end{itemize}
Equip $\overline\fW$ with the $C^\infty$--topology.
We have a natural embedding $\fW \into \overline \fW$ given by $(t,[\Psi,A]) \mapsto (t,\epsilon,[\tilde\Psi,A])$ with $\epsilon \coloneq 1/\Abs{\Psi}_{L^2}$ and $\tilde\Psi \coloneq \Psi/\Abs{\Psi}_{L^2}$.

\begin{prop}
  \label{Prop_DenseInCompactifiedCobordism}
  $\fW$ is dense in $\overline\fW$.
\end{prop}

\begin{proof}
  If $(t_0,\epsilon,[\Psi,A]) \in \overline\fW\setminus\fW$, then $\epsilon = 0$.
  It follows from \autoref{Def_RegularPath} and \cite[Theorem 1.38]{Doan2017a}, that there are is a family $\set*{ \(\Psi_{\epsilon}, A_{\epsilon};t(\epsilon)\) : 0 \leq \epsilon  \ll 1 }$ of solutions to
  \begin{align*}
    \slD_{A_{\epsilon},\bp_{t(\epsilon)}}\Psi_{\epsilon} &= 0 \qand \\
    \epsilon^2\paren*{\frac12F_{A_{\epsilon}}+\eta_{t(\epsilon)}} &= \mu_{\bp_{t(\epsilon)}}(\Psi_{\epsilon})
  \end{align*}
  with $(\Psi_\epsilon,A_\epsilon)$ converging to $(\Psi,A)$ in $C^\infty$ and $t(\epsilon)$ converging to $t_0$ as $\epsilon$ tends to zero.
  Consequently, $\fW$ is dense in $\overline\fW$.
\end{proof}

That $\overline\fW$ is compact does not follows from \autoref{Thm_HW};
it does, however, follow from the next result, whose proof will occupy the remainder of this section.

\begin{prop}
  \label{Prop_SmoothConvergenceIfZEmpty}
  Let $(\bp_i, \eta_i)$ be a sequence in $\sP\times \sZ$ which converges to $(\bp,\eta)$ in $C^\infty$.
  Let $(\epsilon_i,\Psi_i,A_i)$ be a sequence of solutions of
  \begin{equation}
    \label{Eq_BlownUpSeibergWitten}
    \begin{split}
      \slD_{A_i,\bp_i}\Psi_i &= 0, \\
      \epsilon^2\paren*{\frac12F_{A_i}+\eta_i} &= \mu_{\bp_i}(\Psi_i), \qand \\
      \Abs{\Psi_i}_{L^2} &= 1
    \end{split}
  \end{equation}
  with $\lim_{i\to \infty} \epsilon_i = 0$.
  If the set
  \begin{equation*}
    Z \coloneq \set*{ x \in M : \limsup_{i \to \infty} \abs{\Psi_i(x)} = 0 }
  \end{equation*}
  is empty, then, after passing to a subsequence and applying gauge transformations, $(\Psi_i,A_i)$ converges in $C^\infty$ to a solution $(\Psi,A) \in \Gamma(\Hom(E,W))\times\sA(\det W)$ of
  \begin{equation*}
    \slD_{A,\bp}\Psi = 0, \quad 
   \mu_{\bp}(\Psi) = 0, \qandq
    \Abs{\Psi}_{L^2} = 1.
  \end{equation*}
\end{prop}

\begin{prop}
  \label{Prop_CompactifiedCobordismIsCompact}
  $\overline\fW$ is compact.
\end{prop}

\begin{proof}
  We need to show that any sequence $(t_i,\epsilon_i,[\Psi_i,A_i])$ in $\fW$ has a subsequence which converges in $\overline\fW$.
  If $\liminf \epsilon_i > 0$, then this a consequence of standard elliptic estimates and Arzelà--Ascoli.
  It only needs to be pointed out that $\epsilon_i$ cannot tend to infinity, because otherwise there would be a reducible solution to \eqref{Eq_PerturbedTwoSeibergWitten} which is ruled out by \autoref{Def_RegularPath}\itref{Def_RegularPath_Unobstructed}.
  
  If $\liminf_{i\to\infty} \epsilon_i = 0$, then \autoref{Thm_HW} asserts that a gauge transformed subsequence of $(\Psi_i,A_i)$ converges weakly in $W^{2,2}_\loc\times W^{1,2}_\loc$ outside $Z$.
  If $Z$ is non-empty, then the limit represents a singular harmonic $\Z_2$ spinors.
  However, by assumption there are no singular harmonic $\Z_2$ spinors;
  hence, $Z$ is empty and \autoref{Prop_SmoothConvergenceIfZEmpty} asserts that a gauge transformed subsequence of $(\Psi_i,A_i)$ converges in $C^\infty$.
\end{proof}

The proof of \autoref{Prop_SmoothConvergenceIfZEmpty} relies on the following a priori estimate.

\begin{prop}
  \label{Prop_FromL2BoundsToCkBounds}
  For each $m_0 > 0$, there is an $\epsilon_0 > 0$ and, for each, $k \in \N$, a non-decreasing function $f_k \co [0,\infty) \to [0,\infty)$ such that the following holds:
  if $(\Psi,A,\epsilon) \in \Gamma(M,\Hom(E,W)) \times \sA(\det W) \times (0,\epsilon_0]$ satisfies \eqref{Eq_BlownUpSeibergWitten} and 
  \begin{equation*}
    \min \abs{\Psi}^2 \geq m_0,
  \end{equation*}
  then
  \begin{equation*}
    \Abs{\Psi}_{C^k_A} + \Abs{F_A}_{C^k} \leq f_k\(\Abs{F_A}_{L^2} + \Abs{\eta}_{C^k} + \Abs{F_B}_{C^k} + \Abs{R_g}_{C^k}\).
  \end{equation*}
  Here $\Abs{\Psi}_{C^k_A} \coloneq \sum_{i = 0}^k \Abs{\nabla_A^i \Psi}_{L^{\infty}}$ and $R_g$ denotes the Riemann curvature tensor of the metric $g$.
\end{prop}

The key ingredient for the proof of the above a priori estimate are the identity
\begin{equation}
  \label{Eq_CuvatureHiggsEffect} 
  \begin{split}
    (\rd^*\rd+\rd\rd^*)\mu(\Psi)+ \abs{\Phi}^2F_A
    &=
      \sum_{i,j=1}^3
      \frac12\rho^*\paren*{\paren[\big]{(F^B_{ij}+F^\fw_{ij}) \cdot \Phi}\Phi^*} e^{ij} \\
    &\qquad\qquad
      + \rho^*\paren*{(\nabla^A_j\Phi)(\nabla^A_i\Phi)^*} e^{ij},
  \end{split}
\end{equation}
whose proof can be found in \cite[Proposition A.7]{Doan2017a}, and the following two propositions.

\begin{prop}
  \label{Prop_DiracEstimate}
  Let $k \in \N_0, p \geq 2$.
  There is a monotone function $f_{k,p} \co [0,\infty) \to [0,\infty)$ such that for every $(\Psi,A) \in \Gamma(M,\Hom(E,W))\times\sA(\det W)$ solving
  \begin{equation*}
    \slD_{A,\bp}\Psi = 0
  \end{equation*}
  we have
  \begin{equation*}
    \Abs{\Psi}_{W^{k+2,p}_A}
    \lesssim
      f_{k,p}\(\Abs{F_A}_{W^{k,p}} + \Abs{\Psi}_{L^\infty} + \Abs{F_B}_{C^{k+1}} + \Abs{R_g}_{C^{k+1}}\).
  \end{equation*}
  Here $\Abs{\Psi}_{W^{k,p}_A} \coloneq \sum_{i=0}^{k} \Abs{\nabla_A^i \Psi}_{L^p}$.
\end{prop}

\begin{proof}
  Fix a smooth reference connection $A_0$.
  By Hodge theory, after applying a gauge transformation, we can write $A = A_0 + a$ with $\rd^*a = 0$ and
  \begin{equation*}
    \Abs{a}_{W^{k+1,p}}
    \lesssim
    \Abs{F_A}_{W^{k,p}}
  \end{equation*}
  with the implicit constant depending on $k$, $p$, and $\Abs{R_g}_{C^k}$.
  It suffices to estimate $\Abs{\Psi}_{W^{k+2,p}_{A_0}}$,
  since
  \begin{align*}
    \Abs{(\nabla_{A_0}+a)^2\Phi}_{L^p}
    &\lesssim
      \Abs{\nabla_{A_0}^2\Phi}_{L^p}
      + \Abs{\abs{a}\abs{\nabla_{A_0}\Phi}}_{L^p} 
      + \Abs{\abs{a}^2\abs{\Phi}}_{L^p}
      + \Abs{\abs{\nabla a}\abs{\Phi}}_{L^p} \\
    &\lesssim
      \Abs{\nabla_{A_0}^2\Phi}_{L^p}
      + \Abs{a}_{L^{2p}}\Abs{\nabla_{A_0}\Phi}_{L^{2p}} 
      + \Abs{a}_{L^{2p}}^2\Abs{\Phi}_{L^\infty}
      + \Abs{\nabla a}_{L^p}\Abs{\Phi}_{L^\infty} \\
    &\lesssim
      \(1+\Abs{a}_{W^{1,p}}\)^2\cdot\Abs{\Psi}_{W^{2,p}_{A_0}}.
  \end{align*}
  by Hölder's inequality, Sobolev embedding, and Morrey's inequality;
  more generally,
  \begin{equation*}
    \Abs{\Psi}_{W^{k+2,p}_A}
    \lesssim
      \(1+\Abs{a}_{W^{k+1,p}}\)^{k+2}\cdot\Abs{\Psi}_{W^{k+2,p}_{A_0}}.
  \end{equation*}

  The Dirac equation can be written as
  \begin{equation*}
    \slD_{A_0,\bp} \Psi = -\bar\gamma(a)\Psi.
  \end{equation*}
  By standard elliptic estimates
  \begin{equation}
    \label{Eq_DiracEllipticEstimate}
    \Abs{\Psi}_{W^{\ell+1,q}_{A_0}}
    \lesssim
      \Abs{\bar\gamma(a)\Psi}_{W^{\ell,q}_{A_0}}
      + \Abs{\Psi}_{L^\infty}
  \end{equation}
  for any $\ell \in \N_0$ and $q \in (1,\infty)$ with the implicit constant depending on $\ell$, $q$, $\Abs{F_B}_{C^\ell}$, and $\Abs{R_g}_{C^\ell}$.
  By the Sobolev Multiplication Theorem,%
  \footnote{%
    The Sobolev Multiplication Theorem asserts that $\Abs{fg}_{W^{k,p}} \lesssim \Abs{f}_{W^{k_1,p_1}}\Abs{g}_{W^{k_2,p_2}}$ provided $k-\frac{3}{p} < k_1-\frac{3}{p_1} + k_2-\frac{3}{p_2}$ and $k \leq \min\set{k_1,k_2}$.
  }
  we have
  \begin{equation*}
    \Abs{\bar\gamma(a)\Psi}_{W^{k+1,p}_{A_0}}
    \lesssim
    \Abs{a}_{W^{k+1,p}}
    \Abs{\Psi}_{W^{k+1,q}_{A_0}}
  \end{equation*}
  provided $(k+1)q > 3$.
  Therefore, to bound $\Abs{\Psi}_{W^{k+2,p}_{A_0}}$ in terms of
  \begin{equation*}
    \spadesuit_k \coloneq \Abs{F_A}_{W^{k,p}} + \Abs{\Psi}_{L^\infty} + \Abs{F_B}_{C^k} + \Abs{R_g}_{C^k},
  \end{equation*}
  it suffices to show that for some $q$ satisfying $(k+1)q>3$, the norm $\Abs{\Psi}_{W^{k+1,q}_{A_0}}$ can be bounded in terms of $\spadesuit_k$.
  If $k \geq 1$ or $k = 0$ and $q \geq 3$, such an estimate can be assumed by induction.
  Hence, we only need to prove that $\Abs{\Psi}_{W^{1,q}_{A_0}}$ is bounded in terms of $\spadesuit_k$ for some $q > 3$.

  By \eqref{Eq_DiracEllipticEstimate}, $\Abs{\Psi}_{W^{1,2}_{A_0}}$ can be bounded in terms of $\spadesuit_k$.
  Consequently, we have a bound on $\Abs{\gamma(a)\Psi}_{L^3(B_r)}$ and thus, using \eqref{Eq_DiracEllipticEstimate} again, a bound on $\Abs{\Psi}_{W^{1,3}_{A_0}}$ in terms of $\spadesuit_k$.
  Since $W^{1,3} \into L^p$ for any $p \in [1,\infty)$, it follows that $\Abs{\gamma(a)\Psi}_{L^q} \leq \Abs{a}_{L^6}\Abs{\Psi}_{L^p}$ is bounded for any $q$ with $1/q = 1/6 + 1/p$.
  This then implies the desired bound on $\Abs{\Psi}_{W^{1,q}_{A_0}}$ for any $q \in [1,6)$.
\end{proof}

\begin{prop}
  \label{Prop_EpsilonDelta+1}
  Let $k \in \N, p \in [2,\infty)$.
  Let $m \in C^\infty(M)$ with $m_0 \coloneq \min m > 0$.
  There are constants $c > 0$ and $\epsilon_0 > 0$ depending on $k$, $p$, $m_0$, $\Abs{m}_{C^k}$, and $\Abs{R_g}_{C^k}$ such that for all $\epsilon \in (0,\epsilon_0)$
  \begin{equation*}
    \Abs{\alpha}_{W^{k,p}}
    \leq
      c \Abs{(\epsilon^2\Delta + m)\alpha}_{W^{k,p}}
  \end{equation*}
  for all $\alpha \in \Omega^*(M)$. 
  We use here the positive-definite Hodge laplacian $\Delta = \rd \rd^* + \rd^* \rd$.
\end{prop}

\begin{proof}
  Denote by $\fR$ the curvature operator appearing in the Weitzenböck formula $\Delta \alpha= \nabla^*\nabla \alpha + \fR \alpha$.
  If $\epsilon \ll 1$, then by integration by parts we obtain
  \begin{align*}
    \int_M
      \inner{(\epsilon^2\Delta + m)\alpha}{\alpha} \abs{\alpha}^{p-2}
    &=
      \int_M
        \epsilon^2 \abs{\nabla a}^2\abs{a}^{p-2}
        + \frac{\epsilon^2(p-2)}{4} \abs*{\nabla_A \abs{a}^2}^2\abs{a}^{p-4} \\
    &\qquad
        + \epsilon^2 \inner{\fR\alpha}{\alpha}\abs{\alpha}^{p-2}
        + m \abs{\alpha}^p \\
    &\geq
      \int_M \frac{m_0}{2}\abs{\alpha}^p.
  \end{align*}
  By Young's inequality and with $q = p/(p-1)$, we have
  \begin{align*}
    \int_M
      \inner{(\epsilon^2\Delta + m)\alpha}{\alpha} \abs{\alpha}^{p-2}
    &\leq
      \int_M
        \abs{(\epsilon^2\Delta + m)\alpha}\abs{\alpha}^{p-1} \\
    &\leq
      \int \frac{\delta^{-p}}{p} \abs{(\epsilon^2\Delta + m)\alpha}^p + \frac{\delta^q}{q}\abs{\alpha}^p
  \end{align*}
  for all $\delta > 0$.
  Choosing $\delta$ sufficient small, yields
  \begin{equation*}
    \int_M \abs{\alpha}^p
    \lesssim
      \int_M \abs{(\Delta+m)\alpha}^p;
  \end{equation*}
  that is, the desired estimate for $k = 0$.

  If $v_1,\ldots,v_k$ are vector fields on $M$, then the commutator $[\epsilon^2\Delta+m,\del_{v_1}\cdots\del_{v_k}]$ is a differential operator of order $k-1$ with coefficients depending on only on the first $k$ derivatives of $m$ and $R_g$.
  Therefore, the estimates for $k > 0$ follows from the one for $k = 0$ using
  \begin{equation*}
    (\epsilon^2\Delta+m)(\del_{v_1}\cdots\del_{v_k} \alpha)
    = \del_{v_1}\cdots\del_{v_k} (\epsilon^2\Delta+m)\alpha
      + [\epsilon^2\Delta+m,\del_{v_1}\cdots\del_{v_k}]\alpha
  \end{equation*}
  and induction.
\end{proof}

\begin{proof}[Proof of \autoref{Prop_FromL2BoundsToCkBounds}]
  It follows from \cite[Proposition 2.1]{Haydys2014} that $\Abs{\Psi}_{L^\infty}$ can be bounded in terms of $\clubsuit_k \coloneq \Abs{F_A}_{L^2} + \Abs{\eta}_{C^k} + \Abs{F_B}_{C^{k+1}} + \Abs{R_g}_{C^{k+1}}$.
  From \autoref{Prop_DiracEstimate} with $k=0$ and $p=2$ it follows that $\Abs{\Psi}_{W^{2,2}_A}$ bounded in terms of $\clubsuit_k$ as well.
  Since $W^{2,2}_A \into W^{1,6}_A$, it follows that the $L^3$--norm of the right-hand side of \eqref{Eq_CuvatureHiggsEffect} can be bounded in terms of $\clubsuit_k$.
  \autoref{Prop_EpsilonDelta+1} thus yields a bound on $\Abs{F_A}_{L^3}$ in terms of $\clubsuit_k$.
  \autoref{Prop_DiracEstimate} with $k=0$ and $p=3$ bounds $\Abs{\Psi}_{W^{2,3}_A}$ and, hence, $\Abs{\Psi}_{W^{1,p}_A}$ for all $p \in (1,\infty)$ in terms of $\clubsuit_k$.
  Consequently, \autoref{Prop_EpsilonDelta+1} bounds $\Abs{F_A}_{L^p}$ for any $p$ in terms of $\clubsuit_k$.
  Another application of \autoref{Prop_DiracEstimate} shows that $\Abs{\Psi}_{W^{2,p}_A}$ can be bounded in terms of $\clubsuit_k$.
  This yields bounds on the $W^{1,p}$--norm of right-hand side of \eqref{Eq_CuvatureHiggsEffect} and $\Abs{\abs{\Psi}}_{C^1}$ in terms of $\clubsuit_k$.
  It follows from \autoref{Prop_EpsilonDelta+1} that $\Abs{F_A}_{W^{1,p}}$ is bounded in terms of $\clubsuit_k$.
  Iterating applications of \autoref{Prop_DiracEstimate} and \autoref{Prop_EpsilonDelta+1} one shows that $\Abs{F_A}_{W^{k+1,p}}$ and $\Abs{\Psi}_{W^{k+3,p}}$ can be bounded in terms of $\clubsuit_k$.
  Applying Morrey's inequality, $\abs{f}_{C^0} \lesssim \Abs{f}_{W^{1,p}}$ for $p > 3$, completes the proof.
\end{proof}

\begin{proof}[Proof of \autoref{Prop_SmoothConvergenceIfZEmpty}]
  Since $Z$ is empty, there exists $m_0 > 0$ such that, after passing to a subsequence, we have $\min \abs{\Psi_i}^2 \geq m_0$ for all $i$.
  Moreover, it follows from \cite[Definition 3.1 and Proposition 4.1]{Haydys2014} that $\Abs{F_{A_i}}_{L^2}$ is uniformly bounded.
  \autoref{Prop_FromL2BoundsToCkBounds} now yields uniform $C^k$--bounds for $F_{A_i}$ and  and $\Psi_i$ (using the $A_i$--dependent norm).
  After putting $A_i$ in the Uhlenbeck gauge as in the proof of \autoref{Prop_DiracEstimate} we also obtain $C^k$ bounds for $A_i$ and the result follows from the Arzelà--Ascoli theorem.
\end{proof}


\section{Orientation at infinity}
\label{Sec_OrientationAtInfinity}

Suppose that $(t_0,0,[\Psi_0,A_0]) \in \overline\fW\setminus\fW$ is a boundary point in $\overline\fW$.
By \autoref{Prop_SpecialHaydysCorrespondence}, there exists a spin structure $\fs$ inducing the spin$^c$ structure $\fw$ such that $\Psi_0 \in \Gamma(\Re(S_\fs\otimes E)) \subset \Gamma(\Hom(E,W))$,
$\slD^\fs_{\bp_{t_0}}\Psi_0 = 0$,
and $A_0$ is trivial.
By \autoref{Def_RegularPath}\itref{Def_RegularPath_TransverseSpectralFlow},
there exists a unique solution $\set{ \Psi_t : \abs{t-t_0} \ll 1}$ to
\begin{equation*}
  \slD^\fs_{\bp_t}\Psi_t = \lambda(t)\Psi_t \qandq \Abs{\Psi_t}_{L^2} = 1
\end{equation*}
with $\Psi_{t_0} = \Psi_0$.
Moreover, $\lambda$ is a differentiable function of $t$ near $t_0$ whose derivative, to be denoted by $\dot{\lambda}$,  satisfies $\dot\lambda(t_0) \neq 0$.
In this situation,
\cite[Section 6]{Doan2017a} shows that for any choice of $r\in\N$
there exist $\tau \ll 1$ and $\epsilon_0 \ll 1$,
a $C^r$ map
\begin{equation*}
  \ob\co (t_0-\tau,t_0+\tau) \times [0,\epsilon_0) \to \R,
\end{equation*}
an open neighborhood $V$ of $(t_0,0,[\Psi_0,A_0]) \in \overline\fW$, and
a homeomorphism
\begin{equation*}
  \fx \co \ob^{-1}(0) \to V
\end{equation*}
such that:
\begin{enumerate}
\item
  $\fx$ commutes with the projection to the $t$-- and $\epsilon$--coordinates.
\item
  The restriction ${\ob}|_{(t_0-\tau,t_0+\tau) \times (0,\epsilon_0)}$ is smooth and satisfies
  \begin{equation}
    \label{Eq_ObstructionMapExpansion}
    \begin{split}
    \ob(t,\epsilon)
    &=
      \dot\lambda(t_0)\cdot(t-t_0) - \delta \epsilon^4 + O\((t-t_0)^2,\epsilon^6\), \qand \\
    \del_\epsilon \ob(t,\epsilon)
    &=
      - 4\delta \epsilon^3 + O\((t-t_0)^2,\epsilon^5\),
   \end{split}
  \end{equation}
  with $\delta=\delta(\Psi_0,\bp_{t_0},\eta_{t_0})$ as in \autoref{Def_RegularPath}\itref{Def_RegularPath_DeltaNonZero}.
  In particular, since by assumption $\delta \neq 0$, the equation $\ob(t, \epsilon) = 0$ can be solved for $t$:
  \begin{equation}  
    \label{Eq_FormulaForT}
    t(\epsilon)  = t_0 + \frac{\delta}{\dot\lambda(t_0)}\epsilon^4 + O(\epsilon^6).
  \end{equation}
\item
  If $(t,\epsilon) \in (t_0-\tau,t_0+\tau)\times (0,\epsilon_0]$ satisfies $\ob(t,\epsilon) = 0$ and
  \begin{equation*}
    \fx(t, \epsilon) = (t,\epsilon,[\Psi_\epsilon,A_\epsilon])
  \end{equation*}
  then $[\epsilon^{-1}\Psi_\epsilon, A_\epsilon]$ is a solution of the Seiberg--Witten equation \autoref{Eq_PerturbedTwoSeibergWitten}.
  We will prove below that this solution is unobstructed; that is, the operator
  \begin{equation*}
    L_{\epsilon^{-1}\Psi_\epsilon, A_\epsilon, \bp_{t(\epsilon)}}
  \end{equation*}
  is invertible.
\end{enumerate}

It follows from the above that
\begin{equation*}
  \ob^{-1}(0)
  =
  \set* {
    \paren[\big]{t(\epsilon), \epsilon} 
    :
    \epsilon \in [0,\epsilon_0)
  }.
\end{equation*}
Therefore, $\overline\fW$ is a compact, oriented, one--dimensional manifold with boundary, that is: a finite collection of circles and closed intervals. 
Its oriented boundary is
\begin{equation*}
  \del\overline\fW
  =
  \fM_\fw(\bp_1,\eta_1)
  -
  \fM_\fw(\bp_0,\eta_0)
  \cup
  \(\overline\fW\setminus\fW\).
\end{equation*}
It follows from \eqref{Eq_FormulaForT} that, if $\delta/\dot\lambda(0) > 0$, then as $t$ passes through $t_0$ a solution to \autoref{Eq_PerturbedTwoSeibergWitten} is created.
If $\delta/\dot\lambda(0) < 0$, then as $t$ passes through $t_0$ a solution to \autoref{Eq_PerturbedTwoSeibergWitten} is annihilated.
The solution that is created/annihilated at $t_0$ contributes $\sign[\epsilon^{-1}\Psi_{\epsilon},A_\epsilon]$ to the signed count of solutions.
Consequently, for $\tau \ll 1$, we have the local wall-crossing formula
\begin{equation}
  \label{Eq_LocalWallCrossingFormula}
  n_\fw(\bp_{t+\tau},\eta_{t+\tau})
  =
    n_\fw(\bp_{t_0-\tau},\eta_{t_0-\tau})
    + \sign(\dot\lambda(t_0)) \sign(\delta) \sign[\epsilon^{-1}\Psi_{\epsilon},A_\epsilon].
  \end{equation}
  
\begin{prop}
  \label{Prop_HowToDetermineSigma}
  In the above situation, for $\epsilon \ll 1$, the solution $[\epsilon^{-1}\Psi_\epsilon,A_\epsilon]$ is unobstructed and
  \begin{equation}
    \label{Eq_HowToDetermineSigma}
    \sign(\delta) \sign[\epsilon^{-1}\Psi_{\epsilon},A_\epsilon]
    =
    \sigma(\Psi_0,\bp_{t_0})
  \end{equation}    
  with $\sigma(\Psi_0,\bp_{t_0})$ as in \autoref{Def_Sigma}.
  In particular, the local wall-crossing formula \autoref{Eq_LocalWallCrossingFormula} can be written in the form
  \begin{equation}
  \label{Eq_LocalWallCrossingFormula2}
  n_\fw(\bp_{t_0+\tau},\eta_{t_0+\tau})
  =
    n_\fw(\bp_{t_0-\tau},\eta_{t_0-\tau})
    + \sign(\dot\lambda(0))\, \sigma(\Psi_0,\bp_{t_0}).
\end{equation}
\end{prop}

\begin{proof}[Proof of \autoref{Prop_HowToDetermineSigma}]
  The reader might find it helpful to review
  the orientation procedure described in \autoref{Def_Sign}, as well as
  \autoref{Prop_DPsi} and \autoref{Def_Sigma} before reading further.
  A reader who is unfamiliar with the determinant line bundle and its relation to the orientation procedure should consult \autoref{Sec_DeterminantLineBundles}, where this relation is discussed.  
  The proof proceeds in five steps.

  \setcounter{step}{0}
  \begin{step}
    For any $s > 0$, we have
    \begin{multline*}
      \sign[\epsilon^{-1}\Psi_\epsilon, A_\epsilon] \\
      = \OT((0,A_0,\bp_{t_0}),(s\Psi_0,A_0,\bp_{t_0}))\cdot\OT((s\Psi_0,A_0,\bp_{t_0}),(\epsilon^{-1}\Psi_\epsilon,A_\epsilon,\bp_{t(\epsilon)})).
    \end{multline*}
  \end{step}

  This is an immediate consequence of \autoref{Def_Sign} as well as the multiplicative property of the orientation transport \autoref{Eq_OrientationTransportComposition}.

  \begin{step}
    For $s \gg 1$, we have
    \begin{equation*}      
      \OT((0,A_0,\bp_{t_0}),(s\Psi_0,A_0,\bp_{t_0})) = -\sigma(\Psi_0, \bp_{t_0}).
    \end{equation*}
  \end{step}

  Denote by $\slD_{\Re}$ and $\slD_{\Im}$ the restriction of $\slD_{A_0,\bp_{t_0}}$ to $\Re(S_\fs\otimes E)$ and $\Im(S_\fs\otimes E)$ respectively.
  We can write $L_{s\Psi_0,A_0,\bp_{t_0}}$ as
  \begin{equation*}
    L_{s\Psi_0,A_0,\bp_{t_0}}
    =
    \begin{pmatrix}
      -\slD_{\Re} & \\
      & \fD_s
    \end{pmatrix}
    \quad\text{with}\quad
    \fD_s
    =
    \begin{pmatrix}
      -\slD_{\Im} & - s\fa_{\Psi_0} \\
      - s\fa_{\Psi_0}^* & \fd
    \end{pmatrix}.
  \end{equation*}
  Consequently,
  we have
  \begin{equation*}
    \OT((0,A_0,\bp_{t_0}),(s\Psi_0,A_0,\bp_{t_0}))
    = (-1)^{\SF\((\fD_\sigma)_{\sigma\in[0,s]}\)}.
  \end{equation*}

  Denote by $\tilde\fa_{\Psi_0} \coloneq \abs{\Psi_0}^{-1}\fa_{\Psi_0}$ the isometry introduced in \autoref{Prop_DPsi}.
  Define
  \begin{equation*}
    \tilde \fD_s \co (\Omega^1(M,i\R) \oplus \Omega^0(M,i\R))^{\oplus 2} \to (\Omega^1(M,i\R) \oplus \Omega^0(M,i\R))^{\oplus 2}    
  \end{equation*}
  by
  \begin{equation*}
    \tilde \fD_s \coloneq (\tilde\fa_{\Psi_0}^* \oplus \id) \circ \fD_s \circ (\tilde\fa_{\Psi_0} \oplus \id).
  \end{equation*}
  Since $\fd_{\Psi_0} = \tilde\fa_{\Psi_0}^*\circ \slD_{\Im}\circ \tilde\fa_{\Psi_0}$, by definition, and
  $\fa_{\Psi_0}^*\fa_{\Psi_0} = \abs{\Psi}^2$,
  we have
  \begin{equation*}
    \tilde \fD_s
    =
    \begin{pmatrix}
       -\fd_{\Psi_0} & -s\abs{\Psi_0} \\
        -s\abs{\Psi_0} & \fd
    \end{pmatrix}.
  \end{equation*}
  From this discussion, and the additivity of the spectral flow under path composition, it follows that
  \begin{align*}
    \SF\((\fD_{\sigma})_{\sigma\in[0,s]}\) 
    &=
    \SF\((\tilde \fD_{\sigma})_{\sigma\in[0,s]}\) \\
    &= 
      \SF\(
      \begin{pmatrix}
        (t-1)\fd_{\Psi_0}-t\fd & 0\\
        0 & \fd
      \end{pmatrix}_{t\in[0,1]}\) \\
    &\quad
      +
      \SF\(
      \begin{pmatrix}
        -\fd & -ts \\
        -ts & \fd
      \end{pmatrix}_{t\in[0,1]}\) \\
    &\quad
      +
      \SF\(
      \begin{pmatrix}
        (t-1)\fd-t\fd_{\Psi_0} & -((1-t)+t\abs{\Psi_0})s \\
        -((1-t)+t\abs{\Psi_0})s & \fd
      \end{pmatrix}_{t\in[0,1]}\).
  \end{align*}
  
  The first spectral flow is given by
  \begin{equation*}
    \SF\((-\fd_{\Psi_0}^t)_{t\in[0,1]}\)
  \end{equation*}
  in the notation of \autoref{Prop_DPsi}.
  The second spectral flow can be computed to be $-1-b_1(M)$ by observing that if $\sH \coloneq \sH^0\oplus\sH^1$ denotes the space of harmonic forms and $\fd_0 \co \sH^\perp \to \sH^\perp$ is the restriction of $\fd$, then
  \begin{equation*}
    \begin{pmatrix}
      -\fd & -s  \\
      -s & \fd
    \end{pmatrix}
  \end{equation*}
  can be diagonalized to
  \begin{equation*}
    \begin{pmatrix}
      s\cdot \id_\sH & & & \\
      & -s\cdot \id_\sH & & \\
      & & +\sqrt{\fd_0^2 + s^2} & \\
      & & & -\sqrt{\fd_0^2 + s^2}
    \end{pmatrix},
    \footnote{%
      The operator $\fd_0^2 + s^2$ has strictly positive spectrum, so its square root is well-defined and invertible.
    }
  \end{equation*}
  and the spectral flow of this family of operators is $-(1+b_1(M))$, using \autoref{Conv_SpectralFlow}.
  To compute the third spectral flow,
  observe that
  \begin{align*}
    &\begin{pmatrix}
      -\fd - t\fe_{\Psi_0} & -((1-t)+t\abs{\Psi_0})s  \\
      -((1-t)+t\abs{\Psi_0})s & \fd
   \end{pmatrix}^2 \\
   &
   \quad=
    \begin{pmatrix}
      (-\fd - t\fe_{\Psi_0})^2 + ((1-t)+t\abs{\Psi_0})^2s^2 & \\
      & \fd^2 + ((1-t)+t\abs{\Psi_0})^2s^2
    \end{pmatrix}
    +
    s \tilde\fe_{\Psi_0,t}
  \end{align*}
  with $\tilde\fe_{\Psi_0,t}$ a zeroth order operator depending on $\Psi_0$, its derivative, and $t$.
  If $s \gg 1$, then the operator on the left-hand side is invertible for every $t \in [0,1]$.
  Therefore, the third spectral flow is trivial.
  
  This shows that, for $s \gg 1$,
  \begin{equation*}
    \OT((0,A_0,\bp_{t_0}),(s\Psi_0,A_0,\bp_{t_0}))
    = (-1)^{1+b_1(M)}\cdot (-1)^{\SF\((-\fd_{\Psi_0}^t)_{t\in[0,1]}\)}
    = -\sigma(\Psi_0,\bp_{t_0}).
  \end{equation*}  

  \begin{step}
    \label{Step_OperatorAtInfinity}
    If $s \gg 1$,
    then
    \begin{equation*}
      \ker L_{s\Psi_0,A_0,\bp_{t_0}} = \R\Span{\Psi_0}.
    \end{equation*}
  \end{step}

  As above, write
  \begin{equation*}
    L_{s\Psi_0,A_0,\bp_{t_0}}
    =
    \begin{pmatrix}
      -\slD_{\Re} & \\
      & \fD_s
    \end{pmatrix}
    \quad\text{with}\quad
    \fD_s
    =
    \begin{pmatrix}
      -\slD_{\Im} & - s\fa_{\Psi_0} \\
      - s\fa_{\Psi_0}^* & \fd
    \end{pmatrix}.
  \end{equation*}
  By \autoref{Eq_APsiAPsi*} and \autoref{Eq_EPsi},
  we have
  \begin{equation*}
    \fD_s^2
    =
    \begin{pmatrix}
      \slD_{\Im}^*\slD_{\Im} + s^2\abs{\Psi_0}^2 & \\
      & \fd^*\fd + s^2\abs{\Psi_0}^2
    \end{pmatrix}
    +
    s\fe_{\Psi_0}
  \end{equation*}
  with $\fe_{\Psi_0}$ a zeroth order operator depending on $\Psi_0$ and its derivative.
  It follows that $\fD_s$ is invertible for $s \gg 1$.
  Therefore, $\ker L_{s\Psi_0,A_0,\bp_{t_0}} = \ker \slD_{\Re}$.
  By hypothesis the latter is spanned by $\Psi_0$.
  
  \begin{step}
    Comparison of the determinant line bundles of Seiberg--Witten equation with two spinors and the blown-up Seiberg--Witten equation with two spinors.
  \end{step}

  In \cite{Doan2017a} we construct a Kuranishi model for the blown-up Seiberg--Witten equation with two spinors \autoref{Eq_BlownUpPerturbedTwoSeibergWitten}.
  This Kuranishi model allows us to compute orientation transports for the determinant line bundle associated with the blown-up Seiberg--Witten equation with two spinors.
  Although the Seiberg--Witten equation with two spinors and its blown-up version are essentially equivalent,
  the relation between their linearizations is subtle.
  The purpose of this step is to discuss what the precise relation between the determinant line bundles of two linearizations is.

  Consider the space of irreducible configurations
  \begin{equation*}
    \sC^* \coloneq (\Gamma(\Hom(E,W))\setminus\set{0})\times \sA(\det W)
  \end{equation*}
  and its quotient
  \begin{equation*}
    \sB^* \coloneq \sC^* / \sG.
  \end{equation*}
  Let $T\sC^*$ be its tangent bundle---a trivial bundle with fiber $\Gamma(\Hom(E,W)) \oplus \Omega^1(i\R)$. 
  The Seiberg--Witten equation with two spinors (with respect to the parameter $\bp_t$) defines a smooth section $\sw_t\co \sC^*\to  T\sC$:
  \begin{equation*}
    \sw_t(\Psi,A)
    \coloneq
    \begin{pmatrix}
      -\slD_{A,\bp_t}\Psi \\
      *\(F_A + \eta_t - \mu_{\bp_t}(\Psi)\)
    \end{pmatrix}.
  \end{equation*}
  This section is $\sG$--equivariant and, hence, induces a Fredholm section%
  \footnote{%
    Of course, one needs to work with appropriate Sobolev completions of $\sC^*$, $\sB^*$, $T\sC^*$, and $T\sB^*$.
  }
  \begin{equation*}
    \sw_t \co \sB^* \to T\sB^*.
  \end{equation*}
  The moduli space $\fM_\fw(\bp_t,\eta_t)$ of solutions of the Seiberg--Witten equation with two spinors is the zero set $\sw_t^{-1}(0)$.

  A choice of a connection on $\sC^* \to \sB^*$ induces a connection, denoted by $\nabla$, on $T\sB^*$.
  Any choice of gauge fixing induces such a connection.
  The family of Fredholm operators $(\nabla\sw)_\fc \co T_\fc\sB^* \to T_\fc\sB^*$ parametrized by points $\fc \in \sB^*$ define a determinant line bundle
  \begin{equation*}
    \det(\nabla \sw) \to \sB^*;
  \end{equation*}
  see \autoref{Sec_DeterminantLineBundles} for a general discussion of determinant line bundles, and \cite[Section 2.6]{Doan2017} for the construction in the case at hand.  
  This is the determinant line bundle which controls the relative sign two solutions of the Seiberg--Witten equation with two spinors.
  The space of connections on $\sC^* \to \sB^*$ is contractible.
  Thus any two choices of such connections lead to determinant line bundles which are canonically isomorphic.
  Consequently, the choice of gauge fixing does not matter for the purpose of orienting the moduli space.

  Let
  \begin{equation*}
    \sS \coloneq \set{ \Psi \in \Gamma(\Hom(E,W)) : \Abs{\Psi}_{L^2} = 1 }
  \end{equation*}
  be the unit sphere in the space of spinors.
  The blown-up configuration space is defined as
  \begin{equation*}
    \hat\sC^* \coloneq (0,\infty)\times \sS \times \sA(\det W).
  \end{equation*}
  We denote its quotient by
  \begin{equation*}
    \hat\sB^* \coloneq \hat\sC^*/\sG.
  \end{equation*}
  The blown-up Seiberg--Witten equation \autoref{Eq_BlownUpPerturbedTwoSeibergWitten} gives rise to the $\sG$--equivariant section $\hsw_t^1 \co \hat\sC^* \to T\hat\sC^*$
  \begin{equation*}
    \hsw_t^1(\epsilon,\Psi,A)
    =
    \begin{pmatrix}
      \inner{\Psi}{\slD_{A,\bp_t}\Psi} \\
      -\slD_{A,\bp_t}\Psi+\inner{\Psi}{\slD_{A,\bp_t}\Psi}\Psi \\
      *\(\epsilon^2F_A + \epsilon^2\eta_t - \mu_{\bp_t}(\Psi)\)
    \end{pmatrix}.
  \end{equation*}
  This induces a section
  \begin{equation*}
    \hsw_t^1 \co \hat\sB^* \to T\hat\sB^*.
  \end{equation*}

  The $\sG$--equivariant map $\Upsilon\co \hat\sC^* \to \sC^*$ given by
  \begin{equation}
    \label{Eq_Upsilon}
    \Upsilon(\epsilon,\Psi,A) \coloneq (\epsilon^{-1}\Psi,A).
  \end{equation}
  is a diffeomorphism.
  The induced map $\Upsilon\co \hat\sB^* \to \sB^*$ is diffeomorphism as well.
  It is clear that $\Upsilon$ induces a bijection between $(\hsw_t^1)^{-1}(0)$ and $\sw_t^{-1}(0)$;
  however,
  \begin{equation*}
    \hsw_t^1 \neq \Upsilon^*\sw_t.
  \end{equation*}
  To see how $\hsw_t^1$ and $\Upsilon^*\sw_t$ are related,
  we compute the derivative $\rd\Upsilon \co T\hat\sC^* \to T\sC^*$ to be
  \begin{equation}
    \label{Eq_UpsilonDerivative}
    \rd_{(\epsilon,\Psi,A)}\Upsilon(\hat\epsilon,\psi,a)
    =
    \begin{pmatrix}
      -\frac{\hat\epsilon}{\epsilon^2}\Psi + \epsilon^{-1}\psi \\
      a
    \end{pmatrix},
  \end{equation}
  where $\hat\epsilon \in \R$ is a variation of $\epsilon$, $\psi \in T_\Psi \sS$ is a variation of $\Psi$, and $a$ is a variation of $A$.
  The inverse is given by
  \begin{equation}
    \label{Eq_UpsilonDerivativeInverse}
    (\rd_{(\epsilon,\Psi,A)}\Upsilon)^{-1}(\psi,a)
    =
    \begin{pmatrix}
      -\epsilon^2\inner{\Psi}{\psi} \\
      \epsilon\psi-\epsilon\inner{\Psi}{\psi}\Psi \\
      a
    \end{pmatrix}.
  \end{equation}
  Therefore, the pull-back section $\hsw_t^0 \coloneq \Upsilon^*\sw \co \hat\sC^* \to T\hat\sC^*$ is
  \begin{equation*}
    \hsw_t^0(\epsilon,\Psi,A)
    =
    \begin{pmatrix}
      \epsilon\inner{\Psi}{\slD_{A,\bp_t}\Psi} \\
      -\slD_{A,\bp_t}\Psi+\inner{\Psi}{\slD_{A,\bp_t}\Psi}\Psi \\
      *\(F_A + \eta_t - \epsilon^{-2}\mu_{\bp_t}(\Psi)\)
    \end{pmatrix}.
  \end{equation*}
  
  It follows that
  \begin{equation*}
    \hsw_t^1
    =
    \begin{pmatrix}
      \epsilon^{-1} \\
      & 1 \\
      & & \epsilon^2
    \end{pmatrix}
    \hsw_t^0.
  \end{equation*}
  We join $\hsw_t^0$ and $\hsw_t^1$ by the following path.
  For $\sigma \in [0,1]$, define $U_\epsilon^\sigma$ by
  \begin{equation*}
    U_\epsilon^\sigma
    \coloneq
    (1-\sigma)
    \begin{pmatrix}
      1 \\
      & 1 \\
      & & 1
    \end{pmatrix}
    +  
    \sigma
    \begin{pmatrix}
      \epsilon^{-1} \\
      & 1 \\
      & & \epsilon^2
    \end{pmatrix}.
  \end{equation*}
  For fixed $\epsilon > 0$, as $\sigma$ varies in $[0,1]$,
  this is a path of invertible matrices.
  Set
  \begin{equation}
    \label{Eq_SWSigma}
    \hsw_t^\sigma \coloneq U_\epsilon^\sigma\circ \hsw_t^0.
  \end{equation}
  Observe that if $\hsw_t^\sigma(\epsilon,\Psi,A) = 0$ for some $\sigma \in [0,1]$,
  then the same holds for every $\sigma \in [0,1]$.
  Moreover, for such $(\epsilon,\Psi,A)$, we have
  \begin{equation*}
    \rd_{(\epsilon,\Psi,A)} \hsw_t^\sigma \coloneq U_\epsilon^\sigma\circ \rd_{(\epsilon,\Psi,A)}\hsw_t^0(\epsilon,\Psi,A).
  \end{equation*}
  (If $(\epsilon,\Psi,A)$ is not a solution, then an extra term involving the derivative of $U_\epsilon^\sigma$ appears.)

  The path of sections $\hsw_t^\sigma$ descends to a path of sections of $T\hat\sB^*$.
  In what follows we will distinguish between elements in $\hat\sC^*$ or $T\hat\sC^*$ and their $\sG$--equivalence classes by writing the former in round brackets and the latter in square brackets.
  

  \begin{step}
    For $\epsilon \ll 1$, the operator $L_{\epsilon^{-1}\Psi_\epsilon, A_\epsilon,\bp_{t(\epsilon)}}$ is invertible (thus, the solution $[\epsilon^{-1}\Psi_\epsilon, A_\epsilon]$ is unobstructed) and
    \begin{equation} 
      \label{Eq_OrientationTransportAtInfinity}
      \OT(L_{\epsilon^{-1} \Psi_0,A_0,\bp_{t_0}}, L_{\epsilon^{-1}\Psi_\epsilon, A_\epsilon,\bp_{t(\epsilon)}}) = -\sign(\delta).
    \end{equation}
  \end{step}

  We compute the spectral flow from $\Upsilon(\fc_\epsilon^0) = (\epsilon^{-1}\Psi_0,A_0)$ and $\bp_{t_0}$ to $\Upsilon(\fc_\epsilon) = (\epsilon^{-1}\Psi_\epsilon,A_\epsilon)$ and $\bp_{t(\epsilon)}$ as the orientation transport from $\det \rd_{\fc_\epsilon^0}\sw_{t_0}$ to $\det \rd_{\fc_\epsilon^0}\sw_{t(\epsilon)}$; see \autoref{Sec_DeterminantLineBundles}.
  This orientation transport can be computed from the diagram of isomorphisms depicted and explained below.
  \begin{equation}
    \label{Eq_OrientationTransportDiagram}
    \begin{tikzcd}
      \R \ar[r,"\kappa"] & \det \rd_{\Upsilon(\fc_\epsilon^0)}\sw_{t_0} \ar[r,"\Upsilon_*^{-1}"] \ar[d] & \det \rd_{\fc_\epsilon^0}\hsw_{t_0}^0 \ar[r,"H_1"] \ar[d] &  \det \rd_{\fc_\epsilon^0}\hsw_{t_0}^1 \ar[d] \ar[r,"M_1"] & \R \ar[d,"="] \\
      \R  & \det \rd_{\Upsilon(\fc_\epsilon)}\sw_{t(\epsilon)} \ar[l,"\kappa^{-1}"] & \det \rd_{\fc_\epsilon}\hsw_{t(\epsilon)}^0 \ar[l,"\Upsilon_*"] &  \det \rd_{\fc_\epsilon}\hsw_{t(\epsilon)}^1 \ar[l,"H_2"] & \R \ar[l,"M_2"]
    \end{tikzcd}
  \end{equation}
  Here, $\kappa$ denotes the Knudsen--Mumford isomorphism introduced in \autoref{Sec_DeterminantLineBundles}; we use here that $\rd_\fc \sw_t$ is a self-adjoint operator (for any $\fc$ and $t$).
  Our goal is to compute the orientation transport \eqref{Eq_OrientationTransportAtInfinity} which, by definition, is the sign of the composition of $\kappa$, the first vertical arrow, and $\kappa^{-1}$. 
  This will be done by computing the remaining maps in the diagram which we now describe.
  
  The second top and bottom horizontal arrows are induced from the derivative of the map $\Upsilon$ defined by \eqref{Eq_Upsilon}.
  The first square commutes because $\hsw_t^0 = \Upsilon^*\sw_t$.
  
  The maps $H_1$ and $H_2$ are induced from the homotopy of operators $\hsw^{\sigma}_t$, defined by \eqref{Eq_SWSigma}, as $\sigma$ varies in $[0,1]$.
  The second square commutes because orientation transport is homotopy invariant.
    
  Finally, the right-most maps $M_1$ and $M_2$ are the trivializations of the determinant line bundle $\det \rd \hsw^1_t$ obtained from the Kuranishi model.
  The construction of $M_1$ and $M_2$, and the commutativity of the third square are addressed in \autoref{It_KuranishiModelSquare} below.

  We now compute the maps in the diagram.

  \begin{enumerate}
  \item
    By \autoref{Step_OperatorAtInfinity},
    the kernel and cokernel of $\rd_{\Upsilon(c_\epsilon^0)}\sw_{t_0}$ are spanned by $\Psi_0$.
    Thus,
    \begin{equation*}
      \kappa(1) = \Psi_0 \otimes \Psi_0^*.
    \end{equation*}
  \item
    By \eqref{Eq_UpsilonDerivativeInverse}, under the isomorphism $\Upsilon_*$ this correponds to
    \begin{equation*}
      \Upsilon_*^{-1}(\Psi_0\otimes\Psi_0^*) = \epsilon^4 \del_\epsilon\otimes \del_\epsilon^*,
    \end{equation*}
    where $\del_\epsilon$ is the tangent vector in the $\epsilon$--direction in $\hat\sB^* = (0,\infty)\times \sS \times \sA(\det W)/\sG$.
  \item
    The map $H_1$ is given by the homotopy $\hsw^{\sigma}_t$ as $\sigma$ varies in $[0,1]$. 
    We have
    \begin{align*}
      \rd_{\fc_{\epsilon}^0}\hsw_{t_0}^\sigma(\hat\epsilon,\psi,a)
      &=
        U_\epsilon^\sigma \circ \rd_{\fc_\epsilon^0}\hsw_{t_0}^0(\hat\epsilon,\psi,a)
        + \hat\epsilon (\del_\epsilon U_\epsilon^\sigma)\circ \hsw_{t_0}^0(\fc_{\epsilon_0}) \\
      &=
        U_\epsilon^\sigma \circ
        \begin{bmatrix}
          0 & 0 & 0 \\
          0 & -\slD_{A_0,\bp_{t_0}} & -\bar\gamma(\cdot)\Psi_0\\
          0 & -\epsilon^{-2}\rd_{\Psi_0}\mu & *\rd
        \end{bmatrix}
        + \hat\epsilon
          \begin{bmatrix}
            0 \\
            0 \\
            2\sigma\epsilon *\eta_{t_0}
          \end{bmatrix},
    \end{align*}
    because
    \begin{align*}
      \rd_{\fc_\epsilon^0}\hsw_{t_0}^0
      &=
        \begin{bmatrix}
          0 & 0 & 0 \\
          0 & -\slD_{A_0,\bp_{t_0}} & -\bar\gamma(\cdot)\Psi_0\\
          0 & -\epsilon^{-2}\rd_{\Psi_0}\mu & *\rd
        \end{bmatrix} \qand \\
      \hsw_{t_0}^0(\fc_{\epsilon_0})
      &=
        \begin{bmatrix}
          0 \\
          0 \\
          *\eta_{t_0}
        \end{bmatrix}.
    \end{align*}

    Set
    \begin{equation*}
      \ell
      \coloneq
      \begin{bmatrix}
        -\slD_{A_0,\bp_{t_0}} & -\bar\gamma(\cdot)\Psi_0\\
        -\epsilon^{-2}\rd_{\Psi_0}\mu & *\rd
      \end{bmatrix}.
    \end{equation*}
    The discussion in \autoref{Step_OperatorAtInfinity}, showed that $\ell$ is invertible.    
    The fact that the top row of $\rd_{\fc_{\epsilon}^0}\hsw_{t_0}^\sigma$ is zero tells us that its cokernel is spanned by $\del_\epsilon$.
    The kernel of $\rd_{\fc_{\epsilon}^0}\hsw_{t_0}^\sigma$ and it is spanned by
    \begin{equation*}
      \del_\epsilon + v_\epsilon^\sigma
      \qwithq
      v_\epsilon^\sigma = - \ell^{-1}(U_\epsilon^\sigma)^{-1}
      \begin{bmatrix}
        0 \\
        0 \\
        2\sigma\epsilon*\eta_{t_0}
      \end{bmatrix}.
    \end{equation*}
    The exact formula for $v_\epsilon^\sigma$ is not important;
    what matters is that it varies smoothly in $\epsilon$ and $\sigma$.
    The kernel and cokernel of $\rd_{\fc_{\epsilon}^0}\hsw_{t_0}^\sigma$ vary smoothly as $\sigma$ varies in $[0,1]$.
    This gives a trivialization of the line bundle over $[0,1]$ whose fiber over $\sigma \in [0,1]$ is the determinant line $\det \rd_{\fc_{\epsilon}^0}\hsw_{t_0}^\sigma$.
    This trivialization extends the one of $\det \rd_{\fc_{\epsilon}^0}\hsw_{t_0}^0$ obtained by choosing the basis $\del_\epsilon$ and $\del_\epsilon^*$ of the kernel and cokernel of $\rd_{\fc_{\epsilon}^0}\hsw_{t_0}^0$ respectively. 
    Since $H_1$ is defined as the orientation transport along the path $(\rd_{\fc_{\epsilon}^0}\hsw_{t_0}^\sigma)_{\sigma \in [0,1]}$, we see that 
    \begin{equation*}
      H_1(\del_\epsilon\otimes \del_\epsilon^*)
      =
      (\del_\epsilon+ v_\epsilon^1)\otimes \del_\epsilon^*.
    \end{equation*}
  \item
    \label{It_KuranishiModelSquare}
    We now describe the maps $M_1$ and $M_2$.
    We construct the Kuranishi model for $\hsw_t^1$ using the Implicit Function Theorem to find $\Xi_{t,\epsilon}$ such that
    \begin{equation*}
      \hsw_t^1(\epsilon,\Xi_{t,\epsilon}[\Psi_0+\psi,A_0+a])
      =
      \begin{bmatrix}
        f_t(\epsilon,\psi,a) \\
        \ell (\psi,a)
      \end{bmatrix}
      + \hsw_t^1(\epsilon,[\Psi_0,A_0]).
    \end{equation*}
    Here $f_t$ is a smooth $\R$--valued function and $\ell$ is the isomorphism introduced above.
    This yields an exact sequence
    \begin{equation*}
      0 \to \ker \rd_{\fc_{\epsilon}^0}\hsw_{t_0}^1 \to \R \xrightarrow{\del_\epsilon f_{t_0}(\epsilon,0,0)} \R \to \coker \rd_{\fc_{\epsilon}^0}\hsw_{t_0}^1 \to 0;
    \end{equation*}
    cf. \autoref{Eq_KuranishiModelExactSequence}.
    It follows from the description of $\rd_{\fc_{\epsilon}^0}\hsw_{t_0}^1$ above,
    that the first map sends $\del_\epsilon + v_\sigma$ to $1$ and
    the third map sends $1$ to $\del_\epsilon$.
    Consequently, the second map is zero.
    
    Similarly, there is an exact sequence
    \begin{equation*}
      0 \to \ker \rd_{c_{\epsilon}}\hsw_{t(\epsilon)}^1 \to \R \xrightarrow{\del_\epsilon f_{t(\epsilon)}(\epsilon,\psi_\epsilon,a_\epsilon)} \R \to \coker \rd_{\fc_{\epsilon}^0}\hsw_{t(\epsilon)}^1 \to 0.
    \end{equation*}  
    Here $(\epsilon,\Psi_0+\psi_\epsilon,A_0+a_\epsilon)$ is a solution of the blown-up Seiberg--Witten equation with two spinors and, by the definition of the obstruction map $\ob$, see \cite[Section 5.3]{Doan2017a}, we have
    \begin{equation*}
      \del_\epsilon f_{t(\epsilon)}(\epsilon,\psi_\epsilon,a_\epsilon) = (\del_\epsilon\ob)(t(\epsilon),\epsilon).
    \end{equation*}
    Since, by \eqref{Eq_ObstructionMapExpansion},
    \begin{equation*}
      (\del_\epsilon\ob)(t(\epsilon),\epsilon) = -4\delta\epsilon^3 + O(\epsilon^5)
    \end{equation*}
    and $0 < \epsilon \ll 1$,
    it follows that $\rd_{\fc_\epsilon}\hsw_{t(\epsilon)}^1$ is invertible.
    Since $\fc_\epsilon$ is a solution, $\rd_{\fc_\epsilon}\hsw_{t(\epsilon)}^1$ agrees with $\rd_{\fc_\epsilon}\sw = L_{\epsilon^{-1}\Psi_\epsilon,A_\epsilon,\bp_{t(\epsilon)}}$ (up to composition with an isomorphism and pulling back by a diffeomorphism).
    It follows that the solution $[\epsilon^{-1}\Psi_\epsilon,A_\epsilon]$ is unobstructed.

    These exact sequences yield the maps
    \begin{gather*}
      M_1 \co \det(\rd_{\fc_{\epsilon}^0}\hsw_{t_0}^1) \to \R\otimes \R^* = \R
      \qand \\
      M_2 \co \det(\rd_{c_{\epsilon}}\hsw_{t(\epsilon)}^1) \to \R\otimes \R^* = \R
    \end{gather*}
    appearing in the diagram \autoref{Eq_OrientationTransportDiagram}.
    These maps are computed using \autoref{Eq_PsiDet} to be
    \begin{equation*}
      M_1 \co
      (\del_\epsilon+v_\epsilon^1)\otimes(\del_\epsilon^*)
      \mapsto -1\otimes 1^*
      \mapsto -1
    \end{equation*}
    and
    \begin{equation*}
      M_2 \co
      1\otimes 1^*
      \mapsto -1\otimes (\del_\epsilon\ob)^*
      \mapsto -\del_\epsilon\ob = 4\delta\epsilon^3 + O(\epsilon^5).
    \end{equation*}
    (The signs come from the Knudsen--Mumford conventions used in the construction of the determinant line bundle.)
  \item
    Since $\fc_\epsilon$ is a solution,
    we have
    \begin{gather*}
      \det \rd_{\fc_\epsilon}\sw_{t(\epsilon)} = \det(0)\otimes\det(0)^* = \R, \quad
      \det \rd_{\fc_\epsilon}\hsw_{t(\epsilon)}^0 = \det(0)\otimes\det(0)^* = \R, \\ \andq
      \det \rd_{\fc_\epsilon}\hsw_{t(\epsilon)}^1 = \det(0)\otimes\det(0)^* = \R
    \end{gather*}
    and all of the maps on the bottom of \autoref{Eq_OrientationTransportDiagram} are the identity $\R = \R$.
  \item
    Finally, putting everything together we see that composition of all maps is
    \begin{equation*}
      1
      \mapsto \Psi_0\otimes\Psi_0^*
      \mapsto \epsilon^4 \del_\epsilon\otimes\del_\epsilon^*
      \mapsto \epsilon^4 (\del_\epsilon+v_\epsilon^1)\otimes\del_\epsilon^*
      \mapsto -\epsilon^4
      \mapsto \epsilon^4\del_\epsilon\ob = -4\delta\epsilon^7 + O(\epsilon^9).
    \end{equation*}
    We conclude that the orientation transport \eqref{Eq_OrientationTransportAtInfinity}, which is equal to the sign of the composition of all of the above maps,  is $-\sign(\delta)$ as we wanted to show.
    \qedhere
  \end{enumerate}
\end{proof}


\section{Proof of the wall-crossing formulae}
\label{Sec_WallCrossingProof}

This section will conclude the proof of \autoref{Thm_TwoSeibergWittenWallCrossing}.
The local wall-crossing formula \eqref{Eq_LocalWallCrossingFormula} and \autoref{Prop_HowToDetermineSigma} directly imply the wall-crossing formula
\makeatletter
\let\oldtagform@\tagform@
\renewcommand\tagform@[1]{\maketag@@@{\ignorespaces#1\unskip\@@italiccorr}}
\makeatother
\begin{equation}
  \tag{\eqref{Eq_WallCrossingFormula}}
  n_\fw(\bp_1,\eta_1)
  =
    n_\fw(\bp_0,\eta_0)
  +
    \sum_\fs \sum_{i=1}^{N_\fs} \chi^\fs_i \cdot \sigma(\Psi^\fs_i,\bp_{t^\fs_i}).
\end{equation}
\makeatletter
\let\tagform@\oldtagform@
\makeatother
In order to prove \eqref{Eq_WallCrossingFormula_2}, we need to relate $\sigma(\Psi_0,\bp_0)$ to $\sigma(\Psi_1,\bp_1)$ for two different nowhere vanishing spinors $\Psi_0$ and $\Psi_1$ and two parameters $\bp_0$ and $\bp_1$.

\begin{prop}
  \label{Prop_HowSigmaChanges}
  Let $(\bp_t)_{t\in[0,1]}$ be a path in $\sP$.
  If $\Psi_0$ and $\Psi_1$ are two nowhere vanishing sections of $\Re(E\otimes S_\fs)$, then
  \begin{equation}
    \label{Eq_HowSigmaChanges}
    \sigma(\Psi_1,\bp_1)
    =
      \sigma(\Psi_0,\bp_0)
      \cdot (-1)^{\SF(-\slD_{\bp_t}^\fs)}
      \cdot (-1)^{\deg(\Psi_0,\Psi_1)}.
  \end{equation}
  Here $\deg(\Psi_0,\Psi_1)$ denotes the relative the degree of $\Psi_0$ and $\Psi_1$ as in \autoref{Def_RelativeDegree}.
\end{prop}

\begin{proof}
  Suppose first that $\deg(\Psi_0,\Psi_1) = 0$ so that we can find a path $(\Psi_t)_{t\in[0,1]}$ of nowhere vanishing sections from $\Psi_0$ to $\Psi_1$.
  It follows from \autoref{Def_Sigma} that
  \begin{equation*}
    \sigma(\Psi_1,\bp_1)
    =
      \sigma(\Psi_0,\bp_0)
      \cdot (-1)^{\SF(-\fd_{\Psi_t})}.
  \end{equation*}
  The spectral flow along the path of operators $-\fd_{\Psi_t}$ is identical to that of the path of operators
  \begin{equation}
   \label{Eq_PathOfOperators}
    -\tilde\fa_{\Psi_0}\circ\fd_{\Psi_t}\circ \tilde \fa_{\Psi_0}^*.
  \end{equation}
  Consider the homotopy of paths
  \begin{equation*}
    (t,s) \mapsto -\tilde\fa_{\Psi_{st}} \circ \fd_{\Psi_t} \circ \tilde\fa_{\Psi_{st}}^*.
  \end{equation*}
  It is well-defined since $\Psi_t$ is nowhere vanishing for $t \in [0,1]$.
  For $s=0$ it gives us the path \eqref{Eq_PathOfOperators} whereas for $s=1$ we obtain 
  \begin{equation*}
    -\tilde\fa_{\Psi_t}\circ\fd_{\Psi_t}\circ \tilde \fa_{\Psi_t}^*
    =
      -\slD_{\bp_t}^\fs.
  \end{equation*}
  Since the spectral flow is a homotopy invariant, we have $\SF(-\fd_{\Psi_t}) = \SF(-\slD_{\bp_t}^\fs)$ which proves \eqref{Eq_HowSigmaChanges} if $\deg(\Psi_0,\Psi_1) = 0$.

  It remains to deal with the case $\deg(\Psi_0,\Psi_1) \neq 0$.
  By the above, we can assume that $\bp_t = \bp$ for all $t \in [0,1]$.
  Since $\sigma(\Psi_1,\bp)$ and $\sigma(\Psi_2,\bp)$ are defined using the spectral flow from $\fd_{\Psi_0}$ and $\fd_{\Psi_1}$ respectively to a given elliptic operator (see \autoref{Def_Sigma}), it follows that for any path of elliptic operators $(\fd_t)_{t\in[0,1]}$ connecting $\fd_{\Psi_0}$ and $\fd_{\Psi_1}$ we have
  \begin{equation*}
        \sigma(\Psi_1,\bp)\cdot \sigma(\Psi_0,\bp) = (-1)^{\SF(-\fd_t)}.
  \end{equation*}
  By work of Atiyah--Patodi--Singer \cite[Section 7]{Atiyah1976}, the spectral flow $\SF(-\fd_t)$ is equal to the index of an elliptic operator $\fD$ on $S^1 \times M$ constructed as follows.
  Set $V \coloneq (T^*M\oplus\underline{\R})\otimes i\R$ and define an isometry $f \co V \to V$ by
  \begin{equation*}
    f \coloneq \tilde \fa_{\Psi_0}^*\tilde \fa_{\Psi_1}
  \end{equation*}
  By the definition of $f$, we have
  \begin{equation*}
    \fd_{\Psi_1} = f^{-1} \circ \fd_{\Psi_0} \circ f.
  \end{equation*}
  Let $\bV \to S^1 \times M$ be the vector bundle obtained as the mapping torus of $f$; that is, $\bV = V \times [0,1] /\sim$ where $\sim$ denotes the equivalence relation $(v,0) \sim (f(v),1)$. 
  If, as before, $(\fd_t)_{t\in [0,1]}$ is a family of elliptic operators connecting $\fd_{\Psi_0}$ and $\fd_{\Psi_1}$, then the operator $\del_t - \fd_t$ on the pull-back of $V$ to $[0,1] \times M$, with $t$ denoting the coordinate on $[0,1]$, gives rise to a first order elliptic operator $\fD$ on $\bV \to S^1 \times M$ whose index equals $\SF(-\fd_t)$.
  We compute this index as follows.
  Under the isomorphism $\tilde\fa_{\Psi_0}$ between $V$ and $\Re(S_\fs\otimes E)$ the operator $\fd_{\Psi_0}$ corresponds to $\slD^\fs_{\bp}$. 
  Moreover, under this isomorphism, the complex-linear extension of $f$ corresponds to an isomorphism of $S_\fs\otimes E$ given by a gauge transformation $g$ of degree $\deg(\Psi_0,\Psi_1)$ of the $\SU(2)$--bundle $E$ (this is because in a local trivialization $f$ is given simply by right-multiplication by a quaternion-valued function).
  Thus, the complexification of $\bV$ is isomorphic to $\bS^+_\fs\otimes\bE$ where $\bS^+_\fs$ is the positive spinor bundle of $S^1 \times M$ and $\bE$ is obtained by gluing $E \to [0,1]\times M$ along $\{0,1\} \times M$ using $g$.
  The complexification $\fD^\C$ of the operator $\fD$ corresponds in this identification to the Dirac operator on $\bS^+_\fs$ twisted by a connection on $\bE$. 
  By the Atiyah--Singer Index Theorem,
  \begin{equation*}
    \ind \fD^\C = \int_{S^1 \times M} \hat{A}(S^1 \times M) \ch(\bE) = -\int_{S^1 \times M} c_2(\bE) = - \deg(\Psi_0,\Psi_1).
  \end{equation*}
  The real index of $\fD$ is equal to the complex index of $\fD^\C$ and we conclude that
      \begin{equation*}
    \sigma(\Psi_1,\bp)\cdot \sigma(\Psi_0,\bp)
    = (-1)^{\SF(-\fd_t)} 
    = (-1)^{\deg(\Psi_0,\Psi_1)}.
  \end{equation*}
  This completes the proof of this proposition.
\end{proof}

Recall that $t_1^\fs, \ldots, t_{N_\fs}^\fs \in (0,1)$ are the times at which the spectrum of $\slD^\fs_{\bp_t}$ crosses zero.
The crossing is transverse with intersection sign $\chi^\fs_i \in \{ \pm 1 \}$.
The next result relates $\chi_i^\fs$ to $\chi_1^\fs$.

\begin{prop}
  For all $i \in \set{1,\ldots,N_{\fw}}$, we have
  \begin{equation*}
    \chi^\fs_i\cdot (-1)^{\SF\paren*{-\slD_{\bp_t}^\fs : t \in \left[t^\fs_1,t^\fs_i\right]}}
    = (-1)^{i+1}\cdot\chi^\fs_1.
  \end{equation*}
\end{prop}

\begin{remark}
  In the above, the operator $-\slD_{\bp_t}^\fs$ is not invertible for $t \in \set{t_1^\fs, \ldots, t_{N_\fs}^\fs}$.
  According to \autoref{Conv_SpectralFlow},
  the spectral flow from $t_1^\fs$ to $t_i^\fs$ is defined as the spectral flow of the family $(-\slD_{\bp_t}^\fs + \lambda \id)_{t\in [t_1^\fs, t_i^\fs]}$ for any number $0 < \lambda \ll 1$.
\end{remark}

\begin{proof}
  By induction it suffices to consider the case $i = 2$.
  The case $i=2$ can be verified directly case-by-case as follows:
  \begin{center}
    \begin{tabular}{ r r r | c }
      $\chi_1$ & $\chi_2$ & $\SF$ & $\chi_1\cdot \chi_2\cdot(-1)^{\SF}$ \\
      \hline
      $+$1 & $+1$ & $-1$ & $-1$ \\
      $+1$ & $-1$ & $2$ & $-1$ \\
      $-1$ & $+1$ & $0$ & $-1$ \\
      $-1$ & $-1$ & $1$ & $-1$
    \end{tabular}
  \end{center}
  Here $\SF = \SF\paren[\big]{-\slD_{\bp_t}^\fs : t \in \left[t^\fs_1,t^\fs_2\right]}$.
\end{proof}

Combining the two preceding propositions shows that \eqref{Eq_WallCrossingFormula} can equivalently be written as follows:
\makeatletter
\let\oldtagform@\tagform@
\renewcommand\tagform@[1]{\maketag@@@{\ignorespaces#1\unskip\@@italiccorr}}
\makeatother
\begin{equation}
  \tag{\eqref{Eq_WallCrossingFormula_2}}
  n_\fw(\bp_1,\eta_1)
  =
    n_\fw(\bp_0,\eta_0)
  +
    \sum_\fs
      \chi^\fs_1 \cdot \sigma(\Psi^\fs_1,\bp_{t^\fs_1}) \cdot
      \sum_{i=1}^{N_\fs}
        (-1)^{i+1}\cdot(-1)^{\deg(\Psi^\fs_1,\Psi^\fs_i)}.
\end{equation}
\makeatletter
\let\tagform@\oldtagform@
\makeatother
This completes the proof of \autoref{Thm_TwoSeibergWittenWallCrossing}.
\qed


\section{Transversality for paths}
\label{Sec_TransversalityForPaths}

The purpose of this section is to prove \autoref{Prop_TransversalityForPaths}.

\begin{prop}
  \label{Prop_TransverseSpectralFlow}
  For any $\bp_0,\bp_1 \in \sP^\reg$,
  the subspace $\bsP^\reg(\bp_0,\bp_1)$ is residual in the space of all smooth paths from $\bp_0$ to $\bp_1$ in $\sP$.
\end{prop}  

\begin{proof}
  The proof is an application of the Sard--Smale theorem. 
  We will work with Sobolev spaces of sections and connections of class $W^{k,p}$ such that $(k-1)p > 3$.
  The statement for $C^{\infty}$ spaces follows then from a standard argument; see, for example, \cite[Proof of Proposition 2.21]{Doan2017}.

  Since there are only finitely many spin structures on $M$, it suffices to consider the conditions \itref{Def_RegularPath_TransverseSpectralFlow} and \itref{Def_RegularPath_PsiNowhereVanishing} in \autoref{Def_RegularPath} for a fixed spin structure $\fs$.
  Set
  \begin{align*}
    X
    &\coloneq
      \bsP(\bp_0,\bp_1) \times [0,1] \times \frac{\Gamma(\Re(S_\fs\otimes E) \setminus \set{0})}{\R^*}
    \\ \andq
    V
    &\coloneq
    \bsP(\bp_0,\bp_1) \times [0,1] \times \frac{\Gamma(\Re(S_\fs\otimes E) \setminus \set{0}) \times \Gamma(\Re(S_\fs\otimes E))}{\R^*}.
  \end{align*}
  $V$ is a vector bundle over $X$.
  Define a section $\bsigma \in \Gamma(V)$ by
  \begin{equation*}
    \bsigma\paren*{(\bp_t)_{t\in[0,1]},t_*,[\Psi]}
    \coloneq
      \paren*{(\bp_t)_{t\in[0,1]},t_*,[(\Psi, \slD_{\bp_{t_*}}\Psi)]}.
  \end{equation*}
  We can identify a neighborhood of $[\Psi] \in \Gamma(\Re(S_\fs\otimes E) \setminus \set{0})/\R^*$ with the $L^2$--orthogonal complement $\Psi^{\perp} \subset  \Gamma(\Re(S_\fs\otimes E)$. 
  This gives us a local trivialization of $V$ in which $\bsigma$ can be identified with the map
  \begin{equation*}
    \bsigma\paren*{(\bp_t)_{t\in[0,1]},t_*,\psi}
    =
      \paren*{(\bp_t)_{t\in[0,1]},t_*,\slD_{\bp_{t_*}}(\Psi + \psi)}
      \quad \text{for all } \psi \in \Psi^{\perp}.
  \end{equation*}
  In particular, for a fixed path $(\bp_t)_{t\in[0,1]}$, the map $\bsigma((\bp_t)_{t\in[0,1]},\cdot)$ defines a Fredholm section of index zero,
  since $\psi \mapsto \slD_{\bp_t}(\Psi+\psi)$ has index $-1$ and $\dim [0,1] = 1$.

  Let $\bx \coloneq \paren*{(\bp_t)_{t\in[0,1]},t_*,[\Psi]} \in X$ and denote by $\rd_\bx\bsigma$ the linearization of $\bsigma$ at $\bx$ (and computed in the above trivialization).
  We will prove that $\rd_\bx\bsigma$ is surjective provided $\bsigma(\bx) = 0$.
  If $\Phi \in V_\bx = \Gamma(\Re(S_\fs\otimes E))$ is orthogonal to the image of $\rd_\bx\bsigma$, then it follows that
  \begin{equation}
    \label{Eq_ConnectionVariationPerp}
    \inner{\bar\gamma(b)\Psi}{\Phi}_{L^2} = 0
    \quad
    \text{for all } b \in \Omega^1(M,\su(E)).
  \end{equation}
  Since $\Psi$ is harmonic, its zero set must be nowhere dense.
  Clifford multiplication by $T^*M\otimes\su(E)$ on $\Re(S_\fs\otimes E)$ induces a isomorphism between $T^*M\otimes\su(E)$ and trace-free symmetric endomorphisms of $\Re(S_\fs\otimes E)$.
  With this in mind it follows from \eqref{Eq_ConnectionVariationPerp} that $\Psi = 0$.
  This proves that $\rd_\bx\bsigma$ is surjective.

  It follows that $\bsigma^{-1}(0)$ is a smooth submanifold of $X$ and the projection map $\pi\co \bsigma^{-1}(0) \to \bsP(\bp_0,\bp_1)$ is a Fredholm map of index zero.
  The kernel of $\rd\pi$ at $\bx \in \bsigma^{-1}(0)$ can be identified with the kernel of the linearization of $\bsigma$ in the directions of $[0,1]$ and $\Gamma(\Re(S_\fs\otimes E) \setminus \set{0})/\R^*$.
  Writing down this linearization explicitly, we see that the condition $\ker\rd\pi(\bx) = \set{0}$ implies that $\Psi$ spans $\ker \slD_{\bp_{t_*}}$ and $t_*$ is a regular crossing of the spectral flow of $(\slD_{\bp_t})$.
  On the other hand, since $\pi$ is a Fredholm map of index zero, $\ker d\pi(\bx) = \set{0}$ is equivalent to $\bx$ being a regular point of $\pi$.
  By the Sard--Smale theorem, the subspace of regular values of $\pi$ is residual;
  hence, the set of those $(\bp_t)_{t\in[0,1]}$ in $\bsP(\bp_0,\bp_1)$ for which the condition \itref{Def_RegularPath_TransverseSpectralFlow} in \autoref{Def_RegularPath} holds is residual.

  To deal with condition \itref{Def_RegularPath_PsiNowhereVanishing} in \autoref{Def_RegularPath},
  we consider the vector bundle
  \begin{equation*}
    W
    \coloneq
    \bsP(\bp_0,\bp_1) \times [0,1] \times \frac{\Gamma(\Re(S_\fs\otimes E) \setminus \set{0}) \times \Gamma(\Re(S_\fs\otimes E)) \times \Re(S_\fs\otimes E)}{\R^*}
  \end{equation*}
  over $X\times M$ and define a section $\btau \in \Gamma(W)$ by
  \begin{equation*}
    \btau\paren*{(\bp_t)_{t\in[0,1]},t_*,[\Psi],y}
    \coloneq
    \paren*{(\bp_t)_{t\in[0,1]},t_*,[(\Psi, \slD_{\bp_{t_*}}\Psi,\Psi(y))]}.
  \end{equation*}
  For a fixed path $(\bp_t)_{t\in[0,1]}$, the map $\btau((\bp_t)_{t\in[0,1]},\cdot)$ defines a Fredholm section of index $-1$.
  Note that for $\bx \coloneq \paren*{(\bp_t)_{t\in[0,1]},t_*,[\Psi]} \in X$ and $y \in M$ the condition $\btau((\bp_t)_{t\in[0,1]}, \bx,y) = 0$ is equivalent to $\slD_{\bp_{t_*}} \Psi = 0$ and $\Psi(y) = 0$.
  We prove that the linearization of $\btau$ is surjective at any $(\bx,y)$ satisfying these equations. 
  If $(\Phi,\phi) \in W_{(\bx,y)} = \Gamma(\Re(S_\fs\otimes E))\times \Re(S_\fs\otimes E)_y$ is orthogonal to the image of $(\rd\tau)_{(\bx,y)}$, then \eqref{Eq_ConnectionVariationPerp} holds and, moreover,
  \begin{equation}
    \label{Eq_DiracPerp}
    \inner{\slD_{\bp_{t_*}}(\Psi+\psi)}{\Phi}_{L^2} + \inner{\psi(x)}{\phi} = 0
    \quad
    \text {for all } \psi \in \Psi^\perp
  \end{equation}
  Since $\slD_{\bp_{t_*}} \Psi = 0$ and $\Psi(y) = 0$, \eqref{Eq_DiracPerp} holds in fact for all $\psi \in \Gamma(\Re(S_\fs\otimes E))$ and we conclude that $\slD_{\bp_{t_*}}\Psi = 0$.
  Plugging this back into \eqref{Eq_DiracPerp} yields $\inner{\psi(x)}{\phi} = 0$ for all $\psi$, which implies that $\phi = 0$.
  As before \eqref{Eq_ConnectionVariationPerp} implies that $\Phi = 0$.

  It follows that $\btau^{-1}(0)$ is smooth and the projection $\rho \co \btau^{-1}(0) \to \bsP(\bp_0,\bp_1)$ is a Fredholm map of index $-1$; in particular, the preimage of a regular value must be empty.
  It follows that the paths $(\bp_t)_{t\in[0,1]} \in \bsP(\bp_0,\bp_1)$ for which condition \itref{Def_RegularPath_PsiNowhereVanishing} in \autoref{Def_RegularPath} holds is residual.
\end{proof}

To address condition \itref{Def_RegularPath_DeltaNonZero} in \autoref{Def_RegularPath} we compute $\delta(\Psi,\bp,\eta)$.

\begin{prop}
  \label{Prop_Delta}
  Let $(\bp_t)_{t\in[0,1]} \in \bsP(\bp_0,\bp_1)$,
  let $(\eta_t)_{t\in[0,1]}$ be a path in $\sZ$,
  let $t_0 \in (0,1)$, and
  let $\Psi$ be a nowhere vanishing section of $\Re(S_\fs\otimes E)$ spanning $\ker \slD_{\bp_{t_0}}^\fs$ and satisfying $\Abs{\Psi}_{L^2} = 1$.
  There is a linear algebraic operator 
  \begin{equation*}
  \ff_{\bp_{t_0},\Psi} \co \Omega^2(M, i\R) \to \Omega^2(M, i\R)
  \end{equation*}
   such that
  \begin{equation}
  \label{Eq_FormulaForDelta}
    \delta(\Psi,\bp_{t_0},\eta_{t_0})
    =
      \int_M \abs{\Psi}^{-2} \inner{\rd*\eta_{t_0} + \ff_{\Psi,\bp_{t_0}}\eta_{t_0}}{\eta_{t_0}}
  \end{equation}  
  with $\delta(\Psi,\bp_{t_0},\eta_{t_0})$ is as in \autoref{Def_RegularPath}\itref{Def_RegularPath_DeltaNonZero}.
\end{prop}

\begin{proof}
  If we denote by
  \begin{equation*}
    \set*{ \(\Psi_\epsilon = \Psi + \epsilon^2\psi + O(\epsilon^4), A_\epsilon; t(\epsilon) = t_0 + O(\epsilon^2)\) : 0 \leq \epsilon  \ll 1 }
  \end{equation*}
  the family of solutions to
  \begin{align*}
    \slD_{A_\epsilon,\bp_{t(\epsilon)}}\Psi_\epsilon &= 0, \\
    \epsilon^2\paren*{\frac12F_{A_\epsilon}+\eta_{t(\epsilon)}} &= \mu_{\bp_{t(\epsilon)}}(\Psi_\epsilon), \qand \\
    \Abs{\Psi_\epsilon}_{L^2} &= 1
  \end{align*}
  obtained from \cite[Theorem 1.38]{Doan2017a}, then
  \begin{equation*}
    \delta
    = \delta(\Psi,\bp_{t_0},\eta_{t_0})
    = \inner{\slD_{A_0,\bp_{t_0}}\psi}{\psi}_{L^2}.
  \end{equation*}
  The connection $A = A_0$ corresponding to $\Psi$ is flat;
  see \autoref{Prop_SpecialHaydysCorrespondence}.
  For the unperturbed equation we would have $t(\epsilon) = 0$.
  Since we consider the perturbed equation, however, $\eta$ enters into the computation of $\delta$.
  More precisely, by \cite[Equations 6.1, 6.2]{Doan2017a}
  we have
  \begin{equation*}
    \psi
    =
    - \slD_{\Re}^{-1}\gamma\rII^*\nu
    - \nu
    \quad\text{with}\quad
    \nu \coloneq (\fa_\Psi^*)^{-1}*\eta.
  \end{equation*}
  By \eqref{Eq_APsiAPsi*}, we have
  \begin{equation*}
    \nu = \abs{\Psi}^{-2}\bar\gamma(*\eta)\Psi.
  \end{equation*}

  Denote by $\pi_{\Re}$ the projection onto $\Re(S_\fs\otimes E)$.
  From \cite[Proposition 3.8]{Doan2017a}
  we know that for any $a \in \Omega^1(M,i\R)$
  \begin{equation*}
    -\gamma\rII^*\bar\gamma(a)\Psi
    =
      \sum_{i=1}^3 \pi_{\Re}\(\rho(a(e_i))\nabla^A_{e_i}\Psi\)
    =
      0.
  \end{equation*}
  It follows that
  \begin{equation*}
    \psi = -\nu = \abs{\Psi}^{-2}\bar\gamma(*\eta)\Psi.
  \end{equation*}

  Set
  \begin{equation*}
    a \coloneq \abs{\Psi}^{-2}*\eta.
  \end{equation*}
  Since
  \begin{equation*}
    \slD_A \bar\gamma(a) \Psi
    =
    \bar\gamma(*\rd a)\Psi
    + \rho(\rd^*a)\Psi
    -2 \sum_{i=1}^3 \rho(a_i)\nabla_i\Psi,
  \end{equation*}
  we have
  \begin{align*}
    \delta
    &=
      \int_M \inner{\slD_A \bar\gamma(a) \Psi}{\bar\gamma(a)\Psi} \\
    &=
      \int_M \inner{\bar\gamma(*\rd a)\Psi}{\bar\gamma(a)\Psi}
      + \inner{\rho(\rd^*a)\Psi}{\bar\gamma(a)\Psi} \\
    &\qquad
      - 2\sum_{i=1}^3 \inner{\rho(a_i)\nabla_i\Psi}{\bar\gamma(a)\Psi}.
  \end{align*}
  The first term in the integral is
  \begin{equation*}
    \abs{\Psi}^2\inner{*\rd a}{a}
    =
    \inner{\rd*\(\abs{\Psi}^{-2}\eta\)}{\eta}.
  \end{equation*}
  The second term vanishes.
  Therefore,
  \begin{equation*}
    \delta
    = 
      \int_M \abs{\Psi}^{-2}\inner{\rd*\eta}{\eta}
      + \abs{\Psi}^{-2} \inner{\ff_1 \eta}{\ff_2 \eta}
  \end{equation*}
  for linear operators $\ff_1, \ff_2 \co \Omega^2(M,i\R) \to \Gamma(\Re(S_\fs\otimes E))$ of order zero. 
  Set $\ff_{\Psi,\bp} = \ff_2^*\ff_1$. 
\end{proof}

\begin{prop}
  \label{Prop_DeltaNonZero}
  For $(\Psi,\bp,\eta) \in \Gamma(\Re(S_\fs\otimes E)) \times \sP \times \sZ$ such that $\Psi$ is nowhere vanishing, define $\delta(\Psi,\bp,\eta)$ by formula \eqref{Eq_FormulaForDelta}. 
  For each $(\Psi,\bp)$,
  the set
  \begin{equation*}
    \sZ^{\reg}(\Psi,\bp)
    = \set*{ \eta \in \sZ : \delta(\Psi,\bp,\eta) \neq 0 }
  \end{equation*}
  is open and dense in $\sZ$.
\end{prop}

\begin{proof}
  Replace all the spaces in question by their completions with respect to the $W^{k,p}$ norm for any $k$ and $p$ satisfying $(k-1)p >3$.  
  We will prove the statement with respect to the Sobolev topology;
  the corresponding statement for $C^{\infty}$ spaces follows then from the Sobolev embedding theorem and the fact that $\delta$ is continuous with respect to any of these topologies.
  
  By \autoref{Prop_Delta},
  \begin{equation*}
    (\rd\delta)_\eta[\hat\eta]
    =
      \int_M \abs{\Psi}^{-2}\inner{2\rd*\eta + \bar\ff_{\Psi,\bp} \eta}{\hat\eta},
  \end{equation*}
  where 
  \begin{equation*}
    \bar\ff_{\Psi,\bp}\eta
    =
      (\ff_{\Psi,\bp} + \ff_{\Psi,\bp}^*)\eta - 2 \rd(\log \abs{\Psi})\wedge *\eta
  \end{equation*}
  is a linear algebraic operator. 
  Thus, the derivative of $\delta$ vanishes along the set $\sZ_{\rm crit}(\Psi,\bp)$ of solutions $\eta$ to the linear elliptic differential equation
  \begin{align*}
    \rd \eta &= 0, \\
    \rd * \eta + * \bar\ff_{\Psi,\bp} \eta &= 0.
  \end{align*}
  $\sZ_{\rm crit}(\Psi,\bp)$ a closed, finite-dimensional subspace of $\sZ$. 
  By the Implicit Function Theorem, away from $\sZ_{\rm crit}(\Psi,\bp)$, the zero set of $\delta(\Psi,\bp,\cdot)$ is a codimension one Banach submanifold of (the Sobolev completion of) $\sZ$. 
  Hence, the set
  \begin{equation*}
    \sZ^{\reg}(\Psi,\eta) \cap (\sZ \setminus \sZ_{\rm crit}(\Psi,\eta))
  \end{equation*}
  is dense. 
  Since $\delta$ is continuous, $\sZ^{\reg}(\Psi,\eta)$ is open.
\end{proof}

\begin{proof}[Proof of \autoref{Prop_TransversalityForPaths}]
  By \cite[Theorem 2.34]{Doan2017}, $\bsQ_{3}$, the subspace of paths from $(\bp_0,\eta_0)$ to $(\bp_1,\eta_1)$ satisfying \autoref{Def_RegularPath}\itref{Def_RegularPath_Unobstructed}, is residual.

  Denote by $\bsQ_{1,2}$ the space of paths from $(\bp_0,\eta_0)$ to $(\bp_1,\eta_1)$ satisfying conditions \itref{Def_RegularPath_TransverseSpectralFlow} and \itref{Def_RegularPath_PsiNowhereVanishing} in \autoref{Def_RegularPath}.
  By \autoref{Prop_TransverseSpectralFlow}, $\bsQ_{1,2}$ is residual in the space of paths from $(\bp_0,\eta_0)$ to $(\bp_1,\eta_1)$.
  Let $\bsQ_{1,2,4} \subset \bsQ_{1,2}$ be the space of paths from $(\bp_0,\eta_0)$ to $(\bp_1,\eta_1)$ also satisfying \autoref{Def_RegularPath}\itref{Def_RegularPath_DeltaNonZero}.
  Elementary arguments show that $\bsQ_{1,2,4}$ is open in $\bsQ_{1,2}$ and we will shortly prove that $\bsQ_{1,2,4}$ is dense in $\bsQ_{1,2}$.
  A set which is open and dense in a residual set is itself residual. 
  It follows that $\bsQ_{1,2,4}$ is residual; hence, so is $\bsQ^\reg\paren[\big]{(\bp_0,\eta_0),(\bp_1,\eta_1)} = \bsQ_{1,2,4} \cap \bsQ_{3}$.
  
  To prove that $\bsQ_{1,2,4}$ is dense in $\bsQ_{1,2}$, suppose that $(\bp_t,\eta_t)_{t\in[0,1]} \in \bsQ_{1,2}$. 
  There are finitely many times $0 < t_1 < \ldots < t_n < 1$ for which the kernel of $\slD_{\bp_{t_i}}$ is non-trivial.
  For $i = 1,\ldots,n$, denote by $\Psi_i$ a section spanning $\ker \slD_{\bp_{t_i}}$.
  By \autoref{Prop_DeltaNonZero}, for any $\sigma > 0$, there are closed forms $\alpha_1,\ldots,\alpha_n$ such that
  \begin{equation*}
    \delta(\Psi_i,\bp_{t_i},\eta_{t_i}+\alpha_i)
    \neq
      0
    \qandq
    \Abs{\alpha_i}_{L^\infty} \leq \sigma
  \end{equation*}
  for every $i = 1,\ldots,n$.
  Let $(\alpha_t)_{t\in[0,1]}$ be a path of closed forms such that $\alpha_{t_i} = \alpha_i$ for $i=1,\ldots,n$ and $\Abs{\alpha_t}_{L^\infty} \leq \sigma$ for all $t \in [0,1]$.
  The path $(\bp_t,\eta_t+\alpha_t)_{t\in[0,1]}$ satisfies conditions \itref{Def_RegularPath_TransverseSpectralFlow}, and \itref{Def_RegularPath_PsiNowhereVanishing} in \autoref{Def_RegularPath} because these only depend on $(\bp_t)_{t\in[0,1]}$.
  It also satisfies \itref{Def_RegularPath_DeltaNonZero} by construction.
  We conclude that $(\bp_t,\eta_t+\alpha_t)_{t\in[0,1]} \in \bsQ_{1,2,4}$.
  Since $\sigma$ is arbitrary, it follows that $\bsQ_{1,2,4}$ is dense in $\bsQ_{1,2}$.
\end{proof}


\section{Proof of the existence of singular harmonic \texorpdfstring{$\Z_2$}{Z2} spinors}
\label{Sec_ProofOfExistence}

In this section we prove \autoref{Thm_ExistenceOfSingularHarmonicZ2Spinors}.
We begin with defining the set $\sW_b$ appearing in its statement.

\begin{definition}
  \label{Def_Walls}
  Given a spin structure $\fs$, set
  \begin{align*}
    \sW^\fs
    &\coloneq
      \set*{ \bp \in \sP : \dim\ker\slD^\fs_\bp > 0 }, \\
    \sW_{1,\emptyset}^\fs
    &\coloneq
      \set*{
        \bp \in \sP
        :
        \ker\slD^\fs_\bp = \R\Span{\Psi} \text{ with } \Psi \text{ nowhere vanishing}
      }, \\
    \sW^\fs_{1,\star}
    &\coloneq
      \set*{
        \bp \in \sP
        :
        \ker\slD^\fs_\bp = \R\Span{\Psi} \text{ and } \Psi \text{ has a single non-degenerate zero}
      } \qand \\
    \sW_b^\fs
    &\coloneq
      \sW^\fs \setminus \sW_{1,\emptyset}^\fs.
  \end{align*}
  A zero $x \in \Psi^{-1}(0)$ is non-degenerate if the linear map $(\nabla \Psi)_x \co T_xM \to \Re(S\otimes E)_x$ has maximal rank (that is, rank three).
  Set
  \begin{equation*}
    \sW
    \coloneq
      \bigcup_\fs \sW^\fs, \quad
    \sW_b
    \coloneq
      \bigcup_\fs \sW_b^\fs, \qandq
    \sW_{1,\emptyset}
    \coloneq
      \bigcup_\fs \sW_{1,\emptyset}^\fs \setminus \sW_b.
  \end{equation*}
  Here the union is taken over all spin structures $\fs$.
\end{definition}

\begin{prop}
  \label{Prop_NaturalCoorientation}
  $\sW_{1,\emptyset}^\fs$ is a closed, codimension one submanifold of $\sP\setminus \sW_b^\fs$.
  It carries a coorientation such that the following holds.
  Let $(\bp_t)$ be a path in $\sP\setminus \sW_b^\fs$ with $\bp_0,\bp_1 \in \sP\setminus\sW^\fs$ which intersects $\sW_{1,\emptyset}^\fs$ transversely.
  Denote
  \begin{itemize}
  \item
    by $\set{t_1, \ldots, t_N } \subset (0,1)$ the finite set of times at which the spectrum of $\slD_{\bp_t}^\fs$ crosses zero, i.e., $\bp_t \in \sW_{1,\emptyset}^\fs$
  \end{itemize}
  and, for each $i = 1, \ldots, N$, denote
  \begin{itemize}
  \item
    by $\chi_i \in \set{\pm 1}$ the sign of the spectral crossing at $t_i$ and
  \item
    by $\Psi_i$ a nowhere vanishing spinor spanning $\ker \slD_{\bp_t}$.
  \end{itemize}
  The intersection number of $(\bp_t)$ with $\sW_{1,\emptyset}^\fs$ is
  \begin{equation*}
    \sum_{i=1}^{N} \chi_i \cdot \sigma(\Psi_i,\bp_{t_i}).
  \end{equation*}
\end{prop}

\begin{proof}
  Let $\bp_0 \in \sW_{1,\emptyset}^\fs$.
  Let $\Psi_0 \in \ker \slD^\fs_{\bp_0}$ be such that $\Abs{\Psi_0}_{L^2} = 1$.
  It follows from the Implicit Function Theorem that,
  for some open neighborhood $U$ of $\bp_0 \in \sP$,
  there is a unique smooth map $U \to \R \times \Gamma(\Re(S\otimes E))$,
  \begin{equation*}
    \bp \mapsto (\lambda(\bp),\Psi_\bp)
  \end{equation*}
  such that
  \begin{equation*}
    \lambda(\bp_0) = 0 \qandq
    \Psi_{\bp_0} = \Psi_0
  \end{equation*}
  as well as
  \begin{equation}
    \label{Eq_DefinitionOfLambda}
    \slD^\fs_\bp \Psi_\bp = \lambda(\bp)\Psi_\bp \qandq
    \Abs{\Psi_0}_{L^2} = 1.
  \end{equation}
  It follows from the Implicit Function Theorem and the openness of the non-vanishing condition that for $U$ sufficiently small, we have
  \begin{equation*}
    U \cap \sW_{1,\emptyset}^\fs = \lambda^{-1}(0).
  \end{equation*}
  Since
  \begin{equation}
    \label{Eq_DerivativeOfLambda}
    (\rd_{\bp_0}\lambda)(0,b) = \inner{\bar\gamma(b)\Psi}{\Psi}_{L^2}
  \end{equation}
  and Clifford multiplication induces an isomorphism from $T^*M\otimes\su(E)$ to trace-free symmetric endomorphisms of $\Re(S\otimes E)$,
  $\lambda$ is a submersion provided $U$ is sufficiently small.
  Hence, $\sW_{1,\emptyset}^\fs$ is a codimension one submanifold.
  To see that $\sW_{1,\emptyset}^\fs$ is closed, observe that $(\bp_i)$ is a sequence on $\sW_{1,\emptyset}^\fs$ with $\bp_i \to \bp \in \sP$, then $\bp \in \sW^\fs$ and thus either in $\sW_{1,\emptyset}^\fs$ or $\sW^\fs_b$.
  
  The above argument goes through with
  \begin{equation*}
    \sW_1^\fs = \set{ \bp \in \sP : \dim\ker\slD^\fs_\bp = 1 }
  \end{equation*}
  instead of $\sW_{1,\emptyset}^\fs$.
  Define a coorientation of $\sW_1^\fs$ by demanding that the isomorphism
  \begin{equation*}
    \rd_\bp\lambda \co T_\bp\sP/T_\bp\sW_1^\fs \to \R
  \end{equation*}
  is orientation-preserving.  
  This coorientation has the following property.
  If $(\bp_t)_{t\in[0,1]}$ is a path in $\sP$ such that $\dim\ker\slD^\fs_{\bp_t} \leq 1$, $\dim\ker\slD^\fs_{\bp_t} = 0$ for $t=0,1$, 
  then the intersection number of $(\bp_t)_{t\in[0,1]}$ with $\sW_1^\fs$ is precisely the spectral flow of $\slD^\fs_{\bp_t}$.
  Therefore, we call this coorientation the \defined{spectral coorientation}.
    $\sW_{1,\emptyset}^\fs$ is an open subset of $\sW_1^\fs$.
  Thus it inherits the spectral coorientation;
  however, this coorientation does not have the desired property.

  If $\bp \in \sW_{1,\emptyset}^\fs$ and $\Psi_\bp$ spans $\ker \slD_\bp^\fs$,
  then $\Psi_\bp$ is nowhere vanishing and \autoref{Def_Sigma} defines $\sigma(\Psi_\bp,\bp) \in \set{\pm 1}$.
  By \autoref{Prop_HowSigmaChanges}, the $\sigma(\Psi_\bp,\bp)$ depends only on $\bp \in \sW_{1,\emptyset}^\fs$;
  moreover, $\bp \mapsto \sigma(\Psi_\bp,\bp)$ is locally constant on $\sW_{1,\emptyset}^\fs$.
  The \defined{twisted spectral coorientation} on $\sW_{1,\emptyset}^\fs$ is defined by demanding that the isomorphism
  \begin{equation*}
    \sigma(\Psi_\bp,\bp)\cdot \rd_\bp\lambda \co T_\bp\sP/T_\bp\sW_{1,\emptyset}^\fs \to \R
  \end{equation*}
  is orientation-preserving.
  By definition, for any path $(\bp_t)$ as in the statement of the proposition, the intersection number of $(\bp_t)$ with $\sW_{1,\emptyset}^\fs$ with respect to the twisted spectral coorientation is
  \begin{equation*}
    \sum_{i=1}^{N} \chi_i \cdot \sigma(\Psi_i,\bp_{t_i}).
    \qedhere
  \end{equation*}
\end{proof}

\begin{theorem}
  \label{Thm_OmegaDetectsHarmonicZ2Spinors}
  In the above situation, the following hold.
  \begin{enumerate}
  \item
    \label{Thm_OmegaDetectsHarmonicZ2Spinors_OmegaNonTrivial}
    The cohomology class $\omega \in H^1(\sP\setminus\sW_b,\Z) = \Hom(\pi_1(\sP\setminus\sW_b),\Z)$ defined by $\sW_{1,\emptyset}$ together with the coorientation from \autoref{Prop_NaturalCoorientation} is non-trivial.
  \item
    \label{Thm_OmegaDetectsHarmonicZ2Spinors_Existence}
    If $(\bp_0,\eta_0) \in \sQ^\reg$ and $(\bp_t,\eta_t)$ is a loop in $\bsQ^\reg\paren[\big]{(\bp_0,\eta_0),(\bp_0,\eta_0)}$, then $(\bp_t)$ is a path in $\sP\setminus\sW_b$ and if $\omega([\bp_t]) \neq 0$, then there is exists a harmonic $\Z_2$ spinor with respect to some $\bp_t$.
  \end{enumerate}
\end{theorem}

\begin{proof}[Proof of \autoref{Thm_ExistenceOfSingularHarmonicZ2Spinors} assuming  \autoref{Thm_OmegaDetectsHarmonicZ2Spinors}]
  The union of the projections of the subsets $\bsQ^\reg\paren[\big]{(\bp_0,\eta_0),(\bp_0,\eta_0)}$ to $\bsP(\bp_0,\bp_0)$, as $(\bp_0,\eta_0)$ ranges over $\sQ^\reg$, is a residual subset of the space of all loops in $\sP$.
  This shows that the loops in $\sP \setminus \sW_b$ which have a lift to  $\bsQ^\reg\paren[\big]{(\bp_0,\eta_0),(\bp_0,\eta_0)}$ are generic among all loops in $\sP \setminus \sW_b$.
  For such loops \autoref{Thm_OmegaDetectsHarmonicZ2Spinors} applies and thus \autoref{Thm_ExistenceOfSingularHarmonicZ2Spinors} follows.
\end{proof}

The idea of the proof that $\omega \neq 0$ is to exhibit a loop $(\bp_t)$ in $\sP\setminus\sW_b$ on which $\omega$ evaluates non-trivially.
More precisely, we will construct such a loop which intersects $\sW^\fs_{1,\emptyset}$ in two points as illustrated in \autoref{Fig_W11InW1Empty}, which cannot be joined by a path in $\sW^\fs_{1,\emptyset}$;
however, they are joined by a path in $\sW^\fs_1$ passing through $\sW^\fs_{1,\star}$ in a unique point.

While the coorientation on $\sW^\fs_1$ is preserved along this path, the one on $\sW^\fs_{1,\emptyset}$ is not.
Consequently, the intersection number of the loop with $\sW^\fs_{1,\emptyset}$ is $\pm 2$.
The above situation can be arranged so that $(\bp_t)$ does not intersect $\sW^{\tilde \fs}_{1,\emptyset}$ for any other spin structure $\tilde\fs$.
It follows that
\begin{equation*}
  \omega([\bp_t]) \pm 2 \neq 0.
\end{equation*}

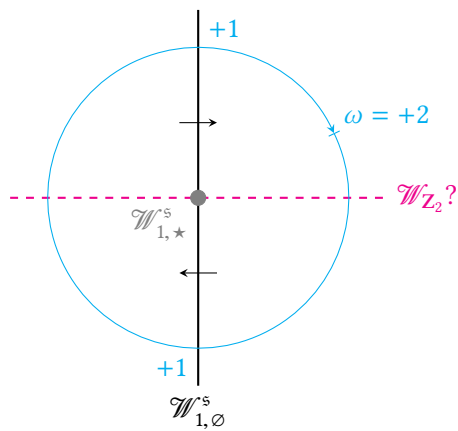
\begin{figure}[h]
  \centering
  \begin{tikzpicture}
    \draw[thick] (0,-2.5) node[below]{$\sW^\fs_{1,\emptyset}$} -- (0,2.5);
    \draw[-stealth] (-.25,1) -- (.25,1);
    \draw[stealth-] (-.25,-1) -- (.25,-1) ;
    \draw[thick,dashed,magenta] (-2.5,0) -- (2.5,0) node[right] {$\sW_{\Z_2}$?};
    \filldraw[gray] (0,0) node[below left]{$\sW^\fs_{1,\star}$} circle (0.1);
    \draw[stealth-|,cyan] ([shift=(25.5:2)]0,0) node[above right] {$\omega = +2$} arc (25.5:385.5:2);
    \draw (0,2) node[above right, cyan]{$+1$};
    \draw (0,-2) node[below left, cyan]{$+1$};
  \end{tikzpicture}
  \caption{A loop linking $\sW^\fs_{1,\star}$ and pairing non-trivially with $\omega$.}
  \label{Fig_W11InW1Empty}
\end{figure}

If $(\bp_0,\eta_0) \in \sQ^\reg$ and $(\bp_t,\eta_t)$ is a loop in $\bsQ^\reg\paren[\big]{(\bp_0,\eta_0),(\bp_0,\eta_0)}$
and $\omega([\bp_t]) \neq 0$, then there is exists a singular harmonic $\Z_2$ harmonic spinor with respect to some $\bp_t$ for otherwise
\begin{equation*}
  \omega([\bp_t]) = 0
\end{equation*}
by \autoref{Thm_TwoSeibergWittenWallCrossing}.

\begin{remark}
  This and the work of \citet{Takahashi2015,Takahashi2017} indicate the presence of a wall $\sW_{\Z_2} \subset \sP$ caused by singular harmonic $\Z_2$ spinors as depicted in \autoref{Fig_W11InW1Empty}.
  In light of the above discussion it is a tantalizing question to ask:
  \begin{center}
    Can the harmonic $\Z_2$ spinors, whose abstract existence is guaranteed by \autoref{Thm_OmegaDetectsHarmonicZ2Spinors}, be constructed more directly by a gluing construction?
  \end{center}
  We plan to investigate this problem in future work.
\end{remark}

\begin{proof}[Proof of \autoref{Thm_OmegaDetectsHarmonicZ2Spinors}]
  To prove \eqref{Thm_OmegaDetectsHarmonicZ2Spinors_Existence}, note that if $(\bp_0,\eta_0) \in \sQ^\reg$ and $(\bp_t,\eta_t)$ is a loop in $\bsQ^\reg\paren[\big]{(\bp_0,\eta_0),(\bp_0,\eta_0)}$
  and $\omega([\bp_t]) \neq 0$, then there is exists a singular harmonic $\Z_2$ spinor with respect to some $\bp_t$ for otherwise
  \begin{equation*}
    n(\bp_0,\eta_0) = n(\bp_0,\eta_0) + \omega([\bp_t])
  \end{equation*}
  by \autoref{Thm_TwoSeibergWittenWallCrossing}.
  Here
  \begin{equation*}
    n(\bp,\eta) = \sum_\fw n_\fw(\bp,\eta)
  \end{equation*}
  and we sum over all spin$^c$ structures $\fw$ with trivial determinant.
  
  In order to prove \itref{Thm_OmegaDetectsHarmonicZ2Spinors_OmegaNonTrivial} we will produce a loop pairing non-trivially with $\omega$.
  The existence of such a loop is ensured by the following result provided we can exhibit a point $\bp_\star \in \sW^\fs_{1,\star}$.

  \begin{prop}
    \label{Prop_SmallLoop}
    Given $\bp_\star \in \sW^\fs_{1,\star}$ and an open neighborhood $U$ of $\bp_\star \in \sP$,
    there exists a loop $(\bp_t)_{t\in S^1}$ in $U\cap (\sP \setminus \sW^\fs_b)$ such that:
    \begin{enumerate}
    \item
      $\bp_{1/4}, \bp_{3/4} \in \sW_{1,\emptyset}^\fs$, and $\bp_t \in \sP\setminus \sW_1^\fs$ for all $t \notin \set{1/4,3/4}$,
    \item
      if $\Psi_{1/4}$ and $\Psi_{3/4}$ denote spinors spanning $\slD^\fs_{\bp_{1/4}}$ and $\slD^\fs_{\bp_{3/4}}$, then
      \begin{equation*}
        \deg(\Psi_{1/4},\Psi_{3/4}) = \pm 1;
      \end{equation*}
      and
    \item
      the spectral crossings at $\bp_{1/4}$ and $\bp_{3/4}$ occur with opposite signs.
    \end{enumerate}
    In particular, the intersection number of $(\bp_t)_{t\in[0,1]}$ with $\sW_{1,\emptyset}^\fs$ with respect to the coorientation from \autoref{Prop_NaturalCoorientation} is $\pm 2$.
  \end{prop}

  \begin{proof}
    Let $\Psi_\star \in \ker \slD^\fs_{\bp_\star}$ and $x_\star \in M$ be such that $\Abs{\Psi_\star}_{L^2} = 1$ and $\Psi_\star(x_\star) = 0$.
    Let $\phi \in \Gamma(\Re(S_\fs\otimes E))$ be such that
    \begin{gather*}
      \im((\nabla\Psi_\star)_{x_\star}) + \R\Span{\phi_\star(x_\star)}
      = \Re(S_\fs\otimes E)_{x_\star}, \\
      \abs{\phi_\star(x_\star)} = 1, \qandq
      \nabla\phi(x_\star) = 0.
    \end{gather*}

    We can assume that $U$ is sufficiently small for the Implicit Function Theorem to guarantee that there is a unique smooth map $U \to \R \times \Gamma(\Re(S_\fs\otimes E)) \times M \times \R$,
    \begin{equation*}
      \bp \mapsto (\lambda(\bp),\Psi_\bp, x_\bp, \nu(\bp))
    \end{equation*}
    such that
    \begin{equation*}
      \lambda(\bp_\star) = 0, \quad
      \Psi_{\bp_\star} = \Psi_\star, \qandq
      x_{\bp_\star} = x_\star
    \end{equation*}
    as well as
    \begin{equation}
      \label{Eq_LambdaPsiXNu}
      \slD^\fs_\bp \Psi_\bp = \lambda(\bp)\Psi_\bp \quad
      \Psi_\bp(x_\bp) = \nu(\bp)\phi_\star(x_\bp), \qandq
      \Abs{\Psi_0}_{L^2} = 1.
    \end{equation}

    As before
    \begin{equation*}
      U \cap \sW_1^\fs = \lambda^{-1}(0).
    \end{equation*}
    Set
    \begin{equation*}
      \sN \coloneq \nu^{-1}(0).
    \end{equation*}
    This is the set of those $\bp \in U$ for which the eigenspinor with smallest eigenvalue has a unique zero which is also non-degenerate.

    From the proof of \autoref{Prop_NaturalCoorientation} we know that $U\cap\sW_1^\fs$ is a codimension one submanifold.
    We will now show that $\sN$ is a codimension one submanifold as well and that it intersects $U\cap\sW_1^\fs$ transversely in $U \cap \sW^\fs_{1,\star}$,
    see \autoref{Fig_SmallLoop}.
    Knowing this, the existence of a loop $(\bp)_{t\in S^1}$ with the desired properties follows easily because crossing $\sN$ changes the relative degree by $\pm 1$.
    Indeed, let $t_0$ be a time at which the path crosses $\sN$ and
    let $\epsilon$ be a small positive number.
    Consider the path of eigenspinors $(\Psi_{\bp_t})$ for $t \in [t_0 - \epsilon, t_0 + \epsilon]$ as introduced in \eqref{Eq_LambdaPsiXNu}.
    For $t \neq t_0$, each of the spinors is nowhere vanishing and $\Psi_{\bp_{t_0}}$ has a single non-degenerate zero.
    Thus, $\deg(\Psi_{\bp_{t_0-\epsilon}}, \Psi_{\bp_{t_0+\epsilon}}) = \pm 1$.
    \begin{figure}[h]
      \centering
      \begin{tikzpicture}
        \draw[thick] (0,-2.5) node[below]{$\sW_1^\fs$} -- (0,2.5);
        \draw[thick] (-2.5,0) -- (2.5,0) node[right] {$\sN$};
        \filldraw (0,0) node[below left]{$\sW^\fs_{1,\star}$} circle (0.1);
        \draw[stealth-,gray,dashed] ([shift=(25.5:2)]0,0) node[above right] {$\bp_t$} arc (25.5:385.5:2);
      \end{tikzpicture}
      \caption{$\sW_1^\fs$ and $\sN$ intersecting in $\sW^\fs_{1,\star}$.}
      \label{Fig_SmallLoop}
    \end{figure}
    
    We will show that $\rd_{\bp_\star}\nu|_{T_{\bp_\star}\sW_1^\fs}$ is non-vanishing.
    This implies both that $\sN$ is a codimension one manifold and that it intersects $U\cap\sW^{\fs_1}$ transversely.
    For $\hat \bp = (0,b) \in T_{\bp_\star}\sP$ to be determined,
    set
    \begin{equation*}
      \lambda_t \coloneq \lambda(\bp_\star + t\hat \bp), \quad
      \Psi_t \coloneq \Psi_{\bp_\star + t\hat \bp}, \quad
      x_t \coloneq x_{\bp\star+t\hat\bp}, \qandq
      \nu_t \coloneq \nu_{\bp_\star + t\hat \bp}
    \end{equation*}
    as well as
    \begin{equation*}
      \dot \lambda \coloneq \left.\frac{\rd}{\rd t}\right|_{t=0} \lambda_t, \quad
      \dot \Psi \coloneq \left.\frac{\rd}{\rd t}\right|_{t=0} \Psi_t, \quad
      \dot x \coloneq \left.\frac{\rd}{\rd t}\right|_{t=0} x_t, \qandq
      \dot \nu \coloneq \left.\frac{\rd}{\rd t}\right|_{t=0} \nu_t.
    \end{equation*}
    Differentiating \eqref{Eq_LambdaPsiXNu} we obtain
    \begin{equation*}
      \bar\gamma(b)\Psi_\star + \slD^\fs_{\bp_\star} \dot\Psi
      =
      \dot\lambda \Psi_\star, \quad
      \dot \Psi(x_\star) + (\nabla\Psi_\star)_{x_\star}\dot x = \dot \nu \phi(x_\star), \qandq
      \inner{\Psi_\star}{\dot\Psi}_{L^2} = 0.
    \end{equation*}
    From this it follows that
    \begin{equation}
      \label{Eq_DotLambdaPsiXNu}
      \begin{split}
        \dot\lambda
        &=
        \inner{\bar\gamma(b)\Psi_\star}{\Psi_\star}_{L^2}, \\
        \slD^\fs_{\bp_\star} \dot \Psi
        &=
        \bar\gamma(b)\Psi_\star - \inner{\bar\gamma(b)\Psi_\star}{\Psi_\star}_{L^2}\Psi_\star, \qand \\
        \dot \nu
        &=
        \inner{\dot\Psi(x_\star)}{\phi(x_\star)}.
      \end{split}
    \end{equation}

    Suppose we can arrange a choice of $b$ such that
    \begin{enumerate}
    \item
      $\inner{\bar\gamma(b)\Psi_\star}{\Psi_\star}_{L^2} = 0$,
    \item
      $\slD^\fs_{\bp_\star}\dot\Psi$ vanishes in a neighborhood of $x_\star$.
    \item
      \label{It_DotNuNonZero}
      $\inner{\dot\Psi(x_\star)}{\phi(x_\star)} \neq 0$, and
    \item
      \label{It_DotPsiPerpPsiStar}
      $\inner{\Psi_\star}{\dot\Psi}_{L^2} = 0$.
    \end{enumerate}
    In this situation it would follow that
    \begin{equation*}
      \dot \lambda = 0 \quad\text{but}\quad \dot\nu \neq 0;
    \end{equation*}
    that is
    \begin{equation*}
      \hat\bp = (0,b) \in T_{\bp_\star}\sW_1^\fs \qandq \rd_{\bp_\star}\nu(b) \neq 0.
    \end{equation*}

    It remains to find such a $b$.
    To begin with, observe that we can certainly find $\dot\Psi$ with the above properties by solving the Dirac equation in a neighborhood of $x_\star$ subject to the constraint \eqref{It_DotNuNonZero} and then extending to all of $M$ so that \eqref{It_DotPsiPerpPsiStar} holds.
    Fix such a choice of $\dot\Psi$.    
    Clifford multiplication by $T^*M\otimes\su(E)$ on $\Re(S_\fs\otimes E)$ induces a isomorphism between $T^*M\otimes\su(E)$ and trace-free symmetric endomorphisms of $\Re(S_\fs\otimes E)$.
    Since $\slD^\fs_{\bp_\star}\dot\Psi$ vanishes in a neighborhood of $x_\star$ and $\Psi_\star$ vanishes only at $x_\star$,
    one can find $b \in \Omega^1(M,\su(E))$ such that
    \begin{equation*}
      \inner{\bar\gamma(b)\Psi_\star}{\Psi_\star}_{L^2} = 0
      \qandq
      \bar\gamma(b)\Psi_\star = \slD^\fs_{\bp_\star} \dot \Psi.
    \end{equation*}
    This completes the proof.
  \end{proof}

  It remains to exhibit a point $\bp_\star \in \sW^\fs_{1,\star}$ for some spin structure $\fs$ but such that $\bp_\star \notin \sW^{\tilde\fs}$ for every other spin structure $\tilde\fs$.
  This requires the following two propositions as preparation.

  \begin{prop}
    \label{Prop_Wsk}
    Let $k \in \set{ 2,3,\ldots }$.
    The subset
    \begin{equation*}
      \sW_k^\fs
      \coloneq
        \set{ \bp \in \sP : \dim\ker\slD^\fs_\bp = k }
      \subset
        \sP
    \end{equation*}
    is contained in a submanifold of codimension three.
    Moreover, $\sW_k^\fs \cap \overline{\sW_1^\fs} = \sW_k^\fs$.
  \end{prop}

  \begin{proof}
    Let $\bp_0 \in \sP$ such that $\dim\ker\slD^\fs_{\bp_0} = k$.
    Choose an $L^2$--orthonormal basis $\set{\Psi_i}$ of $\ker\slD^\fs_{\bp_0}$.
    For a sufficiently small neighborhood $U$ of $\bp_0$,
    by the Implicit Function Theorem, there exists a unique smooth map $U \to \Gamma(\Re(S_\fs\otimes E))^{\oplus^k} \times S^2\R^k$
    \begin{equation*}
      \bp \mapsto \(\Psi_{1,\bp},\ldots,\Psi_{k,\bp},\Lambda(\bp) = (\lambda_{ij}(\bp))\)
    \end{equation*}
    such that
    \begin{equation*}
      \Psi_{i,\bp_0} = \Psi_i \qandq
      \Lambda(\bp_0) = 0
    \end{equation*}
    as well as
    \begin{equation*}
      \slD_\bp^\fs \Psi_{i,\bp} = \sum_{j=1}^k \lambda_{ij}\Psi_{j,\bp} \qandq
      \inner{\Psi_{i,\bp}}{\Psi_{j,\bp}}_{L^2} = \delta_{ij}.
    \end{equation*}
    (It follows from the fact that $\slD_\bp^\fs$ is symmetric, that $\lambda_{ij} = \lambda_{ji}$.)
    We have
    \begin{equation*}
      U \cap \sW_k^\fs = \Lambda^{-1}(0)
    \end{equation*}
    We will show that $\rd\Lambda \co T_{\bp_0}U \to S^2\R^k$ has rank at least three.
    This will imply that $\sW_k^\fs$ has codimension at least three.
    
    Suppose $\Psi_2 = f\Psi_1$ for some function $f \in C^\infty(M)$.
    It follows that
    \begin{equation*}
      0 = \slD_{\bp_0}^\fs \Psi_2 = \gamma(\nabla f)\Psi_1
    \end{equation*}
    This in turn implies that $f$ is constant because $\Psi_1$ is non-vanishing on an dense open subset of $M$.
    However, this is non-sense because $\inner{\Psi_i}{\Psi_j}_{L^2} = \delta_{ij}$.
    It follows that there is an $x \in M$ such that $\Psi_1(x)$ and $\Psi_2(x)$ are linearly independent.
    Clifford multiplication induces an isomorphism from $T^*M\otimes\su(E)$ to trace-free symmetric endomorphisms of $\Re(S_\fs\otimes E)$.
    Therefore, given any $(\mu_{ij}) \in S^2\R^2$,
    we can find $\hat \bp = (0,b) \in T_{\bp_0}U$ such that
    \begin{equation*}
      \inner{\bar\gamma(b)\Psi_i}{\Psi_j}_{L^2} = \mu_{ij} \qforq i,j \in \set{1,2}.
    \end{equation*}
    Since
    \begin{equation*}
      \rd_{\bp_0}\Lambda(\hat\bp) = \(\inner{\bar\gamma(b)\Psi_i}{\Psi_j}_{L^2}\) \in S^2\R^k,
    \end{equation*}
    it follows that $\rd_{\bp_0}\Lambda$ has rank at least three.

    It follows from the above that, for any $\bp_0 \in \sW_k^\fs$, there exists an arbitarily close $\bp \in \sP$ with $0 < \dim\ker \slD^\fs_\bp < k$.
    From this it follows by induction that $\sW_k^\fs \cap \overline{\sW_1^\fs} = \sW_k^\fs$.
  \end{proof}  

  \begin{prop}
    \label{Prop_Ws1s2}
    If $\fs_1,\fs_2$ are two distinct spin structures, then
    $\sW^{\fs_1}_1$ and $\sW^{\fs_2}_1$ intersect transversely.
  \end{prop}

  \begin{proof}
    Let $\bp \in \sW^{\fs_1}_1 \cap \sW^{\fs_2}_1$.
    Denote by $\Psi_1$ and $\Psi_2$ spinors spanning $\ker\slD^{\fs_1}_\bp$ and $\ker\slD^{\fs_2}_\bp$ respectively.
    The spin structures $\fs_1$ and $\fs_2$ differ by twisting by a $\Z_2$--bundle $\fl$.
    This bundle corresponds to a double cover $\pi\co \tilde M \to M$ and upon pulling back to the cover the spin structures $\fs_1$ and $\fs_2$ both correspond to the same spin structure $\tilde\fs$.
    Let $\tilde{\Psi_i} = \pi^* \Psi_i$ for $i=1,2$ be the lifts of $\Psi_1$, $\Psi_2$ to $\tilde M$.
    The natural involution $\sigma$ on $S_{\tilde{\fs}} \to \tilde M$ acts as $-1$ on $\tilde\Psi_1$ and as $+1$ on $\tilde\Psi_2$.
    In particular, we have
    \begin{equation*}
      \inner{\tilde\Psi_1}{\tilde\Psi_2}_{L^2} = \inner{\sigma( \tilde\Psi_1)}{\sigma( \tilde\Psi_2)}_{L^2} = \inner{-\tilde\Psi_1}{\tilde\Psi_2}_{L^2} = -\inner{\tilde\Psi_1}{\tilde\Psi_2}_{L^2};
    \end{equation*}
    hence, $\inner{\tilde\Psi_1}{\tilde\Psi_2}_{L^2}= 0$.
    After renormalization, we can assume that $\inner{\tilde\Psi_i}{\tilde\Psi_j}_{L^2} = \delta_{ij}$.
    It follows from the argument used in the proof of \autoref{Prop_Wsk} that there is an $x \in M$ such that, for $\tilde x \in \pi^{-1}(x)$, $\tilde\Psi_1(\tilde x)$ and $\tilde\Psi_2(\tilde x)$ are linearly independent.

    Let $\lambda^{\fs_1}$ and $\lambda^{\fs_2}$ be the local defining functions for the walls $\sW^{\fs_1}_1$ and $\sW^{\fs_2}_1$ respectively,
    defined via \eqref{Eq_DefinitionOfLambda} in the proof of \autoref{Prop_NaturalCoorientation}.
    The derivative of $\lambda^{\fs_i}$ in the direction of $b \in \Omega^1(M,\su(E))$ is given by \eqref{Eq_DerivativeOfLambda}:
    \begin{equation*}
      \rd_\bp\lambda^{\fs_i}(0,b)
      = \inner{\bar\gamma(b)\Psi_i}{\Psi_i}_{L^2}
      = \frac12\inner{\bar\gamma(\pi^*b)\tilde\Psi_i}{\tilde\Psi_i}_{L^2}.
    \end{equation*}
    Since $\tilde\Psi_1(\tilde x)$ and $\tilde\Psi_2(\tilde x)$ are linearly independent,
    there exists $b(x) \in T_x M \otimes \su(E_x)$ such that
    \begin{equation*}
      \bar\gamma(\pi^*b(x))\tilde\Psi_1(\tilde x) = 0 \qandq
      \bar\gamma(\pi^*b(x)))\tilde\Psi_2(\tilde x) = \tilde\Psi_2(\tilde x)
    \end{equation*}
    We extend $b(x)$ to a section $b \in \Omega^1(M,\su(E))$ such that
    \begin{equation*}
      \rd_\bp\lambda^{\fs_1}(0,b) = 0
      \quad\text{but}\quad
      \rd_\bp\lambda^{\fs_2}(0,b) \neq 0.
    \end{equation*}
    This shows that derivatives of the local defining functions of $\sW_1^{\fs_1}$ and $\sW_1^{\fs_2}$ are linearly independent;
    hence, the walls intersect transversely.
  \end{proof}
  
  Finally, we are in a position to construct $\bp_\star \in \sW_{1,\star}^\fs$.
  Fix a spin structure $\fs$ as well as $\bp_0 = (g_0,B_0)$ and $x_\star \in M$ such that $g_0$ and $B_0$ are flat on a small ball around a point $x_\star\in M$. 
  Choose local coordinates $(y_1,y_2,y_3)$ around $x_\star$ and a local trivialization of $\Re(S_\fs\otimes E)$ in which $g_0$ is given by the identity matrix and $B_0$ is the trivial connection.
  Let $\Psi \in \Gamma(\Re\otimes S_\fs)$ be any section which is nowhere vanishing away from $x_\star$ and around $x_\star$ agrees with the map $\R^3 \to \H$ given by
  \begin{equation*}
    (y_1,y_2,y_3) \mapsto 2iy_1 - jy_2 - ky_3.
  \end{equation*}
  In particular, $\Psi$ has a single non-degenerate zero at $x_\star$ and satisfies $\slD_\bp^\fs\Psi = 0$ in a neighborhood of $x_\star$.
  Using the same argument as in the proof of \autoref{Prop_SmallLoop}, we find $b\in\Omega^1(\su(E))$ vanishing in a neighborhood of $x_\star$ and such that for $\bp_\star = (g_0,B_0+b)$ we have
  \begin{equation*}
    0=\slD_{\bp_\star}^\fs\Psi = \slD_{\bp_0}^\fs\Psi + \bar\gamma(b)\Psi.
  \end{equation*}
  This shows that $\Psi$ is harmonic with respect to $\bp_\star$.
  If $\dim\ker\slD_{\bp_\star} > 1$, then \autoref{Prop_Wsk} and the argument from \autoref{Prop_SmallLoop} can be used to slightly perturb $\bp_\star$ to arrange that $\dim\ker\slD_{\bp_\star} = 1$ and any spinor spanning $\slD_{\bp_\star}$ has a non-degenerate zero (close to $x_\star$).
  Similarly, \autoref{Prop_Wsk} and \autoref{Prop_Ws1s2} can be used to ensure that there are no non-trivial harmonic spinors with respect to $\bp_\star$ for any other spin structure $\tilde\fs$.
\end{proof}


\appendix
\section{Computation of the hyperkähler quotient}
\label{Sec_H2///U(1)}

Set
\begin{equation*}
  S
  \coloneq 
    \Hom_\C(\C^2,\H)
\end{equation*}
with $\H$ considered as a complex vector space whose complex structure is given by right-multiplication with $i$.
$S$ is a quaternionic Hermitian vector space: its $\H$--module structure arises by left-multiplication.
The action of $\U(1)$ on $S$ given by $\rho(e^{i\theta})\Psi = e^{i\theta}\Psi$ is a quaternionic representation with associated moment map
\begin{equation*}
  \mu(\Psi) = \Psi\Psi^* - \frac12\abs{\Psi}^2\,\id_\H.
\end{equation*}

The standard complex volume form $\Omega = e^1\wedge e^2 \in \Lambda^2 (\C^2)^*$ and the standard Hermitian metric on $\C^2$, define a complex anti-linear map $J \co \C^2 \to \C^2$ by
\begin{equation*}
  -\inner{v}{Jw} = \Omega(v,w).
\end{equation*}
This makes $\C^2$ into a $\H$--module.

\begin{prop}
  \label{Prop_H2///U(1)=Re/Z2}
  We have
  \begin{equation*}
    S^\reg\hkred\U(1)
    =
      \parentheses*{\Re(\H\otimes_\C\C^2)\setminus\set{0}}/\Z_2.
  \end{equation*}
  Here the real structure on $\H\otimes_\C\C^2$ is given by $\overline{q\otimes v} \coloneq qj\otimes Jv$.
\end{prop}

\begin{proof}
  The complex volume form $\Omega$ defines a complex linear isomorphism $(\C^2)^* \iso \C^2$;
  hence, we can identify $S \iso \H\otimes_\C \C^2$.
  We will further identify $\C^2$ with $\H$ via $(z,w) \mapsto z + wj$.
  With respect to this identification the complex structure is given by left-multiplication with $i$ and $J$ becomes left-multiplication by $j$.
  If we denote by $H_+$ ($H_-$) the quaternions equipped with their right (left) $\H$--module structure, then we can identify $S$ with
  \begin{equation*}
    \H_+ \otimes_\C \H_-.
  \end{equation*}
  In this identification the action of $\U(1)$ is given by
  \begin{equation*}
    \rho(e^{i\theta})q_+\otimes q_-
    =
      q_+e^{i\theta}\otimes q_-
    =
      q_+\otimes e^{i\theta}q_-,
  \end{equation*}
  and the moment map becomes
  \begin{equation*}
    \mu(q_1\otimes 1 + q_2\otimes j)
    =
      -i\otimes\frac12(q_1i\bar q_1 + q_2i\bar q_2) \in i\R \otimes \Im\H.
  \end{equation*}

  If $\mu(q_1\otimes 1 + q_2\otimes j) = 0$, then
  \begin{equation}
    \label{Eq_Q1IQ2I}
    q_1i\bar q_1 = -q_2i\bar q_2.
  \end{equation}
  This implies that $\abs{q_1} = \abs{q_2}$.
  Unless $q_1$ and $q_2$ both vanish, there is a unique $p \in \H$ satisfying
  \begin{equation*}
    \abs{p} = 1 \qandq q_1 = q_2p.
  \end{equation*}
  From \eqref{Eq_Q1IQ2I}, it follows that
  \begin{equation*}
    pi = -ip;
  \end{equation*}
  hence, $p = je^{i\phi}$ for some $\phi \in \R$.
  It follows that, for any $\theta \in \R$,
  \begin{equation*}
    q_1e^{i\theta} = q_2e^{i\theta} \cdot je^{i(\phi+2\theta)}.
  \end{equation*}
  
  Since
  \begin{equation*}
    \overline{q_1\otimes 1 + q_2\otimes j}
    = -q_2j \otimes 1 + q_1j\otimes j,
  \end{equation*}
  the real part of $\H_+\otimes_\C\H_-$ consists of those $q_1\otimes 1 + q_2\otimes j$ with
  \begin{equation*}
    q_1 = -q_2j.
  \end{equation*}
  Consequently, the $\U(1)$--orbit of each non-zero $\bq = q_1\otimes 1 + q_2\otimes j$ intersects $\Re(\H_+\otimes_\C\H_-)$ twice:
  in $\pm \rho(e^{i(-\phi/2 + \pi/2)})\bq$.
\end{proof}

\section{Determinant line bundles}
\label{Sec_DeterminantLineBundles}

One of the standard tools to orient moduli spaces are determinant line bundles of families of Fredholm operators;
see, e.g., \citet[Section 5.2.1]{Donaldson1990}.
The basic ideas are quite simple, but to pin down a precise orientation procedure certain conventions have to be chosen and one has to verify that those conventions are consistent.
In fact, different choices will lead to different outcomes \cite{Zinger2016} and even seemingly natural conventions may not be consistent \citet{Salamon2017}.
A very careful discussion of the construction of determinant line bundles can be found in the diploma thesis of \citet{Bohn2007}.
In this appendix we summarize this construction.
For proofs and more a detailed discussion we refer to \cite[Appendices B and C]{Bohn2007}.

\subsection{The Knudsen--Mumford conventions}

\citeauthor{Bohn2007}'s construction uses the sign conventions introduced by \citet{Knudsen1976}.
We briefly recall their definitions.

\begin{definition}
  A \defined{graded line} is a $1$--dimensional vector space $L$ together with an integer $\alpha \in \Z$.
\end{definition}

\begin{definition}
  \label{Def_TensorProduct}
  The \defined{tensor product} of two graded lines $(L,\alpha)$, $(M,\beta)$ is defined to be
  \begin{equation*}
    (L,\alpha)\otimes(M,\beta) \coloneq (L\otimes M,\alpha+\beta)
  \end{equation*}
  The \defined{dual} of a graded line $(L,\alpha)$ is defined to be
  \begin{equation*}
    (L,\alpha)^* = (L^*,-\alpha).
  \end{equation*}
\end{definition}

\begin{definition}
  Let $(L,\alpha)$ and $(M,\beta)$ be two graded lines.
  Define the isomorphism $\sigma_{L,M}\co (L,\alpha)\otimes(M,\beta) \iso (M,\beta)\otimes (L,\alpha)$
  by
  \begin{equation*}
    (\ell\otimes m) \mapsto (-1)^{\alpha\beta} m\otimes \ell.
  \end{equation*}
  Define the isomorphism $\epsilon_L\co (L,\alpha)^*\otimes (L,\alpha) \to (\R,0)$ by
  \begin{equation*}
    \lambda\otimes \ell \mapsto (-1)^{\binom{\alpha}{2}}\lambda(\ell).
  \end{equation*}
\end{definition}

If omitted in the notation,
it shall always be assumed that we use the isomorphism $\sigma$ and $\kappa$.
These isomorphism pin down the \defined{Knudsen--Mumford conventions}.
For example, an unlabeled isomorphism $(L,\alpha)\otimes (L,\alpha)^* \iso (\R,0)$ will given by
\begin{equation*}
  \ell\otimes \lambda \mapsto (-1)^{-\binom{\alpha+1}{2}}\lambda(\ell),
\end{equation*}
the composition of $\epsilon$ with $\sigma$.

\begin{definition}
  Let $V$ be a finite dimensional vector space.
  The \defined{determinant line} of $V$ is the graded line
  \begin{equation*}
    \det V \coloneq (\Lambda^{\dim V} V,\dim V).
  \end{equation*}
\end{definition}

\begin{prop}[{\cite[Lemma B.1.14]{Bohn2007}}]
  Let
  \begin{equation*}
    0 \to V_0 \xrightarrow{f_0} V_1 \xrightarrow{f_1} \cdots \xrightarrow{f_{n-1}} V_n \to 0
  \end{equation*}
  be an exact sequence.
  For $i \in \set{0,\ldots,n}$,
  set $c_i \coloneq \dim(V_i/\ker f_i)$.
  \begin{enumerate}
  \item
    Given
    \begin{equation*}
      \nu = \nu_0 \otimes \nu_2 \otimes \cdots  \in \det V_{0} \otimes \det V_{2} \otimes \cdots,
    \end{equation*}
    for every $i \in \set{0,\ldots,n}$,
    there exist $\mu_i \in \Lambda^{c_i}V_i$ such that
    \begin{equation}
      \label{Eq_NuMui}
      \nu
      = \mu_0 \otimes (f(\mu_1)\wedge \mu_2) \otimes (f(\mu_3)\wedge \mu_4) \otimes \cdots
    \end{equation}
  \item
    If $\nu \in \det V_{0} \otimes \det V_{2} \otimes \cdots \otimes \det V_{2\floor{n/2}}$, then
    for any choice of $\mu_i \in \Lambda^{c_i}V_i$ such that the relation \autoref{Eq_NuMui} holds, the element
    \begin{equation*}
      (f(\mu_0)\wedge \mu_1) \otimes (f(\mu_2)\wedge \mu_3) \otimes \cdots 
    \end{equation*}
    depends only on $\nu$.
  \item
    The map
    \begin{equation*}
      \det V_{0} \otimes \det V_{2} \otimes \cdots
      \to \det V_{1} \otimes \det V_{3} \otimes \cdots
    \end{equation*}
    defined by
    \begin{equation*}
      \nu \mapsto \mu_0 \otimes (f(\mu_1)\wedge \mu_2) \otimes (f(\mu_3)\wedge \mu_4) \otimes \cdots,
    \end{equation*}
    for any choice of $\mu_i \in \Lambda^{c_i}V_i$ satisfying \autoref{Eq_NuMui},
    is an isomorphism.
  \end{enumerate}  
\end{prop}

\begin{cor}
  \label{Cor_PsiDet}
  If
  \begin{equation*}
    0 \to V_0 \xrightarrow{f_0} V_1 \xrightarrow{f_1} V_2 \xrightarrow{f_2} V_3 \to 0
  \end{equation*}
  is an exact sequence,
  then the map
  \begin{equation*}
    \kappa \co \det V_0 \otimes (\det V_3)^* \to \det V_1 \otimes (\det V_2)^*
  \end{equation*}
  defined by
  \begin{equation}
    \label{Eq_PsiDet}
    \omega_0 \otimes [f_2(\omega_2)]^* \mapsto (-1)^{\binom{\dim V_2+\dim V_3+1}{2}}(f_0(\omega_0)\wedge\omega_1) \otimes [f_1(\omega_1)\wedge\omega_2]^*
  \end{equation}
  is an isomorphism.
\end{cor}

The sign in this isomorphism comes from
\begin{align*}
  \det V_0 \otimes (\det V_3)^*
  &\iso
    \det V_0 \otimes \det V_2 \otimes (\det V_2)^* \otimes (\det V_3)^* \\
  &\iso
    \det V_1 \otimes \det V_3 \otimes (\det V_2)^* \otimes (\det V_3)^* \\
  &\iso
    \det V_1 \otimes (\det V_2)^*,
\end{align*}
see \cite[pp. 129-130]{Bohn2007}.

\subsection{Knudsen--Mumford's determinant line bundle}

\begin{definition}
  \label{Def_DeterminantLineOfOperator}
  Let $X$, $Y$ be Banach spaces and $D\co X \to Y$ a Fredholm operator.
  The \defined{determinant line} of $D$ is the graded line
  \begin{equation*}
    \det(D) \coloneq \det(\ker D) \otimes \det(\coker D)^*.
  \end{equation*}
  Here $\det(\ker D)$ and $\det(\coker D)$ are the graded lines associated with the finite-dimensional vector spaces $\ker D$ and $\coker D$ as in \autoref{Def_DeterminantLineOfOperator}, and we use \autoref{Def_TensorProduct} to take the graded tensor product and the dual graded line.
\end{definition}

Henceforth, we will suppress the grading from the notation,
but it will always be used in the natural isomorphisms $\epsilon$, $\sigma$, and $\kappa$ introduced above.

Let $(D_p)_{p \in \sP}$ be a family of Fredholm operator parametrized by the space $\sP$.
In general, $\dim \ker D_p$ and $\dim \coker D_p$ vary in $p$.
Therefore, $\ker D_p$ and $\coker D_p$ do not form vector bundles.

\begin{definition}
  Let $\sU$ be an open subset of $\sP$.
  A \defined{stabilizer} of $(D_p)_{p \in \sP}$ over $\sU$ is  a finite-dimensional vector space $V$ together with a linear map $\iota \co V \to Y$ such that, for all $p \in \sU$, the induced map $V \to \coker D_p$ is surjective.
\end{definition}

Every $p \in \sP$ has a neighborhood $\sU$ which admits a stabilizer.
Given such a choice of stabilizer $\iota$,
for every $p \in \sP$,
the operator $(D_p ~ \iota)\co X\oplus V \to Y$ satisfies $\coker(D_p ~ \iota)=\set{0}$.
Consequently, $\ker(D_p ~ \iota)$ forms a vector bundle over $\sU$.
For every $p \in \sP$,
\begin{equation}
  \label{Eq_StabilizationSequence}
  0 \to \ker D_p \xrightarrow{x\mapsto (x,0)}
  \ker(D_p ~ \iota) \xrightarrow{\pr_V} V \xrightarrow{v \mapsto \iota v ~{\rm mod} \im D_p} \coker D_p \to 0.
\end{equation}
is an exact sequence.
Therefore,
\autoref{Cor_PsiDet} provides us with an isomorphism
\begin{equation*}
  \kappa_D^\iota\co \det(D_p) \iso \det(D_p ~ \iota) \otimes(\det V)^*.
\end{equation*}
This indicates that $\det(D_p ~ \iota) \otimes(\det V)^*$ may play the role of the determinant bundle of $(D_p)_{p \in \sP}$ over $\sU$.
The following asserts that the system of isomorphism induces by \autoref{Cor_PsiDet} for various choices of stabilizers are coherent.

\begin{prop}[{\cite[Prop B.1.18]{Bohn2007}}]
  If $\iota, \iota'$ are two choices of a stabilizer,
  then $(\iota ~ \iota')$ is a stabilizer as well and the diagram
  \begin{equation*}
    \begin{tikzcd}[row sep=large, column sep = large]
      \det(D_p) \ar[r,"\kappa_{D_p}^\iota"] \ar[d,swap,"\kappa_{D_p}^{\iota'}"] \ar[dr,"\kappa_{D_p}^{(\iota~\iota')}"]  & \det (D_p~\iota) \otimes \det(V)^* \ar[d,"\kappa_{(D_p~\iota)}^{\iota'}"] \\
      \det (D_p~\iota') \otimes \det(V')^* \ar[r,swap,"\kappa_{(D_p~ \iota')}^\iota"] & \det (D_p~\iota~\iota') \otimes \det(V\oplus V')^*
    \end{tikzcd}
  \end{equation*}
  commutes.
\end{prop}

The \defined{(Knudsen--Mumford) determinant line bundle}
\begin{equation*}
  \det D \to \sP.
\end{equation*}
is obtained by patching together the various local line bundles $\det(D_p ~\iota)\otimes (\det V)^*$ using the isomorphisms $\kappa_D^{\iota'}\circ (\kappa_D^\iota)^{-1}$.
It is important to note that $\det D$ comes with preferred isomorphisms
\begin{equation*}
  \det D_p \iso (\det D)_p.
\end{equation*}

\subsection{Orientation transport}

Suppose that $\sC$ is a Banach manifold.
Suppose $s$ is a Fredholm section of a Banach vector bundle over $\sC$ whose linearization $L_\fc$ is surjective for any $\fc$ in the zero set $\sM = s^{-1}(0)$.
In this situation $\sM$ is a finite-dimensional manifold and an orientation on $\sM$ is simply a trivialization of the real line bundle $\det (T\sM) = \det (\ker L)$.
Since $L_\fc$ was assumed to be surjective for $\fc \in \sM$, we have
$\det(\ker L_\fc) \iso \det(L_\fc)$.
If one can prove that $\det L \to \sC$ is trivial and find a way to pin down a choice of a trivialization,
then this yields a procedure to orient $\sM$.
In the situation considered in this article, the linearization $L_\fc$ is self-adjoint.
Even though in this situation for every $\fc$ we have a natural isomorphism $\ker L_\fc \iso \coker L_\fc$, and thus an isomorphism $\det(\ker L_\fc) \otimes \det(\coker L_\fc)^* \iso \R$, these isomorphisms do not yield a continuous trivialization of $\det L$ over $\sC$ because the construction of $\det L$ involves choosing local stabilizers.
The \defined{orientation transport} described below (introduced also in \autoref{Def_OrientationTransport} in the specific setting required in this paper) is a way of comparing the isomorphisms $\det(\ker L_\fc) \otimes \det(\coker L_\fc)^* \iso \R$ for different choices of $\fc \in \sC$.

In general, given a path $(D_t)_{t\in[0,1]}$ of  self-adjoint Fredholm operators,
the line bundle $\det D$ is trivial (because $[0,1]$ is a contractible).
Picking any trivialization defines an isomorphism
\begin{equation*}
  (\det D)_0 \iso (\det D)_1.
\end{equation*}
This isomorphism depends on the trivialization only up to multiplication by \emph{positive} real number.
Since $D_t$ is self-adjoint,
there is a canonical isomorphism
\begin{equation*}
  \ker D_t \iso \coker D_t.
\end{equation*}
This induces an isomorphism
\begin{equation*}
  \det D_t \iso \R.
\end{equation*}
The composition of the isomorphisms
\begin{equation*}
  \R \iso \det D_0 \iso (\det D)_0 \iso (\det D)_0 \iso \det D_1 \iso \R
\end{equation*}
is given my multiplication with a non-zero real number.
The sign of this number is
\begin{equation*}
  \OT\((D_t)_{t\in[0,1]}\) = (-1)^{\SF (D_t)_{t\in[0,1]}},
\end{equation*}
see \cite[Theorem C.2.5]{Bohn2007}.
Recall, that we use \autoref{Conv_SpectralFlow} to define the spectral flow for degenerate end points.

In light of this discussion,
\autoref{Prop_OrientationTransport} shows that the determinant line bundle associated to the Seiberg--Witten equation with two spinors is trival.
The convention adopted in \autoref{Def_Sign} amounts to pinning down a trivialization.
With respect to this convention the orientation procedure using the determinant line bundle and the spectral flow agree.

\subsection{A toy example}

The following example illustrates how to work with the determinant bundle.
It also explains why \autoref{Conv_SpectralFlow} is compatible with the Knudsen--Mumford conventions.

Consider the family of linear maps $M_t\co \R \to \R$ defined by $M_t(x) = tx$.
Since the operator $M_t$ is not surjective at $t=0$, in order to define the determinant line bundle $\det M$ over $\R$ we need to choose a stabilizer for the family $(M_t)_{t\in \R}$. 
A natural choice is $V = \R$ and $\iota = \id_\R$.
The stabilized maps $\tilde M_t \co \R\oplus \R \to \R$ are given by
\begin{equation*}
  \tilde M_t =
  \begin{pmatrix}
    t & 1
  \end{pmatrix}.
\end{equation*}
The kernel of $\tilde M_t$ is spanned by $(1,-t)$. 
This gives a trivialization of the vector bundle on $\R$ whose fiber at $t$ is $\ker \tilde M_t$.
The determinant line bundle of $(M_t)$ is thus given by
\begin{equation*}
  (\det M)_t = \ker \tilde M_t \otimes V^* = \R\Span{(1, -t)} \otimes \R^*.
\end{equation*}
This isomorphism trivializes the bundle $\det M$ and for any $t_0, t_1 \in  \R$ gives an isomorphism
\begin{equation*}
  (\det M)_{t_0} \ni (1,-t_0)\otimes 1^* \mapsto (1,-t_1)\otimes 1^* \in (\det M)_{t_1}.
\end{equation*}
If $t \neq 0$, then $M_t$ is invertible and, by definition,
\begin{equation*}
  \det M_t = \det(0) \otimes \det(0)^* = \R\otimes \R^*.
\end{equation*}
The exact sequence \eqref{Eq_StabilizationSequence} in this case is
\begin{equation*}
  0 \to 0 \to
  \ker(M_t ~ \iota) \xrightarrow{(1,-t) \mapsto -t} \R \to 0 \to 0.
\end{equation*}
The induced isomorphism $\kappa\co \det M_t \iso (\det M)_t = \ker \tilde M_t \otimes \R^*$ defined in \autoref{Eq_PsiDet} is
\begin{equation*}
  1\otimes 1^* \mapsto (-1)\cdot[(1,-t)\otimes -t] = (1,-t)\otimes t^*.
\end{equation*}
This allows us to compute the orientation transport from $t=-1$ to $t=1$.  
The sequence of isomorphisms
\begin{equation*}
  \R \iso \R\otimes \R^* = \det M_{-1} \iso \ker \tilde M_{-1} \otimes \R^* \iso \ker \tilde M_1 \otimes \R^* \iso \det M_1 = \R\otimes \R^* \iso \R
\end{equation*}
maps $1 \in \R$ as follows
\begin{equation*}
  1 \mapsto 1\otimes 1^* \mapsto (1,1)\otimes (-1)^* \mapsto (1,-1)\otimes (-1)^* \mapsto -1\otimes 1^* = -1.
\end{equation*}
Thus the orientation transport is $\OT( (M_t)_{t\in [0,1]} ) = -1$.
This agrees with 
\begin{equation*}
  \SF( (M_t)_{t\in[-1,1]}) = 1 \qandq \OT( (M_t)_{t\in [0,1]} ) = (-1)^{\SF( (M_t)_{t\in[-1,1]})}.
\end{equation*}
The spectral flow $\SF( (M_t)_{t\in[-1,1]})$ is independent of \autoref{Conv_SpectralFlow} since the operators $M_{-1}$ and $M_1$ at the endpoints of the path are invertible.

However, for $t = 0$, the operator $M_0$ is not invertible.
Let us investigate the orientation transport from $t=0$ to $t=1$.
By definition,
\begin{equation*}
  \det M_0 = \det(\R) \otimes \det(\R)^* = \R\otimes \R^*.
\end{equation*}
This is different from the situation considered before because $\R$ now has degree $1$ and therefore the isomorphism $\R \to \det M_0$ is given by $1 \mapsto -1\otimes 1^*$ according to the Knudsen--Mumford conventions.
The  isomorphism $\kappa\co \det M_0 \iso (\det M)_0 = \ker \tilde M_0 \otimes \R^*$ defined in \autoref{Eq_PsiDet} is
\begin{equation*}
  1\otimes 1^* \mapsto (-1)\cdot [(1,0) \otimes 1^*].
\end{equation*}
Therefore,
the orientation transport from $t=0$ to $t=1$ is given by $+1$ because the  isomorphisms
\begin{equation*}
  \R \iso \R\otimes \R^* = \det M_0 \iso \ker \tilde M_0 \otimes \R^* \iso \ker \tilde M_1 \otimes \R^* \iso \det M_1 = \R\otimes \R^* \iso \R
\end{equation*}
map $1 \in \R$ as follows
\begin{equation*}
  1 \mapsto - 1\otimes 1^* \mapsto (1,0) \otimes 1^* \mapsto (1,-1) \otimes 1^* = 1\otimes 1^* = 1.
\end{equation*}
Similarly,
the orientation transport from $t=0$ to $t=-1$ is given by $-1$ because the  isomorphisms
\begin{equation*}
  \R \iso \R\otimes \R^* = \det M_0 \iso \ker \tilde M_0 \otimes \R^* \iso \ker \tilde M_{-1} \otimes \R^* \iso \det M_{-1} = \R\otimes \R^* \iso \R
\end{equation*}
map $1 \in \R$ as follows
\begin{equation*}
  1 \mapsto - 1\otimes 1 \mapsto (1,0) \otimes 1^* \mapsto (1,1) \otimes 1^* = -1\otimes 1^* = -1.
\end{equation*}
Both of these agree with the spectral flow defined using \autoref{Conv_SpectralFlow}.

\subsection{Kuranishi models and determinant line bundles}

Finally, we briefly discuss the interaction of Kuranishi models with determinant line bundles.
Let $X$, $Y$, and $\sP$ be Banach spaces,
and let $F_p \co X \to Y$ be a Fredholm map depending smoothly on a  parameter $p \in \sP$.
Denote by
\begin{equation*}
  L_{p,x} \coloneq \rd_xF_p \co X \to Y
\end{equation*}
the linearization of $F_p$ at $x$.
This is a family of Fredholm operators $X \to Y$ depending smoothly on $(p,x)$ and so it defines the determinant line bundle $\det L$ over $\sP \times X$.

Suppose that $F_0(0) = 0$ and construct a Kuranishi model for $F$ near $p=0$ and $x = 0$ as follows.
Choose splittings
\begin{equation*}
  X = X_0\oplus X_1 \qandq Y = Y_0 \oplus Y_1
\end{equation*}
with $X_1$ and $Y_1$ closed, $X_0$ and $Y_0$ finite dimensional, and
such that the operator $L_{0,0} \co X_1 \to Y_1$ induced by $L_{0,0}$ is invertible.
It is a consequence of the Implicit Function Theorem that
there exist an open neighborhood $\sU$ of $0 \in \sP$,
an open neighborhood $V = V_0\times V_1$ of $(0,0) \in X_0\oplus X_1$,
and smooth maps $\Xi \co \sU\times V \to X_1$ and $f \co \sU\times V \to Y_0$ such that
\begin{equation*}
  \tilde F_p(x_0,x_1)
  \coloneq
  F_p(x_0,\Xi(p,x_0;x_1))
  =
  \begin{pmatrix}
    f(p,x_0,x_1) \\
    L_{0,0} x_1
  \end{pmatrix}.
\end{equation*}
Moreover, the maps $(x_0,x_1) \mapsto (x_0,\Xi(p,x_0,x_1))$ are diffeomorphisms onto their images.
Consequently,
we have an exact sequence
\begin{equation}
  \label{Eq_KuranishiModelExactSequence}
  0 \to \ker L_{p,(x_0,x_1)} \to X_0 \xrightarrow{\del_{x_0} f(p,x_0,x_1)} Y_0 \to \coker L_{p,(x_0,x_1)} \to 0.
\end{equation}
It follows that
\begin{equation*}
  \ker L_{p,(x_0,x_1)} \iso \ker \del_{x_0} f(p,x_0,x_1) \qandq
  \coker L_{p,(x_0,x_1)} \iso \coker \del_{x_0} f(p,x_0,x_1).
\end{equation*}

By construction, $\iota\co Y_0 \into Y$ is a stabilizer for $L$ over $\sU\times V$, that is: the map $(L_{p,x} ~ \iota)$ is surjective for all $(p,x) \in \sU \times V$.
For this choice of a stabilizer, the kernel bundle $\ker (L ~ \iota)$ over $\sU \times V$ is isomorphic to the trivial bundle with fiber $X_0$.
Therefore, over $\sU\times V$,
\begin{equation*}
  \det L \iso \det(X_0)\otimes \det(Y_0)^*
\end{equation*}
On the other hand, using $\id \co Y_0 \to Y_0$ as a stabilizer for the family of finite-dimensional operators $\del_{x_0}f(p,x_0,x_1) \co X_0 \to Y_0$ as $(p,x_0,x_1)$ varies in $\sU \times V$, we obtain 
\begin{equation*}
  \det (\del_{x_0}f) \iso \det(X_0)\otimes \det(Y_0)^*.
\end{equation*}
The induced square of isomorphisms
\begin{equation*}
  \begin{tikzcd}
    \det L_{p,(x_0,x_1)} \ar[r,"\kappa"]\ar[d,"\kappa"] & \det \del_{x_0}f(p,x_0,x_1) \ar[d,"\kappa"] \\
    \det(X_0)\otimes \det(Y_0)^* \ar[r,"="] & \det(X_0)\otimes \det(Y_0)^*
  \end{tikzcd}
\end{equation*}
commutes.
This allows one to work carry out computations involving $\det L$ using $\det \del_{x_0}f$.


\section{Relation to gauge theory on $\Gtwo$--manifolds}
\label{Sec_JoyceDonaldsonSegal}

\citet[Section 3]{Donaldson1998} suggested that one might be able to construct $\Gtwo$ analogues of the Casson invariant/instanton Floer homology associated to a natural functional whose critical point are $\Gtwo$--instantons.
One key difficulty with this proposal is that $\Gtwo$--instantons can degenerate by bubbling along associative submanifolds.
\citet[Section 6]{Donaldson2009} explain that this bubbling could be caused by the appearance of (nowhere vanishing) harmonic spinors of $\Re(E\otimes S)$ over associative submanifolds.
In particular, the signed count of $\Gtwo$--instantons can jump along a one-parameter family.
\citeauthor{Donaldson2009} propose to compensate this jump with a counter-term consisting of a weighted count of associative submanifolds.

\citet[Example 8.5]{Joyce2016} poses the following scenario.
Consider a one--parameter family of $\Gtwo$--manifolds $\set{ (Y,\phi_t) : t \in [0,1] }$ together with an $\SU(2)$--bundle $E$ where:
\begin{itemize}
\item
  there is a smooth family of irreducible connections $(A_t)_{t \in [0,1]} \in \sA(E)^{[0,1]}$ such that $A_t$ is an unobstructed $\Gtwo$--instanton with respect to $\phi_t$ for each $t \in [0,1]$, 
\item
  there are no relevant associatives in $(Y,\phi_t)$ for $t \in [0,1/3)\cup(2/3,1]$, and
\item
  there is an obstructed associative $P_{1/3}$ in $(Y,\phi_{1/3})$,
  which splits into two unobstructed associatives $P_t^\pm$ in $(Y,\phi_t)$ for $t \in (1/3,2/3)$, and
  which then annihilate each other in an obstructed associative $P_{2/3}$ in $(Y,\phi_{2/3})$.
\end{itemize}
According to \cite[Theorem 1.2]{Walpuski2013a} a regular crossing of the spectral flow of family of Dirac operators $\slD_t^\pm\co \Gamma(\Re(E|_{P_t^\pm}\otimes \slS_{P_t^\pm})) \to \Gamma(\Re(E|_{P_t^\pm}\otimes \slS_{P_t^\pm}))$ causes a jump in the signed count of $\Gtwo$--instantons;
however, the sign of this jump has not been analyzed.%
\footnote{%
  To be more precise, the jump occurs in the signed count of $\Gtwo$--instantons on a bundle $E'$, which is related to $E$ by $c_2(E') = c_2(E) + \PD[P]$ with $[P] = [P_{1/3}] = [P_t^\pm] = [P_{2/3}]$.
}
\citet[Section 6]{Donaldson2009} and \citet[Section 8.4]{Joyce2016} suggest that this is the only source of jumping phenomena.
The difference in the spectral flows of the Dirac operators $\slD_t^\pm$ is a topological invariant, say $k \in \Z$, which may be non-zero.
\cite[Section 8.4]{Joyce2016} thus concludes that passing from $t < 1/3$ to $t > 2/3$ the signed number of $\Gtwo$--instantons should change by $k \cdot \abs{H_1(P_{1/3},\Z_2)}$;
and, since there are no associatives for $t \in [0,1/3) \cup (2/3,1]$, no counter-term involving a weighted count of associatives could compensate this jump.

It is proposed in \cite{Haydys2014} that the weight associated with each associative $3$-manifold should be the signed count of solutions to the Seiberg--Witten equation with two spinors. 
The loop of associatives can equivalently be seen as a path of parameters $(\bp_t)_{t\in[0,1]}$ on a fixed $3$--manifold $P$, with $\bp_1$ gauge equivalent to $\bp_0$.
Therefore, one can ask how $n(\bp_t)$ varies in this scenario.
Suppose that $b_1(P_{1/3}) > 1$. 
Assuming there are no harmonic $\Z_2$ spinors along the path $(\bp_t)_{t\in[0,1]}$,
a jump in $n(\bp_t)$ would occur precisely when the spectrum of one of the Dirac operators $\slD_{\bp_t}^\fs$ crosses zero.
If the wall-crossing formula for $n(\bp_t)$ were given by the sum of the spectral flows of $(\slD_{\bp_t}^\fs)_{t\in[0,1]}$,
then we would have arrive a contradiction just like in Joyce's argument:
\begin{equation*}
  0 \neq k\cdot \abs{H_1(P,\Z_2)} = n(\bp_1) - n(\bp_0) = 0
\end{equation*}
since $\bp_1$ and $\bp_0$ are gauge equivalent.
However, the conclusion of our work is that:
\begin{enumerate}
\item
  The wall-crossing for $n(\bp_t)$ caused by harmonic spinors is not given by the spectral flow.
\item
  There exist singular harmonic $\Z_2$ spinors which cause additional wall-crossing.
\end{enumerate} 

It is possible that the same happens for the signed count of $\Gtwo$--instantons.
To evaluate the viability of the proposal in \cite{Haydys2014} it is important to answer the following questions.

\begin{question}
  What is the sign of the jump in the number of $\Gtwo$--instantons caused by a harmonic spinor?
\end{question}

\begin{question}
  Do singular harmonic $\Z_2$ spinors cause a jump in the number of (possibly singular) $\Gtwo$--instantons?
\end{question}


\printreferences
\end{document}
